\documentclass[reqno,a4paper]{amsart}
\pdfoutput 1

\usepackage[all,numberbysection]{bi-discrete}

\usepackage[backend=biber,maxbibnames=5,maxalphanames=4,style=alphabetic,bibencoding=utf8,giveninits,url=false,isbn=false,doi=true,eprint=false]{biblatex}
\renewbibmacro{in:}{}
\AtBeginBibliography{\small}

\usepackage{blindtext}

\usepackage[T1]{fontenc}
\allowdisplaybreaks

\usepackage{mathtools}
\usepackage{a4wide}
\usepackage{bm}
\usepackage{dutchcal}
\usepackage{xcolor}

\usepackage{pifont}

\usepackage{graphicx}

\definecolor{cobalt}{rgb}{0.0, 0.28, 0.67}
\definecolor{darkcerulean}{rgb}{0.03, 0.27, 0.49}
\definecolor{morado}{rgb}{0.65, 0.14, 0.78}

\usepackage[colorlinks=true]{hyperref}
\hypersetup{citecolor=darkcerulean,linkcolor=darkcerulean, urlcolor=darkcerulean}

\DeclareGraphicsExtensions{.pdf,.png,.jpg}
\usepackage{pgfplots}
\pgfplotsset{compat=1.3}
\usepackage{placeins}
\usepackage{tabularx}
\usepackage{enumitem}
\usepackage{tikz}
\usepackage[font=footnotesize,margin=0.2cm]{caption}

\usepackage[caption=false]{subfig}
\usepackage{dsfont}

\usepackage[thinlines]{easytable}

\makeatother

\def\abs#1{\left| #1 \right|}
\def\norm#1{\left|\!\left| #1 \right|\!\right|}
\def\ll{\left\langle}
\def\rr{\right\rangle}

\def\tens#1{\pmb{\mathsf{#1}}}
\def\vec#1{\boldsymbol{#1}}

\def\R{\mathbb{R}}

\def\sym{\mathop{\mathrm{sym}}\nolimits}
\def\Rds{\mathbb{R}^{d \times d}_{\sym}}

\def\diam{\mathop{\mathrm{diam}}\nolimits} 		
\def\dim{\mathop{\mathrm{dim}}\nolimits} 

\def\loc{\mathop{\mathrm{loc}}\nolimits}

\def\tr{\mathop{\mathrm{tr}}\nolimits}

\def\diver{\mathop{\mathrm{div}}\nolimits} 

\def\d{{\rm d}}
\def\dx{\,\d x}
\def\ds{\,\d s}
\def\dt{\,\d t}
\def\difft{\d_t}

\def\Du{\BD\bu}
\def\Dv{\BD\bv}

\def\wconv{\rightharpoonup}
\def\wsconv{\overset{*}{\rightharpoonup}}

\def\toto{\rightrightarrows}

\def\Leb{\mathrm{L}} 
\def\l2d{ \Leb^2_{\diver}(\Omega)^d}

\def\Hn{ H^1_{\bn}(\Omega)}
\def\W1pn{ W^{1,p}_{\bn}(\Omega)^d}
\def\Hdivn{ H^1_{\bn,\divergence}(\Omega)}
\def\Bh{{W_h}}

\def\Bspace{B}

\def\Bdiv{{B_{\divergence}}}
\def\Wdiv{{W_{\divergence}}}

\def\b0{\vec{0}}

\def\bb{\vec{b}}

\def\bf{\vec{f}}

\def\bn{\vec{n}}
\def\bs{\vec{s}}
\def\bu{\vec{u}}
\def\bv{\vec{v}}
\def\bw{\vec{w}}
\def\bx{\vec{x}}

\def\bpsi{\vec{\psi}}

\def\bphi{\vec{\varphi}}
\def\bsigma{\vec{\sigma}}
\def\brho{\vec{\rho}}

\def\btheta{\vec{\theta}}

\def\B0{\tens{0}}

\def\BD{\tens{D}}

\def\BI{\tens{I}}

\def\BS{\tens{S}}

\def\Xdiv{X_{h,\diver}}
\def\Xh{X_{h}}

\def\gbd{\vec{\mathcal{g}}}

\def\Srel{\vec{\mathcal{s}}}
\def\Seps{\Srel_{\varepsilon}}

\def\wz{{z_1}}
\def\vz{{z_2}}
\def\yz{{z_3}}

\def\tria{\mathcal{T}_h}
\def\faces{\mathcal{F}_h}

\def\discr{\oldepsilon}

\def\Pid{{\Pi_\delta}}
\def\PiBh{{\Pi_h}}
\def\PiSZ{{I_h}}

\usepackage[top=2.75cm,bottom=2.75cm,left=3cm,right=3cm]{geometry}

\usepackage{lineno}
\newcommand*\linenomathpatch[1]{%
	\cspreto{#1}{\linenomath}%
	\cspreto{#1*}{\linenomath}%
	\csappto{end#1}{\endlinenomath}%
	\csappto{end#1*}{\endlinenomath}%
}
\linenomathpatch{equation}
\linenomathpatch{gather}
\linenomathpatch{multline}
\linenomathpatch{align}
\linenomathpatch{alignat}
\linenomathpatch{flalign}

\theoremstyle{theorem}

\newtheorem{theorem}{Theorem}[section]
\newtheorem{lemma}{Lemma}[section]
\newtheorem{proposition}{Proposition}[section]
\newtheorem{corollary}{Corollary}[section]
\newtheorem{definition}{Definition}[section]

\theoremstyle{definition}

\newtheorem{remark}{Remark}[section]
\newtheorem{example}{Example}[section]
\newtheorem{assumption}{Assumption}[section]

\begin{document}
	
	\title[A Nitsche method for fluids with dynamic boundary conditions]{A Nitsche method for incompressible fluids \\ with general dynamic boundary conditions
	}
	
	\author[P.A.~Gazca-Orozco]{Pablo Alexei Gazca-Orozco}
	\author[F.~Gmeineder]{Franz Gmeineder}
	\author[E.~Maringov\'a Kokavcov\'{a}]{{Erika  Maringov\'a Kokavcov\'a}}
	\author[T.~Tscherpel]{Tabea Tscherpel}
	
        \address[P.\ A. Gazca-Orozco]{Faculty of Mathematics and Physics, Charles University, Sokolovská 83, 186 75, Prague, Czech Republic}
	\email{gazca@karlin.mff.cuni.cz}
	\address[F. Gmeineder]{Fachbereich Mathematik und Statistik, Universität Konstanz, Universitätsstr.~10, 78464 Konstanz, Germany}
	\email{franz.gmeineder@uni-konstanz.de}
	\address[E. Maringov\'{a} Kokavcov\'{a}]{{Institute of Science and Technology Austria (ISTA), Am Campus 1, 3400 Klosterneuburg, Austria}}
	\email{erika.kokavcova@ist.ac.at}
	\address[T.\ Tscherpel]{Department of Mathematics, Technische Universit\"at Darmstadt, Dolivostraße 15, 64293 Darmstadt, Germany}
	\email{tscherpel@mathematik.tu-darmstadt.de}
	
	
	
	\date{\today}
	
	\setcounter{tocdepth}{1} 	
	\begin{abstract} 
		Both Newtonian and non-Newtonian fluids may exhibit complex slip behaviour at the boundary.  
		We {examine} a broad class of slip boundary conditions that generalises the commonly used Navier slip, perfect slip, stick-slip and Tresca friction boundary conditions. 
		In particular, set-valued, nonmonotone, noncoercive and dynamic relations may occur. For a unifying framework of such relations, we present a fully discrete numerical scheme for the time-dependent Navier--Stokes equations subject to impermeability and general slip-type boundary conditions on polyhedral domains. 
		Based on compactness arguments, we prove convergence of subsequences, finally ensuring the existence of a weak solution. 
		The numerical scheme uses a general inf-sup stable pair of finite element spaces for the velocity and pressure, a regularisation approach for the implicit slip boundary condition and, most importantly, a general Nitsche method to impose the impermeability and a backward Euler time stepping. 
		One of the key tools in the convergence proof is an inhomogeneous Korn inequality that includes a normal trace term. 
	\end{abstract}

    \maketitle

	\tableofcontents
	
	\section{Introduction}
	
	The behaviour of fluids crucially depends on both constitutive relations and on the type of boundary conditions. {Indeed, the boundary condition is a constitutive law describing the fluid behaviour at the fluid-solid interface.}  
	Well-investigated boundary conditions in PDE and numerical analysis -- such as periodic, natural and no-slip (homogeneous Dirichlet) boundary conditions -- are insufficient for a variety of practical applications. 
	Both for Newtonian and for non-Newtonian fluid flow, the interaction with the boundary can be much more complicated, and various types of slip and friction boundary conditions must be considered. 
	For impermeable boundaries, meaning that  the normal fluid velocity is zero at the boundary, such conditions relate the tangential velocity to the tangential wall shear stress. 
	
	Slip-type boundary conditions can already be observed for Newtonian fluids, and we refer the reader to ~\cite{NEBBC.2005,MalekRajagopal2023} for detailed discussions of the supporting experiments; for non-Newtonian fluids, they are thus at least as typical. 
	For instance, linear relations between the tangential velocities and the tangential wall shear stress can be found in the modelling of slip for the flow of blood through the aortic root, see~\cite{Chabiniok2022}, and energy preserving linear boundary conditions for the Navier--Stokes equations are investigated in~\cite{BotheKoehnePruess2013}. 
	As already observed and suggested by Stokes in 1845~\cite{Stokes1845}, the description of the flow in a circular pipe with  {sufficiently large} velocities  requires the incorporation of nonlinear boundary relations. 
	
	Nowadays, such nonlinear relations can be found in numerous models. 
	A typical instance thereof are the Tresca and stick-slip models, where slip occurs only if the wall shear stress exceeds a certain threshold, see, e.g.,~\cite{Denn2004} for an overview.  
    Explicit examples in applications include channel flows with beds covered by shingle or mud, iron melts leaving the furnace, and avalanches. 
    For non-Newtonian fluids such as polymer melts, slip boundary conditions are even more typical~\cite{Hervet2003} and may  exhibit some relaxation behaviour due to the entanglement of molecular chains. 
	This can be modelled by use of dynamic boundary conditions, that is, boundary conditions which include the time derivative of the velocity, see, e.g.,~\cite{Hervet2003,Hatzikiriakos2012}. Notably, the aforementioned relaxation behaviour is observed analytically for special scenarios in~\cite{Abbatiello2021}. Dynamic boundary conditions are also relevant for a wider class of equations, see, e.g.,~\cite{BotheKashiwabaraKoehne2017,DenkPlossRauEtAl2023}.
	
	Furthermore, in certain cases the relation between slip velocity and wall shear-stress can be nonmonotone~\cite{YarinGraham1998,Hervet2003}, as described in more detail in Section~\ref{sec:model-results} below. It is clear that similar contact and friction boundary conditions play an important role in elasticity too, leading to related mathematical models; see, e.g.,~\cite{Kikuchi1988}. 
	
	From a purely mathematical perspective, slip-type boundary conditions are also advantageous concerning regularity properties of the solutions.
	For instance, the pressure corresponding {to} the unsteady Navier--Stokes system with homogeneous Dirichlet boundary conditions is in general only a distribution in time.
	In contrast, when allowing for slip on the boundary, it is possible to prove that the pressure is an integrable function, see~\cite{BM.2017}. 
	
	Given the ubiquity of slip-type boundary conditions, it is particularly important to come up with stable and converging numerical approximation of solutions. 
    This is precisely the goal of this work; inspired by~\cite{Abbatiello2021}, we present a unifying framework that covers a wide array of slip boundary conditions and we develop appropriate finite element discretisations. 
	
	\subsection{Impermeability boundary  {conditions} {and Nitsche's method}}	
	As mentioned above, the present work focuses on impermeable boundaries, meaning that there is no inflow or outflow of the fluid through the boundary. 
	In view of numerical approximations of solutions, we shall employ Nitsche's method~\cite{Nitsche1971}. 
 In particular, we do  not impose the impermeability condition strongly but by penalising the normal velocity; see also~\cite{S.1995}. 
	This especially implies that the discretisation is non-conforming with respect to spaces of weakly differentiable functions with vanishing zero normal trace. 
	
	Originally derived for elliptic equations subject to inhomogeneous Dirichlet boundary conditions~\cite{Nitsche1971}, the Nitsche method  has been further developed and applied in a multitude of scenarios. 
	For a non-exhaustive list of contributions on the Stokes and Navier--Stokes equations with various choices of boundary conditions, we refer the reader to~\cite{FS.1995,S.1995,BeckerCapatinaLuceEtAl2015,GS.2022}. 
	
	Even though we will exclusively consider polyhedral spatial domains, we wish to stress that Nitsche's method proves particularly powerful when approximating curved domains by polyhedral ones. Indeed, imposing the impermeability boundary conditions strongly might cause the limits of approximate solutions to satisfy no-slip boundary conditions. This phenomenon is referred to as \emph{Babuška's paradox} and was first detected in elasticity~\cite{Babuska1959}, see also~\cite{MNP.2000}. 
	The analogous phenomenon in fluid dynamics was described in~\cite[Sec.~5]{V.1985}, and is due to the fact that imposing zero normal velocity on a polyhedral domain enforces the velocity to be zero at boundary vertices. In this regard, the Nitsche method is one way of  relaxing the impermeability boundary conditions. 
	Alternative strategies include Lagrange multipliers or  different penalisations, and we refer the reader to~\cite{V.1987,V.1991,BaenschDeckelnick1999,Basaric2022} for more detail; similar strategies in the context of elasticity are employed, e.g., in~\cite{ChoulyFabreHildEtAl2017,BartelsTscherner2024,Gustafsson2022}. 
	
	The focus of the present paper is on giving a unified approach to various boundary conditions for polyhedral domains. 
	Based on our above discussion, however, it is clear that the method presented here is finally geared towards applications on potentially curved domains. 
	With this being the goal of future work, we now proceed to discuss the   mathematical key challenges that already have to be tackled in the relevant base case of polyhedral domains as considered here.
	
	\subsection{Models and key difficulties}
	As alluded to above, we aim for a proof of convergence of numerical approximations to a weak solution. 
	This simultaneously serves as a proof of existence of weak solutions, and we briefly pause to address the challenges that arise for the class of boundary conditions to be examined here. 
	
	For slip boundary conditions, the impermeability boundary condition is usually supplemented by a relation between the tangential velocity on the boundary $\bu_\tau$ and the tangential part of the normal traction $- (\BS \bn)_{\tau}$, where $\BS$ denotes the deviatoric stress tensor. In studying such models, the aforementioned relations might be nonlinear, nonmonotone, implicit, noncoercive, or even time-dependent, including dynamic slip laws. Referring the reader to Section~\ref{sec:model:ex} for a more detailed account, we highlight the following aspects: \\
	
	\textbf{Nonlinear relations {and growth behaviour at infinity}.} 
	For linear boundary conditions, such as for the \emph{Navier friction law} or \emph{Navier slip boundary condition}, it is straightforward to identify the boundary condition in the limit. This is substantially harder for nonlinear relations such as, e.g., power-law relations, and crucially depends on the growth behaviour of the relations. 
	In the setting considered here, we will often be able to establish strong convergence of the traces, and if only weak convergence is available, we may still employ tools from monotone operator theory for a large range of growth exponents. 
	This enables us to cover relations with $r$-growth for a certain, yet natural range of $r>1$ as in the Navier slip or power-law boundary conditions, and in the case $r=1$ as, e.g., in the Tresca law. 
	
	In situations where the velocity is not an admissible test function, the convergence proof requires regularisation and truncation (similarly as in~\cite{DieningRuzickaWolf2010,BulicekGwiazdaMalekEtAl2012,DieningKreuzerSueli2013}), and thus is considerably more difficult and technical. 
	Most contributions on implicit constitutive relations imposed pointwisely, such as ~\cite{Abbatiello2021}, work in reflexive spaces and require $r>1$, see also~\cite{BM.2017,MZ.2018}.  
	The case $r = 1$ is usually handled within the framework of variational inequalities ~\cite{Fuj.1994,LeRoux.2005}. 
	\medskip
	
	\textbf{Implicit monotone relations.} 
	If the relations are monotone and not explicit, like the power-law relation, but instead are \emph{set-valued} or \emph{implicit},  like the Tresca~\cite{Fuj.1994} and stick-slip boundary conditions~\cite{BulicekMalek2016,BulicekMalekMaringova2023}, monotone operator theory may still be employed. 
	Then, graph approximation is used  to obtain an explicit relation, and the limit of this approximation has to be taken additionally to the discretisation limit. 
	Note that this can be achieved without the use of variational inequalities. 
	While the latter are classically used~\cite{Duvaut1976, Glowinski2008}, they often come with extra complications in numerical approximations. 	
	
	A key issue is the identification of boundary terms in the limit. While, even for implicit relations, this can be accomplished more easily in the presence of strong convergence of the boundary traces, mere weak convergence of the traces requires a Minty-type convergence lemma. 
	In particular, this necessitates a version which involves graph approximation. 
    In order to  apply this convergence result, we need to pass to the limit in an energy identity. 
	This is possible only for a skew-symmetric version of the Nitsche method. 
	\medskip
	
	\textbf{Noncoercive relations, Korn's inequality and Nitsche's method.} 
	The aforementioned boundary conditions are  \emph{coercive} in the sense that control of the boundary term entails estimates on the tangential velocity traces. 
	In this case, one may employ standard versions of Korn's inequality to bound the norms of the full gradient in terms of the symmetric gradient and a \emph{full} trace term, see~\cite{Nitsche1981, BulicekMalekRajagopal2007}. 
	However, there are also non-coercive boundary conditions such as the perfect slip boundary condition $-(\BS \bn)_{\tau} = \bm{0}$, for which no bounds on the full traces are available. 
    In particular, this requires  a Korn inequality which only involves the impermeability boundary condition $\bu\cdot\bn = 0$. 
	If the domain is not axisymmetric, such inequalities have been addressed in~\cite{DesvillettesVillani2002} for $H^{1}$-functions with zero normal traces; see also~\cite{Bauer2016,BP.2016,Dominguez2022}. 
	In view of Nitsche's method, which imposes the impermeability constraint weakly rather than strongly, such inequalities are not applicable, and  estimates on quotient spaces modulo rigid deformations as in~\cite[(2.5)-(2.6)]{V.1987} come closest to what is required here. This requires a novel inequality, which we state and prove via a self-contained and direct argument in Section~\ref{sec:korn}.  
	This comes with optimal conditions on the underlying domain and, {in view of}~\cite{LeRoux.2005, AABS.2019, WLHJ.2020}, is well-suited to handle mixed boundary conditions as well. 
	\medskip 
	
	\textbf{Nonmonotone relations.} Sharp versions of trace inequalities allow to handle also certain types of \emph{nonmonotone} relations between~$-(\BS\bn)_{\tau}$ and~$\bu_{\tau}$.  
	This concerns both implicit (coercive) relations and non-coercive (explicit) relations. 
	In the explicit case, we only need to assume certain bounds on the explicit {nonlinearity}. 
	On the other hand, in the implicit case, we proceed by splitting the relation into an implicit monotone relation and a linear relation with negative slope $-\lambda$.  
	Still, by the very method, it is essential to assume that $\lambda$ is sufficiently small. This gives rise to the notion of \emph{$\lambda$-monotone relations}, and it is customary to formulate the problem by means of hemivariational inequalities, see, e.g., in~\cite{Fang2016,MP.2018,MD.2022} or, regarding numerical approximations, ~\cite{Fang2020,HCJ.2021}. 
	Here, however, we use a formulation with a \emph{(pointwise) nonlinear relation} imposed on the boundary as in~\cite{Abbatiello2021,BulicekMalekMaringova2023} or, in the related context of nonmonotone boundary conditions from~\cite{BulicekMalek2019}. This allows to avoid the use of hemivariational inequalities and, aside from being conceptually simpler, is also advantageous, since standard solution methods (such as Newton's method) can be applied to compute the numerical solutions. Yet, we wish to point out that handling nonmonotone relations for boundary conditions is easier than nonmonotone relations on the bulk~\cite{LeRoux2013, JaneckaMalekProcircusaEtAl2019}. 
	\medskip
	
	\textbf{Time dependence and pressure integrability.} 
	Due to the convective term, the velocity function is not an admissible test function for the three-dimensional time-dependent Navier--Stokes equations. 
	For this reason, only an energy inequality rather than an equality is available. 
	For non-Newtonian fluids this makes the identification of the nonlinear stress strain relation substantially more difficult, and certain regularisations and truncations are employed~\cite{DieningRuzickaWolf2010,BulicekGwiazdaMalekEtAl2012}. 
	For the sake of readability, we will treat only the Newtonian case, and focus on the boundary conditions. 
	
	For the time-dependent problem the integrability of the pressure on Lipschitz domains depends on the type of boundary conditions, see~\cite{BulicekMalek2016}. 
	Aiming to cover the general case, we refrain from using any assumptions on integrability of the pressure. In particular, it would be preferrable to work in a pressure-free formulation. 
	For a convergence proof of this sort, a suitable Fortin operator would be useful~\cite{DieningKreuzerSueli2013,SueliTscherpel2020}. 
	Such operators preserve the divergence in the dual of the discrete pressure space, and here a local Fortin operator \emph{compatible with non-zero traces} would be required. 
	The construction of such an operator is highly dependent on the specific finite element pair and is not available to date. 
	In consequence, we work with the (local) Scott--Zhang operator, and additional  care is needed to show convergence of the pressure terms. 
	\medskip 
	
	\textbf{Dynamic slip laws and Gelfand triplets.} Lastly,  for dynamic slip-type boundary conditions, the relation between $\bu_\tau$ and $-(\BS\bn)_{\tau}$ also includes the term $\partial_t \bu$. 
	Such a term guarantees bounds on the velocity traces and hence improves the situation in the non-coercive case. 
	However, for the convergence proof and compactness, it requires a modified setup for the Gelfand triplet. 
	Triplets of this sort have been  presented in~\cite{Kovacs2017} for parabolic equations and in~\cite{Abbatiello2021} for incompressible fluids, {where} the function spaces consist of  divergence-free functions. 
	Since the numerical solutions are not divergence-free in general, we are bound to work with a Gelfand triplet that does \emph{not} include the divergence-constraint. In particular, this will be employed in the time-dependent nondynamic case \emph{and} in the dynamic one. 
	
	\subsection{Available results on numerical approximations} 
	Most numerical schemes introduced before cover one specific type of boundary condition. For example,  this includes ~\cite{V.1987,V.1991} for the Navier slip law,~\cite{DjokoKokoMbehouEtAl2022} for power-law slip relations, \cite{DM.2013,D.2014,DK.2016} for the stick-slip boundary conditions with $r = 2$ and in~\cite{HKSS.2018} for general $r\in (1,\infty)$, \cite{LL.2010,LA.2011,ABGS.2014,JHYW.2018,AABS.2019,DjokoKokoKucera2019,ZJK.2023} for the Tresca condition ($r=1$), and~\cite{Fang2020,HCJ.2021} for nonmonotone relations. 
	Most of these contributions are concerned with error estimates based on regularity assumptions and not with plain convergence to weak solutions. 
	Many numerical methods use (hemi-)variational inequalities to cover set-valued relations~\cite{DM.2013,D.2014,DK.2016} and $\lambda$-monotone relations~\cite{Fang2020,HCJ.2021}. 
	Exceptions include regularisation approaches in~\cite{D.2014,LA.2011,DjokoKokoKucera2019}, and a Lagrange multiplier method in~\cite{ZJK.2023} for the set-valued relation in the stick-slip and the Tresca friction law. 
	The majority of numerical methods imposes the impermeability strongly rather than by penalisation, except of the Lagrange multiplier method in~\cite{AABS.2019} and a DG method with a penalisation approach~\cite{JHYW.2018} for the Tresca boundary conditions.
	We also stress that many of these works are based on specific finite elements, as opposed to general inf-sup pairs, as done here. 
	To date, Nitsche methods have only been presented for linear boundary conditions~\cite{BansalBarnafiPandey2024}, and  no other relaxation methods for the impermeability for the time dependent problem seem to be available in  the general nonlinear case. 
	For dynamic boundary conditions, finite element methods have been developed for scalar parabolic equations with linear boundary conditions~\cite{Fairweather1979,Kovacs2017,AltmannZimmer2024}, but to the best of our knowledge not for fluid equations.  
	
	\subsection{Contributions of this work}		
	Here, for a general class of nonlinear, possibly noncoercive, nonmonotone, implicit slip-type boundary conditions, we present a numerical scheme and establish convergence of subsequences to a weak solution. 
	This is in a similar spirit as the convergence proofs for non-Newtonian fluid equations with homogeneous boundary conditions in~\cite{SueliTscherpel2020,FarrellGazcaOrozcoSueli2020}. The general setup allows us to cover a wide range of boundary conditions of practical interest. 
	For the space discretisation, we use a general mixed finite element method for the velocity and pressure with Nitsche penalisation for the impermeability boundary condition. 
	For the time discretisation, we employ a backward Euler time stepping. 
	In case an implicit relation is involved, we additionally use a regularisation.  
	As far as we are aware,  this is the first Nitsche method for incompressible fluid equations for such general slip-type boundary conditions, both in the stationary and in the unsteady case. 
	In proving the convergence result, we consider the simultaneous limit in the discretisation parameters in space and time $h,\delta$ as well as in the regularisation parameter $\epsilon$. 
	Since in the stationary case the method reads analogously and the arguments are easier, we present the full proof only in the time-dependent case and in the stationary case we only present the corresponding results. 
	
	In particular, the convergence proof establishes the existence of weak solutions. 
	Thus, it extends previous existence results~\cite{BulicekMalek2019,Abbatiello2021, BulicekMalekMaringova2023} to $r = 1$, to the noncoercive case, and to the nonmonotone case. Note that our approach  applies also to nonlinear dynamic boundary conditions for parabolic equations, and hence extends~\cite{Kovacs2017,AltmannZimmer2024} as well. 
 
	Finally, we present numerical simulations that showcase that the method is capable of approximating very general boundary conditions. 
    Notably, they confirm the relaxation behaviour for dynamic boundary conditions, meaning a nonmonotone in time behaviour of the tangential velocity, as observed in experiments~\cite{Hatzikiriakos2012} and as analysed in~\cite{Abbatiello2021}.
	
	\subsection{Outline}
	In Section~\ref{sec:model-results}, we introduce the models including boundary conditions in detail. 
	Moreover, we present the main result on convergence of solutions to the numerical scheme in form of a metatheorem, see Subction~\ref{sec:main-results}; examples for the general slip-type boundary conditions are discussed in Subsection~\ref{sec:model:ex}. 
	In Section~\ref{sec:korn},  we prove a Korn inequality with a normal trace terms  (see Theorem~\ref{thm:korn-n}); here, we also discuss geometric conditions on the boundary for the estimate to hold. 
	Section~\ref{sec:prel} gathers auxiliary results required for the convergence proof,  as well as the discretisation and approximation to be employed afterwards. 
	Section~\ref{sec:un-NS} introduces the Nitsche method, and it is here where we state and prove the precise version of our main convergence result, Theorem~\ref{thm:main-unsteady}. 
    Furthermore, for some special cases we obtain convergence results for a larger range or parameters in Propositions~\ref{prop:beta}--\ref{prop:skew-symmetric-Nitsche}. The corresponding results in the steady case are presented in Theorem~\ref{thm:main-steady}. 
    All convergence results and their respective conditions are summarised in Tables~\ref{tab:overview} and~\ref{tab:overview-steady}. 
    Lastly, in Section~\ref{sec:num-exp}, we present numerical experiments for the case of a steady flow with nonmonotone slip, Tresca friction and stick-slip boundary conditions, and dynamic boundary conditions. 
	
	\section{Model and main results}\label{sec:model-results}

	In this section, we introduce the Navier--Stokes equations subject to general slip boundary conditions, see Section~\ref{sec:bc-slip}. 
	We shall present the mathematical framework and showcase its richness by a range of examples in Section~\ref{sec:model:ex}. 
	Section~\ref{sec:main-results} states the main results on the existence of weak solutions and the convergence of an approximation by means of a Nitsche method. 
	For the precise convergence results, the reader is referred to Theorem~\ref{thm:main-unsteady} and to Propositions~\ref{prop:beta}--\ref{prop:skew-symmetric-Nitsche} below; see also Table~\ref{tab:overview} for an overview.

	\subsubsection*{Navier--Stokes equations}
	
	For $d \in \{2,3\}$ let $\Omega \subset \mathbb{R}^d$ be a bounded (connected) Lipschitz domain and let $T >0$ be a given final time. 
	We denote the time interval by $I \coloneqq (0,T)$ and the parabolic cylinder by $Q \coloneqq I \times \Omega$. 
	Further, let $\bf\colon Q \to \mathbb{R}^d$ be a given external force, and let $\bu_0 \colon \Omega \to \setR^d$ be a given initial velocity.
	We aim to find a velocity field $\bu\colon \overline{Q} \to \mathbb{R}^d$, a pressure function $\pi\colon Q \to \mathbb{R}$, and a trace-free stress tensor field $\BS\colon Q \to \Rds$ satisfying
	\begin{alignat}{2}
		\begin{aligned}
			\label{eq:NS-unst}
			\partial_t \bu +	\diver(\bu \otimes \bu) + \nabla \pi
			-\diver\BS &= \bf \qquad \quad &&\text{ in } Q,\\
			\diver\bu &= 0 \qquad \quad &&\text{ in } Q,\\
			\bu(0,\cdot) &= \bu_0 \quad && \text{ on } \Omega.
		\end{aligned}
	\end{alignat}
	The first equation represents the balance law of linear momentum, the second one the incompressibility constraint, and the last one the initial condition. 
	For simplicity, we assume that the fluid is Newtonian, i.e., for given viscosity $\nu >0$ the symmetric velocity gradient $\BD \bu = \tfrac{1}{2}(\nabla \bu + (\nabla \bu)^\top)$ and the stress tensor $\BS$ are related by the linear constitutive law
	\begin{align}\label{eq:NS-law}
		\BS = 2 \nu \BD \bu.
	\end{align}
	
	\subsection{General slip-type boundary conditions and assumptions} \label{sec:bc-slip}	

	We complement the system~\eqref{eq:NS-unst}, \eqref{eq:NS-law} with boundary conditions on $\Gamma\coloneqq \partial \Omega$. 
    Below in Remark~\ref{rmk:mixed-bc} we also discuss the case $\Gamma \subsetneq \partial \Omega$, where on the remaining boundary other boundary conditions are imposed. 
    
	Note that since $\Omega$ is a Lipschitz domain, the outer unit normal $\bn$ exists for $\mathcal{L}^{d-1}$ every boundary point. 
	For a function $\bv \colon \Gamma \to \setR^d$ its normal and tangential parts are defined by 
	\begin{align*}
		\bv_n  \coloneqq (\bv \cdot \bn) \bn 
		\quad \text{ and } \quad 
		\bv_{\tau} \coloneqq \bv - \bv_n,
	\end{align*}
	respectively. 
	Thus, $\bv$ decomposes into $\bv = \bv_{\bn} + \bv_{\tau}$. 
	In the following we supplement~\eqref{eq:NS-unst}, \eqref{eq:NS-law} with the \emph{impermeability boundary condition} 
	\begin{align}\label{eq:imperm}
		\bu \cdot \bn = 0 \quad \text{ on } I\times\Gamma,
	\end{align}	
	and the \emph{dynamic boundary condition}
	\begin{align}\label{eq:bc-Gamma}
		- (\BS \bn)_{\tau} = \bsigma + \beta \partial_t \bu \quad \text{ on } I\times\Gamma, 
	\end{align}
	where $\beta\geq 0$ is a given constant and $\BS = 2 \nu \BD \bu$ is the stress tensor in~\eqref{eq:NS-unst}. 
	The function $\bsigma$ is related to $\bu_{\tau}$ by a possibly implicit and possibly nonmonotone relation to be specified in the following. 
	Let us point out that $\bsigma$ is an auxiliary variable without a clear physical interpretation. 
	
	In the following, we shall consider two situations, namely
	\begin{itemize}		
		\item 
		explicit, not necessarily coercive relations, and 
		\item 
		possibly implicit, coercive relations. 
	\end{itemize}
	In both cases slightly nonmonotone relations are addressed. 
	We now proceed with a detailed formulation of both settings. 
	In combination, they cover a variety of dynamic boundary conditions to be discussed in Section~\ref{sec:model:ex} below. 
	
	\subsubsection{Explicit noncoercive relations}\label{sec:expl-rel} 
	
	We first consider the case of $\bsigma$ in~\eqref{eq:bc-Gamma} is represented in terms of an explicit function~$\mathcal{\bs}$ as
	\begin{align}\label{eq:expl-s}
		\bsigma = \mathcal{\bs}(\bu_\tau). 
	\end{align}
	Here,~$\mathcal{\bs}$  is not necessarily coercive, but it has certain boundedness properties.
	
	\begin{assumption}[explicit noncoercive relation]\label{assump:s-expl}
	We suppose that there is a continuous map $\Srel \colon \setR^d \to \setR^d$ which satisfies the following growth condition for some $r\in [1,\infty)$: 
	There exists a constant $\lambda\geq 0$ such that
	\begin{align*}
			\abs{ \Srel(\bv)} \leq \lambda (1 +  \abs{\Gamma}^{\tfrac{r-2}{2}}
			\abs{\bv}^{r-1})  \qquad \text{ for all } \bv \in \setR^d. 
	\end{align*}
        
	\end{assumption}
	\noindent 
	This is a very general assumption which also allows for nonmonotone explicit functions $\mathcal{\bs}$. 
	For explicit relations in the convergence proof we shall assume that $r \leq 2$ throughout. 
	In this situation, we may record that
	\begin{align}\label{est:case1-bd}
		\int_{\Gamma }\abs{\Srel(\bv) \cdot \bv} \ds &\leq \lambda   \left(c + \int_{\Gamma}\abs{\bv}^2  \ds \right),\\ \label{est:case1-bd2}
		\int_{\Gamma }\abs{\Srel(\bv)}^2 \ds &\leq c  \left( 1 +  \int_{\Gamma}\abs{\bv}^2  \ds \right), \qquad \text{ for any } \bv \in L^2(\Gamma)^d, 
	\end{align}
	 with a constant $c>0$ depending on $\mathcal{H}^{d-1}(\Gamma)$. 
	For future reference, we note that the boundary or trace terms appearing on the right-hand side of~\eqref{est:case1-bd}, \eqref{est:case1-bd2} will typically be estimated above by~$H^{1}$-norms in $\Omega$. 
	In particular, this will provide us with a weak substitute for the lack of monotonicity or coercivity conditions in the presence of Assumption~\ref{assump:s-expl}.
  In the convergence proof we shall assume that $\lambda$ is sufficiently small compared to the viscosity $\nu$, see Assumption \ref{as:param}. 

Instead of the estimate in Assumption~\ref{assump:s-expl} one may assume that 
\begin{align*}
\abs{\Srel(\bv)} \leq c (1 + \abs{\bv}^{r-1}) \quad \text{ for all } \bv \in \setR^d.
\end{align*}  
Then, the convergence proof below works with the obvious modification of Assumption~\ref{as:param} below. 
        
	\subsubsection{Implicit coercive relations}\label{sec:impl-mon-rel} 
	Let us now consider the case of implicit relations, including set-valued relations which satisfy certain coercivity estimates. 
	
	\subsubsection*{The monotone case}
	Let us first focus on the case, where $\bsigma$ and $\bu_{\tau}$ in~\eqref{eq:bc-Gamma} are related by a monotone, possibly implicit relation. 
	This will be formulated as
	\begin{align}\label{eq:gbd}
		\gbd(\bsigma, \bu_\tau) = \b0 \quad \text{ on } I\times\Gamma,
	\end{align}
	where $\gbd$  has a maximal monotone zero set. 
	This means that it is a monotone subset of $\mathbb{R}^d\times \mathbb{R}^d$, i.e., 
	\begin{align*}
		\skp{\bs_1- \bs_2}{\bv_1 - \bv_2} \geq 0 \qquad \text{ for any } (\bv_i, \bs_i)\in \setR^d, \; i \in \{1,2\}, \;\text{ with } \gbd(\bs_i,\bv_i) = \b0, 
	\end{align*}
	and that it is maximal as monotone set with respect to inclusion, see, e.g.,~\cite[Sec.~1]{Phelps1997}.
	
	\begin{assumption}[implicit monotone relation]\label{assump:gbd-mon}
		We assume that $\gbd \in C^{0}(\setR^d \times \setR^d)^d$ satisfies the following properties:
		\begin{enumerate}[label = (A\arabic*)]
			\item \label{itm:gbd-0} $\gbd(\b0,\b0) = \b0$; 
			\item \label{itm:gbd-q} $\gbd$ satisfies $r$-growth and coercivity for some $r\in [1,\infty)$ and $\frac{1}{r} + \frac{1}{r'} = 1$, in the sense that there exists a constant $c>0$ such that for all $(\bs, \bv) \in \setR^d \times \setR^d$ with $\gbd(\bs,\bv) = \b0$ one has  
			\begin{alignat*}{3}
				\bs \cdot \bv &\geq c (\abs{\bs}^{r'} + \abs{\bv}^r - 1)&& \qquad \qquad &\text{ if } r >1, \\
				\bs \cdot \bv &\geq c ( \abs{\bv} - 1) \quad   \text{ and }\quad & &\abs{\bs} \leq c \qquad 
				&\text{ if } r = 1; 
			\end{alignat*} 
			\item \label{itm:gbd-mon} 
			the zero set of $\gbd$, denoted by
			\begin{align*}
				\mathcal{A} \coloneqq \{(\bv,\bs) \in \setR^d \times \setR^d \colon \gbd(\bs,\bv) = \b0\}, 
			\end{align*}
			is a maximal monotone subset of $\setR^d \times \setR^d$; 
			\item \label{itm:gbd-mon-lp}	
			for any open set $M \subset \setR^n$, with $n \in \mathbb{N}$, the set defined by 
			\begin{align}\label{def:A}
				{A} \coloneqq \{(\bv,\bs) \in L^{r}(M)^d\times L^{r'}(M)^d \colon \gbd(\bs(\bx),\bv(\bx)) = \b0 \;\text{ for a.e.~$\bx\in M$}\}
			\end{align}
			is a maximal monotone subset of $L^{r}(M)^d\times L^{r'}(M)^d$, where $L^{r'}(M)^d$  is the dual space of $L^{r}(M)^d$. 
		\end{enumerate}
	\end{assumption}
	
	\noindent Also explicit coercive relations fit into this framework by considering $\gbd \colon \setR^d \times \setR^d \to \setR^d$ of the form 
	\begin{align}\label{eq:gbd-expl}
		\gbd(\bs,\bv) = \Srel(\bv) - \bs \qquad \text{ for } \bs, \bv \in \setR^d,
	\end{align}
	for some given monotone continuous function $\Srel \colon \setR^d \to \setR^d$.
	
	\begin{remark}[maximal monotonicity] For later purposes, we record the following:
		\begin{enumerate}[label = (\alph*)] \label{rmk:max-mon}
			\item 
			Condition~\ref{itm:gbd-mon-lp} implies that for any $(\bv,\bs) \in  L^{r}(M)^d\times L^{r'}(M)^d$ one has 
			\begin{align*}
				(\bv,\bs) \in A \quad \Leftrightarrow \quad 
				\int_{M} (\bs - \overline{\bs}) \cdot (\bv - \overline{\bv}) \ds  
				\geq 0 \quad \text{ for all } (\overline{\bv},\overline{\bs}) \in {A}. 
			\end{align*} 
			\item \label{itm:rmk-A4} 
			For $r \in (1,\infty)$, the boundedness and coercivity condition~\ref{itm:gbd-mon} implies condition~\ref{itm:gbd-mon-lp}, cf.~\cite[Thm.~(3.18)~(e)]{Browder1976} and~\cite[Thm.~1.9]{ChiadoPiatDalMasoDefranceschi1990}. 
			For this reason, for implicit relations we cannot drop the assumption of coercivity. 
			
			\item \label{itm:max-mon-r1} 
			For $r = 1$ the result in~\cite{ChiadoPiatDalMasoDefranceschi1990} is not available to ensure condition~\ref{itm:gbd-mon-lp}. 
			However, for special cases the condition can be verified by other means. 
			A theorem by Rockafellar~\cite[Thm.~A]{Rockafellar1970} states that for a Banach space $B$ and a lower semicontinuous, convex function $j \colon B \to \setR$, its subdifferential $\partial j \colon B \toto B'$ is a maximal monotone operator. 
			Here, $\toto$ is used to denote set-valued mappings.   	
			For this reason, for $r = 1$ one may use the following condition, which in combination with~\ref{itm:gbd-q} is sufficient to guarantee~\ref{itm:gbd-mon} and~\ref{itm:gbd-mon-lp}: 
			\smallskip  
			
			\begin{enumerate}[label = (A\arabic*)]   
				\item[(A5)] \label{itm:gbd-pot}
				There exists a convex function $\mathcal{j} \colon \setR^d \to \setR$, such that $\mathcal{A}$ defined in~\ref{itm:gbd-mon} is the graph of the subdifferential $\partial \mathcal{j} \colon \setR^d \toto \setR^d$ of $\mathcal{j}$.  
			\end{enumerate}
			\smallskip 
			By the Rockafellar theorem  mentioned before, condition~\ref{itm:gbd-mon} holds. 
			Furthermore, the functional $J \colon L^1(M)^d \to \setR$, defined by 
			$J(\bv) = \int_{M} \mathcal{j}(\bv) \dx $ 
			is convex and one can show that it is continuous. 
			Moreover, it is possible to verify that the graph of $\partial J \colon L^1(M)^d \toto L^\infty(M)^d$ is  precisely $A$ as defined in~\eqref{def:A}; cf.~\cite[Thm.~14.60]{Rockafellar2009}. 
			\item \label{itm:rmk-A3} 
			For $r \in (1,\infty)$, conditions that are easier to verify, and which are sufficient to ensure Assumption~\ref{assump:gbd-mon} are presented in~\cite[Sec.~2]{BulicekMalekMaringova2023}.  
			More specifically, therein the authors present a condition, that may replace~\ref{itm:gbd-mon}, and by~\ref{itm:rmk-A4} also~\ref{itm:gbd-mon-lp}. 
		\end{enumerate}
	\end{remark}
	
	\subsubsection*{The nonmonotone case}
	Our framework also allows to include certain nonmonotone implicit relations, namely nonmonotone relations that arise as sum of an implicit and a linear relation with negative slope. 
	More precisely, we consider the relation between $\bsigma$ and $\bu_\tau$ of the form 
	\begin{align}\label{eq:gbd-nonmon}
		\gbd(\bsigma + \lambda \bu_\tau,\bu_\tau) = \b0 \quad \text{ on } I \times \Gamma,
	\end{align}
	for some given constant $\lambda \geq 0$, and for $\gbd$ satisfying Assumption~\ref{assump:gbd-mon}. 
	This means that $\bsigma$ and $\bu_\tau$ are related by a so-called \emph{$\lambda$-monotone} relation, which is a weaker notion than monotonicity. 
	Indeed, for any $(\bsigma,\bu_\tau)$ and $(\overline \bsigma,\overline \bu_\tau)$ satisfying~\eqref{eq:gbd-nonmon} with $\lambda \geq 0$ one has that
	\begin{align}\label{est:lambda-mon}
		(\bsigma - \overline \bsigma) \cdot (\bu_\tau - \overline \bu_\tau) \geq - \lambda \abs{\bu_\tau - \overline \bu_\tau}^2.
	\end{align}	
	For $\lambda = 0$ the condition for monotonicity is recovered, while for $\lambda>0$ it is a weaker condition. 
	In the convergence proof we shall assume that the constants $\lambda$ and $r$ are sufficiently small. 
	The condition on $\lambda$ will be used in the a priori estimates to absorb the terms with negative sign.  
	
	Let us summarize the conditions in Section~\ref{sec:expl-rel} and~\ref{sec:impl-mon-rel}: 
	In the remaining work we consider for $\beta\geq 0$ and for $r \in [1,\infty)$ the general boundary conditions 
	\begin{align}\label{eq:genbc}
		- (\BS \bn)_{\tau}  &= \bsigma 
		+ \beta \partial_t \bu,
	\end{align}
	with one of the following: 
	\begin{alignat}{2} 
		\label{itm:case-expl-noncoerc-1}
		\bsigma &= \Srel(\bu_\tau) \qquad &&
		\text{ with  $\Srel$ satisfying Assumption~\ref{assump:s-expl}, }
		\\			
		\label{itm:case-impl-coerc-1} 
		\gbd(\bsigma  + \lambda \bu_\tau, \bu_\tau)&= \b0 
		&&\text{ with $\gbd$ satisfying Assumption~\ref{assump:gbd-mon},} 
	\end{alignat}
	with exponent $r > 1$, and constant $\lambda \geq 0$.
	
	The constants in Assumption~\eqref{assump:s-expl} and in~\eqref{eq:gbd-nonmon} are labelled by the same variable $\lambda$, since they will take the same role in the a priori estimates in Section~\ref{sec:unst-apriori} below. 
	
	\subsection{Statement of main convergence result and context}\label{sec:main-results}
	
	To give a brief idea on our main result let us state it in the form of a metatheorem.  
	The precise formulation and its proof can be found in Section~\ref{sec:un-NS}, see specifically Definition~\ref{def:unst-w-sol},  Theorem~\ref{thm:main-unsteady} and Propositions~\ref{prop:beta}--\ref{prop:skew-symmetric-Nitsche}, see also Table~\ref{tab:overview}. 
    We refer to Theorem~\ref{thm:main-steady} and Table~\ref{tab:overview-steady} for the corresponding results in the steady case. 
	Recall that we consider the Navier--Stokes equations~\eqref{eq:NS-unst}, \eqref{eq:NS-law} subject to impermeability boundary conditions~\eqref{eq:imperm} and~\eqref{eq:bc-Gamma} and either~\eqref{eq:expl-s} in case~\eqref{itm:case-expl-noncoerc-1}, or~\eqref{eq:gbd-nonmon} in case~\eqref{itm:case-impl-coerc-1} with $\lambda \geq 0$. 
	\medskip 
.8

	{\hfill
		\begin{minipage}{0.9\linewidth}
			\textbf{Metatheorem.}
			\emph{For a bounded polyhedral domain in dimension $d \in \{2,3\}$, let $\bf$ and $\bu_0$ be given functions. 
			Moreover, let $\nu>0$, $\beta \geq 0$ and $\lambda \geq 0$ be constants. 
			In case~\eqref{itm:case-expl-noncoerc-1} let Assumption~\ref{assump:s-expl} be satisfied and in case~\eqref{itm:case-impl-coerc-1}, let Assumption~\ref{assump:gbd-mon} and assume that there is a regularisation of the implicit relation with parameter~$\varepsilon \to 0$.}
			
			\emph{Consider a Nitsche method for an inf-sup stable pair of finite element spaces with space discretisation parameter $h>0$, where we employ
			an implicit Euler time stepping on a uniform time grid with time step size~$\delta >0$ for the discretisation in time.
			For a certain parameter range for $r$, under the assumption that~$\lambda$ is sufficiently small, and that the penalty parameter of the Nitsche method is sufficiently large,  the discrete solutions exist. Moreover, 
				there are suitable subsequences of the approximate velocity and pressure functions  which converge weakly to a weak solution to the equations in the simultaneous limit $h,\delta,  \varepsilon \to 0$. 
					In particular, a weak solution exists.}
		\end{minipage}  
		\hfill}
	\medskip 
	
	\medskip 
	With the choice $\Gamma \subsetneq \partial \Omega$, other (and, in particular, mixed) boundary conditions can be included as well; see Remark~\ref{rmk:mixed-bc} below for more details. 
	
	\subsection{Examples of boundary conditions}\label{sec:model:ex}
	Let us present a panorama of slip-type boundary conditions occurring in fluid flow models and having been investigated mathematically, that fit into the framework presented above.  
	Thanks to the general framework our results apply to all of them. 
	The case~$\beta = 0$ in~\eqref{eq:genbc} occurs both in the stationary as well as in the time-dependent problem, and we refer to the case $\beta>0$ as dynamic case. 
	
	\begin{example}[Navier boundary conditions]\label{ex:Navier}
		The classical \emph{Navier slip} or \emph{Navier friction} boundary condition, see, e.g.,~\cite{V.1987,Gjerde2022} with friction coefficient $\gamma \geq 0$ reads
		\begin{align}\label{eq:Navier-slip}
			- (\BS \bn)_{\tau} = \gamma \bu_\tau. 
		\end{align}
		For $\gamma = 0$ this describes perfect slip for a flow without friction, whereas no-slip  is formally recovered in the limit $\gamma \to \infty$. 
		It is an explicit relation of the form~\eqref{eq:expl-s} with~$\mathcal{\bs}(\bv) = \gamma \bv$, which is coercive with $r = 2$, provided that~$\gamma >0$.  
		
		The case~$\gamma >0 $ is covered by Assumption~\ref{assump:gbd-mon} and case~\eqref{itm:case-impl-coerc-1}. 
		The case of perfect slip~$\gamma = 0$ is trivially covered by~\eqref{itm:case-expl-noncoerc-1} with $\Srel(\bv) = \b0$, and no further assumptions are needed. 
		Negative parameter~$\gamma$ in~\eqref{eq:Navier-slip} can also be treated in case~\eqref{itm:case-expl-noncoerc-1} for $r = 2$. 
		As argued in~\cite{Gjerde2022} such boundary conditions may describe an active wall. 
		In this case for existence and the convergence proof we  have to assume that $\lambda = - \gamma $ is sufficiently small, see Assumption~\ref{as:param} below.    
		Our convergence result in Theorem~\ref{thm:main-unsteady}  below applies with $\beta \geq 0$ in~\eqref{eq:genbc}. 
		Also the generalisation for a matrix $\gamma$ is contained in our framework, see~\cite{MaekawaMazzucato2018} for a review of mathematical results.  
		Numerical methods are presented, e.g., in~\cite{V.1987,V.1991}. 
	\end{example}
	
	\begin{example}[Power-law relations]\label{ex:powerlaw}
		All monotone explicit laws with $r$-growth, for $r \in (1,\infty)$ fit into our framework. 
		For example we consider the following explicit function $\Srel$, see also~\cite[Table~3]{BulicekMalekMaringova2023},  
		\begin{align*}
			\mathcal{s}(\bv) &= c \abs{\bv}^{r-2} \bv,\\
			\mathcal{s}(\bv) &= c (1 + \abs{\bv})^{r-2} \bv,
		\end{align*}
		for $c >0$ constant or a positive definite matrix and $r \in (1,\infty)$. 
		They are explicit, but coercive, and hence are covered by case~\eqref{itm:case-impl-coerc-1} with $\lambda = 0$, where $\gbd$ is represented by~\eqref{eq:gbd-expl} and satisfies Assumption~\ref{assump:gbd-mon} with $r$. Our convergence result in Theorem~\ref{thm:main-steady} applies for any $r \in [1,\infty)$. 
		Power-law models can be reformulated as relation of the form~$\bv = \mathcal{v}(\bsigma)$, as commonly used in the engineering community, cf.~\cite{KalikaDenn1987,HatzikiriakosDealy1991}. 
		
		Existence of weak solutions is available in~\cite{LeRoux2023} and numerical schemes for the corresponding Stokes equations are investigated in~\cite{DjokoKokoMbehouEtAl2022} for $r\in (1,2)$. 
	\end{example}
	
	\begin{table}[t] \small
		\begin{TAB}(r)[5pt]{|c|c|c|}{|c|cc|c|}
			& 	explicit & implicit  \\
			monotone	&  Navier friction $r =2 $ (Ex~\ref{ex:Navier}) & Tresca slip $r = 1$ (Ex.~\ref{ex:Tresca}) \\
			& power-law type $r > 1$ (Ex.~\ref{ex:powerlaw}) 
			& stick-slip $r > 1$ (Ex.~\ref{ex:Tresca}) \\
			nonmonotone  & 
			Le~Roux, Rajagopal (Ex.~\ref{ex:gbd-nonmon})
			& 	nonmonotone friction (Fang et al.) (Ex.~\ref{ex:Fang})
			\\
		\end{TAB}
		\caption{Examples of coercive boundary conditions~\eqref{itm:case-impl-coerc-1}.  
		}
		\label{tbl:examples}
	\end{table}
	
	\begin{example} \label{ex:gbd-nonmon}
		Explicit nonmonotone relations as in~\cite{LeRoux2013} of the form 
		\begin{align*}
			\bsigma = \Srel(\bu_\tau) = \left(a(1 + b \abs{\bu_{\tau}}^2)^{\theta} + c \right)\bu_\tau,				
		\end{align*}
		for constants~$a, b, c >0$ and~$\theta \in \setR$  are contained in our framework.  
		If $\theta>-1/2$, then $\Srel$  is monotone, and otherwise, it may be nonmonotone. 
		For the explicit constants, we refer to~\cite{LeRoux2013}. 	
		Note that for~$\theta \leq 0$ the relation is $2$-coercive, and hence it is covered by case~\eqref{itm:case-impl-coerc-1}.  
		Alternatively, since the relation is explicit, one can also work with case~\eqref{itm:case-expl-noncoerc-1}. 
		In both cases, for the purpose of the convergence proof we have to assume that the parameter $\lambda \geq 0$ in Assumption~\ref{assump:s-expl} and in~\eqref{eq:gbd-nonmon}, respectively, is sufficiently small, see    Assumption~\ref{as:param} below. 
	\end{example}
	
	\begin{example}[Tresca slip and stick-slip boundary condition]\label{ex:Tresca}
		The Tresca slip and stick-slip boundary conditions are given by 	
		\begin{equation}\label{def:Tresca}
			\begin{array}{rlr}
				\bsigma &= \gamma_\star \bu_\tau + \mu_\star \frac{\bu_\tau}{|\bu_\tau|} \;\; & \text{ if }\bu_\tau \neq 0  \\
				|\bsigma| &\leq \mu_\star  & \text{ if }\bu_\tau = 0
			\end{array}
			\quad
			\text{ for constants  }	\gamma_\star,\mu_\star \geq 0.
		\end{equation}
		For $\gamma_{\star} = 0$ this reduces to the Tresca slip boundary condition~\cite{Fuj.1994}, and for~$\gamma_{\star}>0$ it is referred to as stick-slip boundary condition~\cite{LeRoux.2005}.  
		
		For $\mu_\star>0$ there is no explicit relation to express $\bsigma$ in terms of $\bu_\tau$. 
		Defining with $s^+ \coloneqq \max(0,s)$ the mappings 
		\begin{alignat}{3} \label{eq:gbd-Tresca}
			\gbd(\bs,\bv) &= \bv -  \frac{1}{\mu_\star} (\abs{\bs + \bv} - \mu_\star)^+ \bs  \qquad 
			& \text{ if } \gamma_\star = 0,\\
			\label{eq:gbd-stickslip}
			\gbd(\bs,\bv) &= \bv - 
			\frac{1}{\gamma_\star} \frac{(\abs{\bs} - \mu_\star)^+}{\abs{\bs}} \bs   
			& \text{ if } \gamma_\star > 0,
		\end{alignat}
		one can show that $\gbd(\bsigma, \bv) = \b0$ if and only if~\eqref{def:Tresca} is satisfied, cf.~\cite{BM.2017}. 
		
		For $\gamma_\star>0$ it is straightforward to see that $\gbd$ satisfies  Assumption~\ref{assump:gbd-mon} with $r = 2$. 
		Note that by choosing $\bsigma = \gamma_{\star} \abs{\bu_\tau}^{r-2}  \bu_\tau  + \mu_{\star} \frac{\bu_\tau}{\abs{\bu_\tau}}$, if $\bu_\tau \neq \b0$, or by combining it with any power-law type relation as in Example~\ref{ex:powerlaw} one can generalise this to $r \in (1,\infty)$. 
		The analogous laws in the bulk correspond to the  Bingham model for $r = 2$ , and to the Herschel--Bulkley model for $r \in (1,\infty)$.  
		
		For~$\gamma_\star = 0$ we wish to verify  Assumption~\ref{assump:gbd-mon} with $r = 1$. 
		The only difference compared to the previous case, where $r = 1$, is the proof of maximal monotonicity of $\mathcal{A}$ and $A$ as in~\ref{itm:gbd-mon} and~\ref{itm:gbd-mon-lp}, respectively.  
		As mentioned in Remark~\ref{rmk:max-mon}~\ref{itm:max-mon-r1}, it suffices to find a convex potential $j \colon \setR^d \to \setR$ such that $\mathcal{A} = \partial \mathcal{j}$; see~(A5). 
		It is straightforward to check that $\mathcal{j}(\bv) = \mu_\star \abs{\bv}$ serves the purpose. 
		Altogether, the Tresca slip and stick-slip boundary conditions are covered by case~\eqref{itm:case-impl-coerc-1} choosing $\lambda = 0$ in~\eqref{eq:genbc}, with $r = 1$ and $r = 2$, respectively, and our convergence result applies. 
		
		Existence of weak solutions to the Navier--Stokes equations with Tresca friction is proved in~\cite{Fuj.1994,Fuj.2002}, and for generalisations in~\cite{LeRoux.2005, LeRoux2007}.  
		For the stick-slip-type boundary for non-Newtonian fluids we refer to~\cite{BulicekMalek2016} for~$r = 2$ and to~\cite{BulicekMalekMaringova2023}  for $r \in (1,\infty)$. 
		
		Numerical methods for the (Navier--)Stokes equations  with stick-slip conditions with~$r = 2$ and direct imposition of the impermeability  are presented in~\cite{DM.2013,D.2014,DK.2016}. 
		An extension to~$r \in (1,\infty)$ is presented for the stationary Stokes equation in~\cite{HKSS.2018}. 
		
		Approximations for the (Navier)--Stokes equations  with Tresca boundary law are considered in~\cite{LL.2010, LA.2011, Kas.2013c, ABGS.2014,JHYW.2018,AABS.2019}. 
	\end{example}
	
	\begin{example}[implicit, nonmonotone]\label{ex:Fang} 
		Also nonmonotone and implicit relations are covered. 
		The following example is presented in~\cite{Fang2020}, where it is handled by means of variational inequalities: 
		\begin{equation}\label{eq:nonmon-impl}
			\begin{array}{rlr}
				\bsigma &= \mu(\abs{\bu_\tau})\frac{\bu_\tau}{\abs{\bu_\tau}}  & \text{ if }\bu_\tau \neq 0  \\
				|\bsigma| &\leq \mu(0)  & \text{ if }\bu_\tau = 0
			\end{array}
			\qquad \text{ where }\;
			\mu(s) \coloneqq (a-b)e^{-\alpha s} + b, \;\; \text{ for } s \geq 0,
		\end{equation}
		for given constants $a > b \geq 0$ and $\alpha \geq 0$. 
		For $\lambda \coloneqq \alpha (a-b)$  the set-valued mapping 
		\begin{align*}
			\bv \mapsto \{\bsigma + \lambda \bv \colon  (\bsigma, \bv) \text{ satisfy~\eqref{eq:nonmon-impl}} \}
		\end{align*}
		is maximally monotone and $2$-coercive. 
		Hence, the implicit relation can be cast in the form~\eqref{itm:case-impl-coerc-1} with $\gbd$ satisfying Assumption~\ref{assump:gbd-mon} with $r = 2$ and $\lambda$ as above. 
		Our convergence proof below applies if $\lambda = \alpha (a-b)$ satisfies Assumption~\ref{as:param}.
		
		Existence results in this direction can be found in~\cite{Fang2016,MP.2018,MD.2022}, and numerical schemes are presented in~\cite{Fang2020,HCJ.2021}. 			
	\end{example}
	
	All previous examples can be considered for stationary and for time-dependent problems, which leads to~\eqref{eq:genbc} with $\beta = 0$.

	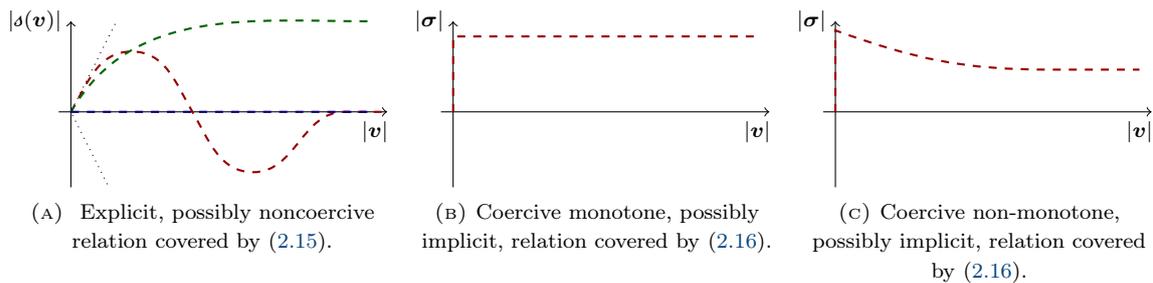
\begin{figure}
		\centering
		\subfloat[a][{\centering { 
				Explicit, possibly noncoercive relation covered by~\eqref{itm:case-expl-noncoerc-1}.}
		}]{{%
				\begin{tikzpicture}[scale = 0.8]
					\draw[->] (0,-1.25) -- (0,1.5); 
					\draw[->] (-0.2,0) -- (5.2,0);
					\draw[-,thick,red!60!black,dashed] (0,0) [out = 60, in = 180] to (1,1) [out = 0, in =120] to (2,0) to [out=-60, in =180] (3,-1) to [out = 0, in = 180] (4.5,0) -- (5.15,0);
					\draw[-,thick,green!40!black,dashed] (0,0) [out = 60, in = 180] to (5,1.5);
					\draw[-,thick,blue!50!black,dashed] (0,0) -- (5,0);
					\draw[dotted] (0.6,-1.2) -- (0,0) -- (0.75,1.5);
					\node[left] at (0,1.5) {\footnotesize $\abs{\Srel(\bv)}$};
					\node[below] at (5,0) {\footnotesize $\abs{\bv}$};
				\end{tikzpicture}
		}}%
		\subfloat[b][{\centering {Coercive monotone, possibly implicit, relation covered by~\eqref{itm:case-impl-coerc-1}}.}]{{%
				\begin{tikzpicture}[scale = 0.8]
					\draw[->] (0,-1.25) -- (0,1.5); 
					\draw[->] (-0.2,0) -- (5.2,0);
					\draw[-,thick,red!60!black,dashed] (0,0) to (0,1.25) -- (5,1.25);
					\node[left] at (0,1.5) {\footnotesize $\abs{\bsigma}$};
					\node[below] at (5,0) {\footnotesize $\abs{\bv}$};
				\end{tikzpicture}
		}}%
		\subfloat[c][{\centering {Coercive nonmonotone, possibly implicit, relation covered by~\eqref{itm:case-impl-coerc-1}}.}]{{%
				\begin{tikzpicture}[scale = 0.8]
					\draw[->] (0,-1.25) -- (0,1.5); 
					\draw[->] (-0.2,0) -- (5.2,0);
					\draw[-,thick,red!60!black,dashed] (0,0) to (0,1.35) to [out =-20, in = 180] (3.5,0.7) -- (5,0.7);
					\node[left] at (0,1.5) {\footnotesize $\abs{\bsigma}$};
					\node[below] at (5,0) {\footnotesize $\abs{\bv}$};
				\end{tikzpicture}
		}}\\
		\caption{Examples of relations covered by the framework in Section~\ref{sec:bc-slip}.}%
		\label{fig:examples}
	\end{figure}

	\begin{example}[time-dependent and dynamic boundary conditions]
		For the time-dependent problem boundary conditions of the form 
		\begin{align}
			g(- (\BS \bn)_{\tau})  &= \bu_\tau 
			+ \beta \partial_t \bu
		\end{align}
		are proposed in the context of polymer melts~\cite{Hatzikiriakos2012}, where $g$ is for example a power-law relation and $\beta \geq 0$. 
		We refer to them as dynamic, if $\beta>0$.  
		If $g$ is the identity, we have
		\begin{align}
			- (\BS \bn)_{\tau}  &= \bu_\tau 
			+ \beta \partial_t \bu.
		\end{align}
		Thus, the relation is of the form~\eqref{eq:genbc} with~$\bsigma = \bu_\tau$, i.e., we have an explicit linear relation. 
		This means that it is an extension of the Navier slip boundary condition in Example~\ref{ex:Navier} which allows for a relaxation  effect. 
		
		More general dynamic boundary conditions of the form~\eqref{eq:genbc}, with $\beta>0$ and  with $\bsigma$ and $\bu_\tau$ related by an $r$-coercive, possibly implicit maximal monotone relation are considered in~\cite[Ch.~8-10]{Maringova2019} and in~\cite{Abbatiello2021}. 
		This corresponds to case~\eqref{itm:case-impl-coerc-1} with $\lambda  = 0$ and $r \in (1,\infty)$. 
		Therein, existence of solutions is proved for $r\in (1,\infty)$ and $\beta > 0$. 
		For the non-dynamic case $\beta = 0$ existence of solutions was proved in~\cite{BulicekMalekMaringova2023}. 
		While for time-dependent problems a Nitsche type method has been tested, e.g., in~\cite{BansalBarnafiPandey2024}, for dynamic slip boundary conditions we are not aware of any contributions. 
	\end{example}
	
	\section*{General notation}
	\noindent Even though mostly standard, we wish to comment on selected aspects regarding our notation:
	
	\textbf{Vectors, sets and measures.} 
	For vectors $\bx\in\R^{n}$ and matrices $A\in\R^{n\times n}$, for some $n \in \mathbb{N}$, we denote by $|\bx|$ the Euclidean norm of $\bx$ and by $|A|$ the Frobenius norm of $A$. 
	These are induced by the corresponding Euclidean and Frobenius inner products, both of which are denoted by $\skp{\cdot}{\cdot}$. 
	To highlight that a certain quantity is vectorial rather than scalar, we write $\bv$ instead of $v$. 
	
	Throughout, $\Omega\subset\R^{d}$ for $d \in \{2,3\}$ denotes an open and bounded (spatial) domain with Lipschitz boundary $\partial\Omega$, and $\bn\colon\partial\Omega\to\mathbb{S}^{d-1}\coloneqq \{\bx\in\R^{d}\colon\,|\bx|=1\}$ is the outer unit normal to $\partial\Omega$. 
	The $d$-dimensional Lebesgue measure is denoted by $\mathcal{L}^{d}$, and we write $\mathcal{H}^{d-1}$ for the $(d-1)$-dimensional Hausdorff measure. 
	To alleviate notation, we sometimes use the convention $\mathrm{s}=\mathcal{H}^{1}$ or $\mathrm{s}=\mathcal{H}^{2}$ for the surface measure in two or three dimensions, respectively. 
    Both for Lebesgue measurable  and for Hausdorf-measurable subsets $S\subset \R^{d}$  we denote by $\abs{S}$ the $d$-dimensional Lebesgue measure, or the $(d-1)$-dimensional Hausdorf measure, respectively. 
    Which of the two is considered will be clear from the context. 
    
	\textbf{Function spaces.} 
	As usual, we write $C^{0,1}(\Omega)$ and $C^{0,1}(\Omega)^{d}$ for the Lipschitz functions and vector fields on $\Omega$, respectively. 
	We use the standard notation for the Lebesgue spaces $L^p(\Omega)$ and the Sobolev spaces $W^{1,p}(\Omega)$, where  $p \in [1,\infty]$, and the vector-valued variants $L^{p}(\Omega)^{d}$ and $W^{1,p}(\Omega)^{d}$. 
	It is customary to set 
	\begin{align*}
		L^p_0(\Omega)\coloneqq \Big\{v \in L^p(\Omega)\colon \int_{\Omega} v \dx = 0\Big\}.
	\end{align*}
	If $p=2$, we write $H^{1}(\Omega)=W^{1,2}(\Omega)$, and denote by $H_{0}^{1}(\Omega)$ the closure of $C_{c}^{\infty}(\Omega)$ with respect to the $H^{1}$-norm $\|\cdot\|_{H^{1}(\Omega)}$. 
	Moreover, whenever a $\mathcal{H}^{d-1}$-measurable subset $\Gamma\subset\partial\Omega$ is fixed, we define 
	\begin{align*}
		\Hn & \coloneqq \overline{\{\bf \in C^{\infty}(\overline \Omega)^d\colon \bf \cdot \bn |_{\Gamma} = 0 \}}^{\|\cdot\|_{H^{1}(\Omega)}}, \\ 
		\Hdivn & \coloneqq \overline{\{\bf \in C^{\infty}(\overline \Omega)^d\colon \bf \cdot \bn |_{\Gamma} = 0 \text{ and } \divergence \bf = 0\}}^{\|\cdot\|_{H^{1}(\Omega)}},
	\end{align*}
	but since $\Gamma=\partial\Omega$ in most parts of the paper, no ambiguities will arise from this. 
	The image of the boundary trace operator $\mathrm{tr}\colon W^{1,p}(\Omega)^{d}\to W^{1-1/p,p}(\partial\Omega)^{d}$ involves the fractional Sobolev spaces $W^{s,p}(\partial\Omega)^{d}$. 
	For $0<s<1$ and $1\leq p<\infty$, we recall that $W^{s,p}(\partial\Omega)^{d}$ is defined as the collection of all $\mathcal{H}^{d-1}$-measurable $v\colon\partial\Omega\to\R^{d}$ such that 
	\begin{align*}
		\|v\|_{W^{s,p}(\partial\Omega)}\coloneqq \Big(\|v\|_{L^{p}(\partial\Omega)}^{p}+\iint_{\partial\Omega\times\partial\Omega}\frac{|v(x)-v(y)|^{p}}{|x-y|^{d-1+sp}}\,\mathrm{d}\mathcal{H}^{d-1}(x) \,\mathrm{d}\mathcal{H}^{d-1}(y)\Big)^{\frac{1}{p}}<\infty. 
	\end{align*}
	With the obvious modifications, the spaces $W^{s,p}(\Omega)^{d}$ are defined analogously. 
	Moreover, for an $\mathcal{L}^{d}$-measurable subset $\omega\subset\Omega$, we write 
	\begin{align*}
		\langle \bu,\bv\rangle_{\omega} \coloneqq \int_{\omega}\bu\cdot \bv\,\mathrm{d}x\qquad \text{for measurable functions}\;\bu,\bv\colon\omega\to\R^{d},
	\end{align*}
	whenever this expression is well-defined, and simply $\langle \bu,\bv\rangle \coloneqq \langle \bu,\bv\rangle_{\Omega}$; similar notation will be employed for measurable subsets of $\partial\Omega$. 
	Different from this convention, a capital subscript as in $\langle\cdot,\cdot\rangle_{X} \coloneqq \skp{\cdot}{\cdot}_{X',X}$ denotes the duality pairing between a Banach space $X$ and its dual $X'$. 
	Lastly, we write $(\cdot,\cdot)_{B}$ for the inner product in Hilbert spaces $B$.
	
	\textbf{Inequalities.} 	
	For two non-negative quantities $A,B$, we write $A\lesssim  B$ if there exists a constant $c>0$ independent of $A$ and $B$ such that $A\leq cB$. 
	Finally, $c>0$ and $C>0$ denote two generic constants whose values might change from one line to the other, and their precise values will only be specified if they are required in the sequel. 
	
	\section{A Korn-type inequality with normal trace terms}
	\label{sec:korn}
	In this section, we establish a Korn-type inequality which involves normal traces on a fixed part $\Gamma$ of the boundary of the domain. 
	This constitutes a key ingredient in the convergence proof of the Nitsche method in the noncoercive case, since we do not have a priori control of the tangential traces in this situation. 
	Moreover, by the very nature of the Nitsche method, the zero normal trace condition is imposed by penalisation. In particular, our finite element functions do not have zero normal trace.
	
	To the best of our knowledge, the available Korn inequalities involving trace terms either include the full traces, see~\cite{Pompe2003}, or vanishing partial (normal) traces as in ~\cite{DesvillettesVillani2002} and~\cite{BP.2016,Bauer2016}. 
	Theorem~\ref{thm:korn-n} below extends these results by not requiring the normal trace to vanish; different from e.g. ~\cite{BulicekMalekRajagopal2007}, it only requires bounds on the normal traces on a fixed part of the boundary.
 Since it might be of independent interest, we state the theorem in slightly higher generality than is actually required here. 
	
	 Throughout this section, let $d\in\mathbb{N}$ be arbitrary, and let $\Omega \subset \setR^d$ be a bounded Lipschitz domain. 
	We record that the nullspace of the symmetric gradient is given by the space of \emph{rigid deformations} 
	\begin{align*}
		\mathcal{R}(\Omega) \coloneqq \operatorname{ker} \BD =  
		\{\bx \mapsto A\bx + \bb \text{ for some } A \in \setR^{d \times d} \text{ with } A = -A^\top \text{ and } \bb \in \setR^d \}. 
	\end{align*} 
	To state our result conveniently, we work subject to the following geometric assumption: 
	\begin{assumption}[geometry of $\Gamma$]\label{assump:dom}
		We assume that  $\Gamma\subset\partial\Omega$ is an $\mathcal{H}^{d-1}$-measurable subset, and that 
		\begin{align}\label{eq:trampolin}
			\bv \mapsto \norm{\bv \cdot \bn}_{L^q(\Gamma)}\qquad \text{is a norm on the  rigid deformations $\mathcal{R}(\Omega)$}
		\end{align}
		for some $1\leq q \leq \infty$, where $\bn\colon\partial\Omega\to\mathbb{S}^{d-1}$ is the outer unit normal to $\partial\Omega$. 
	\end{assumption}
	It is clear that if~\eqref{eq:trampolin} holds for some $1\leq q\leq\infty$, then it holds for all $1\leq q\leq \infty$. 
	Deferring the discussion of the geometric impact of Assumption~\ref{assump:dom} on $\Omega$ to the end of the section, the required Korn-type inequality is as follows: 
	
	\begin{theorem}[Korn-type inequality with normal traces I]\label{thm:korn-n}
		Let $\Omega\subset \setR^d$ be an open and bounded Lipschitz domain, and suppose that $\Gamma\subset\partial\Omega$ satisfies Assumption~\ref{assump:dom} with some $1\leq q \leq\infty$. 
		Moreover, let $1<p<\infty$  be such that the trace operator is continuous as a mapping $ W^{1,p}(\Omega)^{d}\to L^{q}(\partial\Omega)^{d}$. 
		Then we have the estimate
		\begin{align}\label{eq:kornetto0}
			\norm{\bu}_{W^{1,p}(\Omega)} &\lesssim \norm{\BD \bu}_{L^p(\Omega)} + \norm{\tr(\bu)\cdot \bn}_{L^q(\Gamma)}\qquad \text{ for any } \bu \in W^{1,p}(\Omega)^{d}. 
		\end{align}
		In particular, we have  
		\begin{align}\label{eq:kornetto}
			\norm{\nabla \bu}_{L^p(\Omega)} &\lesssim \norm{\BD \bu}_{L^p(\Omega)} + \norm{\tr(\bu)\cdot \bn}_{L^q(\Gamma)}\qquad \text{ for any } \bu \in W^{1,p}(\Omega)^{d}.
		\end{align}
		Here, the underlying constants only depend on $\Omega$, $\Gamma$, $p$ and $q$. The same conclusion holds true if $\Omega\subset\R^{d}$ is open and bounded with Lipschitz boundary and has finitely many connected components. 
	\end{theorem}
	\begin{proof}
		Let $\bu \in W^{1,p}(\Omega)^{d}$. 
		Since $\Omega$ is connected and has Lipschitz boundary, it is well-known that there exists a $\brho \in \mathcal{R}(\Omega)$, independent of $p$, such that 
		\begin{align}\label{est:korn-1}
			\norm{\bu - \brho}_{W^{1,p}(\Omega)} \lesssim \norm{\BD \bu}_{L^p(\Omega)},
		\end{align} 
		for instance, see ~\cite{Necas1966} for Korn-type inequalities of this sort. 
		Note that $\brho$ can be chosen as the $L^2$-projection of $\bu$ onto the finite dimensional space $\mathcal{R}(\Omega)$. 
		For this choice of $\brho$, we have 
		\begin{align}\label{est:korn-2}
			\norm{\bu}_{W^{1,p}(\Omega)}
			& \leq 	\norm{\brho}_{W^{1,p}(\Omega)} + \norm{\bu - \brho}_{W^{1,p}(\Omega)}. 
		\end{align}
		For the first term we may use the fact that on the finite dimensional space $\mathcal{R}(\Omega)$, all norms are equivalent and that, by Assumption~\ref{assump:dom}, $\bv \mapsto \norm{\bv \cdot \bn}_{L^q(\Gamma)}$ is a norm on $\mathcal{R}(\Omega)$. 
		In combination with $\mathrm{tr}\colon W^{1,p}(\Omega)^{d}\to L^{q}(\partial\Omega)^{d}$, we arrive at 
		\begin{align}\label{est:korn-3}
			\norm{\brho}_{W^{1,p}(\Omega)}
			\lesssim \norm{\brho\cdot \bn}_{L^q(\Gamma)} 
			\leq \norm{(\brho - \tr(\bu)) \cdot \bn}_{L^q(\Gamma)} +  \norm{\tr(\bu) \cdot \bn}_{L^q(\Gamma)}, 
		\end{align}
		where the last term is already in the requisite form. 
		For the first term on the right-hand side of~\eqref{est:korn-3}, we estimate the normal trace by the full trace. 
		Again using that  $\mathrm{tr}\colon W^{1,p}(\Omega)^{d}\to L^{q}(\partial\Omega)^{d}$ is continuous, we find that 
		\begin{align}\label{est:korn-4}
			\norm{(\brho - \tr(\bu)) \cdot \bn}_{L^q(\Gamma)} 
			\leq  \norm{\brho - \tr(\bu)}_{L^q(\Gamma)} 
			\lesssim \norm{\brho-\bu}_{W^{1,p}(\Omega)}. 
		\end{align}
		Combining~\eqref{est:korn-2}--\eqref{est:korn-4} and finally applying~\eqref{est:korn-1} gives us
		\begin{align*}
			\norm{\bu}_{W^{1,p}(\Omega)}
			& \leq 	\norm{\bu - \brho}_{W^{1,p}(\Omega)} + 	\norm{\brho}_{W^{1,p}(\Omega)}\\
			&\lesssim 	\norm{\bu - \brho}_{W^{1,p}(\Omega)} + \norm{(\tr(\bu)-\brho) \cdot \bn}_{L^q(\Gamma)} + \norm{\tr(\bu) \cdot \bn}_{L^q(\Gamma)} \\
			&\lesssim  \norm{\bu-\brho}_{W^{1,p}(\Omega)} + \norm{\tr(\bu) \cdot \bn}_{L^q(\Gamma)} \\
			&\lesssim  \norm{\BD \bu}_{L^p(\Omega)} + \norm{\tr(\bu) \cdot \bn}_{L^q(\Gamma)}.
		\end{align*}
  This is~\eqref{eq:kornetto0}, and~\eqref{eq:kornetto} is an immediate consequence of~\eqref{eq:kornetto0}. 
		For the supplementary statement, we use the preceding inequality on each of the finitely many connected components and sum over the latter. This completes the proof. 
	\end{proof}
Different from proofs by contradiction, it is in principle possible to extract bounds on the underlying constants; note that~\eqref{est:korn-1} can be proved constructively. 
Moreover, in~\eqref{eq:kornetto0}, we estimate the full Sobolev norm. 
	Hence, the following Poincar\'{e}-type inequality is an immediate consequence:
	
	\begin{corollary}[Poincaré inequality] 
		Under the conditions of Theorem~\ref{thm:korn-n}, one also has
		\begin{align}\label{eq:poincare}
			\norm{ \bu}_{L^p(\Omega)} 
			&\lesssim \norm{\BD\bu}_{L^p(\Omega)} + \norm{\tr(\bu) \cdot \bn}_{L^r(\Gamma)}\qquad \text{ for any } \bu \in W^{1,p}(\Omega)^{d}.
		\end{align}
	\end{corollary}
	Based on the above proof, it is moreover straightforward to obtain a version for the trace-free symmetric gradient (or deviatoric symmetric gradient) 
	\begin{align*}
		\BD^{\mathrm{tf}}\bv \coloneqq \BD\bv - \frac{1}{d}\mathrm{div}(\bv)\BI,
	\end{align*}
	too, where $\BI \in \mathbb{R}^{d \times d}$ denotes the unit matrix. 
    If $d\geq 3$ and $\Omega\subset\R^{d}$ is connected, then the nullspace of $\BD^{\mathrm{tf}}$ is given by the conformal Killing vectors $\mathcal{K}(\Omega)$, see~\cite{Resetnjak1}. 
	This is a space of polynomials of degree at most $2$, hence finite dimensional. We then have the following result: 
	\begin{corollary}[Korn-type inequality with normal traces II] 
		Let $d\geq 3$. 
		In the situation of Theorem~\ref{thm:korn-n}, suppose moreover that 
  \begin{align}\label{eq:trampolin1}
			\bv \mapsto \norm{\bv \cdot \bn}_{L^q(\Gamma)}\qquad \text{is a norm on the  conformal Killing vectors $\mathcal{K}(\Omega)$}.
		\end{align}
Then we have the estimate 
		\begin{align}\label{eq:tfsym}
			\norm{\bu}_{W^{1,p}(\Omega)} &\lesssim \norm{\BD^{\mathrm{tf}}\bu}_{L^p(\Omega)} + \norm{\tr(\bu) \cdot \bn}_{L^q(\Gamma)}\qquad \text{ for any } \bu \in W^{1,p}(\Omega)^{d}. 
		\end{align}
	\end{corollary}
	\begin{proof}
	Solely based on the aforementioned fact that $\BD^{\mathrm{tf}}$ has a finite dimensional space of polynomials as nullspace, it suffices to note that one the has following Korn-type inequality: 
		For each $\bu\in W^{1,p}(\Omega)^{d}$, there exists $\brho\in\mathcal{K}(\Omega)$ such that 
		\begin{align}\label{eq:uso}
			\|\bu - \brho\|_{W^{1,p}(\Omega)}\lesssim \|\BD^{\mathrm{tf}}\bu\|_{L^{p}(\Omega)}. 
		\end{align}
  See, e.g., \cite{Resetnjak1,DieningGmeineder2024} for this inequality, even subject to more general assumptions on $\Omega$ than imposed here. Without knowing the precise structure of the elements of $\mathcal{K}(\Omega)$, the mere fact that $\dim(\mathcal{K}(\Omega))<\infty$ allows to replace estimate~\eqref{est:korn-1} by~\eqref{eq:uso}. 
  By~\eqref{eq:trampolin1}, we also have the analogue of~\eqref{est:korn-3} for $\brho\in\mathcal{K}(\Omega)$,  and then may argue as above to conclude~\eqref{eq:tfsym}. 
		This completes the proof. 
	\end{proof}
 For the preceding proof, the assumption $d\geq 3$ cannot be avoided; in $d=2$ dimensions, the nullspace of $\BD^{\mathrm{tf}}$ is isomorphic to the holomorphic functions and hence is not of finite dimension, see~\cite{Resetnjak1}. Moreover, in this case, trace estimates as required for the right-hand side of~\eqref{eq:tfsym} to be meaningful are impossible; see~\cite{BDG.2020,DieningGmeineder2024} for more on this matter.   

 \medskip
	
	We now study the geometric Assumption~\ref{assump:dom} in more detail, and first address its optimality: 
	
	\begin{remark}[optimality of Assumption~\ref{assump:dom}]  
		Assumption~\ref{assump:dom} is in fact equivalent to the validity of~\eqref{eq:kornetto0}. 
		To see this, note that $\bv\mapsto \norm{\bv\cdot\bn}_{L^{q}(\Gamma)}$ can only fail to be a norm on $\mathcal{R}(\Omega)$ provided that there exists $\bv\in\mathcal{R}(\Omega)\setminus\{\b0\}$ such that $\norm{\bv\cdot\bn}_{L^{q}(\Gamma)}=0$. 
		Inserting this map $\bv$ into~\eqref{eq:kornetto0} consequently yields the contradictory $\bv\equiv \b0$. 	
	\end{remark} 
	
	As discussed in the following example, Assumption~\ref{assump:dom} is not satisfied for all domains: 
	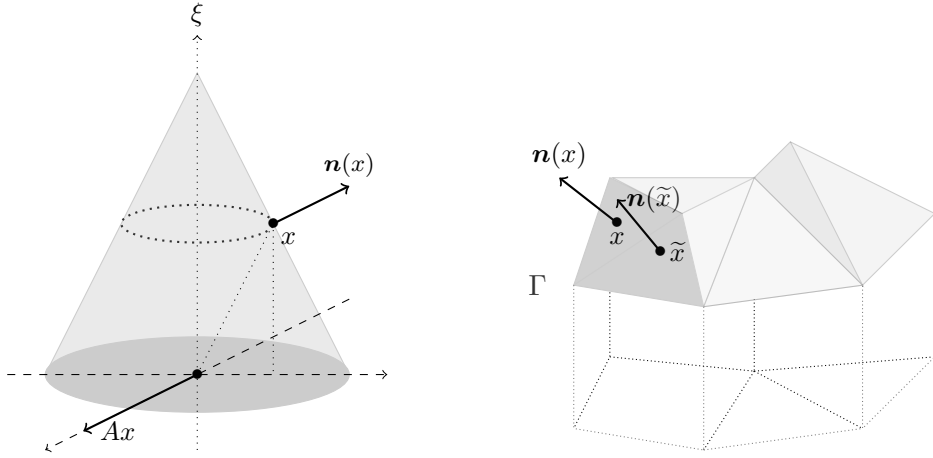
\begin{figure}[t]
		\begin{tikzpicture}
			\draw[-,fill=black!40!white,opacity=.2] (2,0) -- (0,4) -- (-2,0);
			\draw[black!20!white,fill=black!20!white] (0,0) ellipse (2cm and 0.5cm);
			\draw[black!80!white,dotted,thick] (0,2) ellipse (1cm and 0.25cm);
			\draw[dotted,->] (0,-1) -- (0,4.5);
			\node[above] at (0,4.5) {$\xi$};
			\node at (0,0) {\small\textbullet};
			\draw[dotted] (0,0) -- (1,2);
			\draw[dashed,->] (-2.5,0) -- (2.5,0);
			\node at (1,2) {\small\textbullet};
			\node[right,below] at (1.2,2) {$x$};
			\draw[dotted] (1,2) -- (1,0);
			\draw[->,thick] (1,2) -- (2,2.5);
			\node[above] at (2,2.5) {\small$\bn(x)$};
			\draw[dashed,->] (2,1) -- (-2,-1);
			\draw[->,thick] (0,0) -- (-1.5,-0.75);
			\node[right] at (-1.4,-0.75) {$Ax$};
		\end{tikzpicture}
		\hspace{1.5cm}
		\begin{tikzpicture}[scale=1.9]
  \node at (2.3,1.435) {\small\textbullet};
   \draw[->,thick] (2.3,1.435) -- (1.9,1.75);
   \node[above] at (1.9,1.75) {$\bn(x)$};
   \node[below] at (2.3,1.435) {$x$};
    \node at (2.6,1.235) {\small\textbullet};
    \node[right] at (2.6,1.235) {$\widetilde{x}$};
    \draw[->,thick] (2.6,1.235) -- (2.3,1.6);
    \node[right] at (2.3,1.6) {$\bn(\widetilde{x})$};
			\draw[-,fill=black!90!white,opacity=0.2] (2,1) -- (2.75,1.5) -- (2.25,1.75) -- (2,1) -- (2.9,0.85) -- (2.75,1.5);
            \node[black!80!white] at (1.75,1) {\large $\Gamma$};
			\draw[-,fill=black!20!white,opacity=0.2] (2.75,1.5) -- (3.25,1.75) -- (2.27,1.75) -- (2.75,1.5);
			\draw[-,fill=black!40!white,opacity=0.2] (2.9,0.85) -- (3.25,1.75) -- (4,1) -- (3.5,2) -- (3.25,1.75);
			\draw[-,fill=black!20!white,opacity=0.2] (2.9,0.85) -- (4,1);
			\draw[-,fill=black!20!white,opacity=0.2] (4,1) -- (4.5,1.5) -- (3.5,2);
			\draw[-,densely dotted,black!50!white] (2,1) -- (2,0);
			\draw[-,densely dotted,black!50!white] (2.9,0.85) -- (2.9,-0.15) -- (2,0);
			\draw[-,densely dotted,black!50!white] (2.9,-0.15) -- (4,0) -- (4,1);
			\draw[densely dotted,-,black!50!white]  (4.5,1.5) -- (4.5,0.5) -- (4,0);
			\draw[-,black!40!white,fill=black!10!white,opacity=0.3] (2.9,0.85) -- (3.25,1.75) -- (4,1) -- (2.9,0.85);
			\draw[-,fill=black!10!white,black!10!white,opacity=0.3] (2.9,0.85) -- (3.25,1.75) -- (2.75,1.5) -- (2.9,0.85);
   \draw[densely dotted] (2.25,0.95) -- (2.25,0.5) -- (2,0);
   \draw[densely dotted] (2.25,0.5) -- (3.25,0.4) -- (3.25,0.9);
     \draw[densely dotted] (3.25,0.4) -- (4.5,0.5);
      \draw[densely dotted] (3.25,0.4) -- (2.9,-0.15);
      \draw[densely dotted] (3.25,0.4) -- (4,0);
		\end{tikzpicture}
		\caption{Assumption~\ref{assump:dom} and its geometric impact. 
			Left-hand figure: An axisymmetric cone $\Omega$, for which $Ax\bot \bn(x)$ holds for any $x\in\partial\Omega$, where $Ax\coloneqq x\times\xi$. Note that $\bn(x)$ is always contained in the plane spanned by $\xi$ and $x$, and $Ax$ is orthogonal to this plane; see Example~\ref{ex:failureAssump}. Right-hand figure: Polyhedral domains as in the overall setting of the paper, and Corollary~\ref{cor:polyhedral}. If $\Gamma\subset\partial\Omega$ is polyhedral and contains two non-collinear normals $\bn(x),\bn(\widetilde{x})$ as indicated for $x$ and $\widetilde{x}$, Assumption~\ref{assump:dom} is satisfied.}\label{fig:Korno}
	\end{figure}
	\begin{example}[On Assumption \ref{assump:dom}]\label{ex:failureAssump}
		As observed in~\cite{DesvillettesVillani2002} (also see~\cite{BP.2016,Bauer2016,Dominguez2022}), the special case $\|\bu\|_{W^{1,p}(\Omega)}\lesssim\|\BD u\|_{L^{p}(\Omega)}$ for all $\bu\in W^{1,p}(\Omega)^{d}$ with $\bu\cdot\bn=0$ along $\partial\Omega$ rules out the axisymmetry of $\partial\Omega$. In the case considered here, Assumption~\ref{assump:dom} particularly implies that $\Gamma\subset\partial\Omega$ must not be contained in an axisymmetric $(d-1)$-dimensional manifold. We consider the following instances: 
		\begin{enumerate}[label = (\roman*)]
			\item Let $B=B_{r}(x_{0})\subset\R^{d}$ be an open ball of radius $r>0$ and centered at $x_{0}\in\R^{d}$. Then Assumption~\ref{assump:dom} is never satisfied: Indeed, in this situation $\brho(x)=A(x-x_{0})$ with $A=-A^{\top}$ belongs to $\mathcal{R}(\Omega)$ and satisfies $\brho(x)\cdot\bn(x)=A(x-x_{0})\cdot \frac{x-x_{0}}{|x-x_{0}|}=0$ for all $x\in\partial B_{r}(x_{0})$. 
			\item Following~\cite[\S 5]{DesvillettesVillani2002}, if $\Gamma\subset\partial\Omega$ is such that $\mathcal{H}^{d-1}(\Gamma)>0$ and $\Gamma$ is contained in an axisymmetric domain with axis of symmetry $\R\xi$, then Assumption~\ref{assump:dom} is not satisfied. Indeed, e.g., in $d=3$ dimensions, we may consider 
			\begin{align*}
				\brho(x)\coloneqq Ax = x\times \xi,\qquad \text{so that, in particular,}\;\;\;A=\left(\begin{matrix}  0 & \xi_{3} & -\xi_{2} \\ -\xi_{3} & 0 & \xi_{1} \\ \xi_{2} & -\xi_{1} & 0 \end{matrix} \right), 
			\end{align*}
			i.e., $A\in\R_{\mathrm{skew}}^{3\times 3}$,
			and this map satisfies $\brho(x)\bot\bn(x)$ for $\mathcal{H}^{2}$-a.e.~$x\in\Gamma$; 
			see Figure~\ref{fig:Korno} for the case of parts $\Gamma$ contained in cones.
		\end{enumerate}
	\end{example}
	As a key point for the present	paper, Assumption~\ref{assump:dom} however is always fulfilled for polyhedral domains provided that a simple non-collinearity condition on the normals on $\Gamma$ is satisfied: 
	
	\begin{corollary}[Assumption~\ref{assump:dom} and polyhedral domains] \label{cor:polyhedral}
		Let $\Omega\subset\R^{d}$ be open and bounded with Lipschitz boundary $\partial\Omega$. 
		Moreover, suppose that $\partial\Omega = \Gamma_{1}\cup \Gamma_{2}\cup R$, where 
		\begin{enumerate}
			\item\label{item:pinegrove1} $\Gamma_{1},\Gamma_{2}$ are $\mathcal{H}^{d-1}$-measurable with  $\mathcal{H}^{d-1}(\Gamma_{1}),\mathcal{H}^{d-1}(\Gamma_{2})>0$, and 
			\item\label{item:pinegrove2} there exist affine hyperplanes $H_{1}=\xi_{1}+\{\nu_{1}\}^{\bot}$, $H_{2}=\xi_{2}+\{\nu_{2}\}^{\bot}$ with $\nu_{1}\notin\R \nu_{2}$, such that $\Gamma_{1}\subset H_{1}$, $\Gamma_{2}\subset H_{2}$. 
		\end{enumerate}
		Then $\Gamma=\Gamma_{1}\cup\Gamma_{2}$ satisfies Assumption~\ref{assump:dom}. 
	\end{corollary}
	\begin{proof}
		Let $A\in\R^{d\times d}$ be skew-symmetric,  $b\in\R^{d}$, and suppose that $\pi(x)\coloneqq Ax+b$ satisfies $\pi=0$ on $\Gamma$, whereby $\pi=0$ on $H_{1}\cup H_{2}$. 
		Let $\mathbf{x}_{1},...,\mathbf{x}_{d-1}$ be an orthonormal basis of $\{\nu_{1}\}^{\bot}$, and let $\mathbf{y}_{2},...,\mathbf{y}_{d}$ be an orthonormal basis of $\{\nu_{2}\}^{\bot}$. 
		In consequence, we have 
		\begin{align}\label{eq:pino}
			\sum_{j=1}^{d-1}\lambda_{j}A\mathbf{x}_{j} = -A\xi_{1} - b \;\;\;\text{and}\;\;\;\sum_{j=2}^{d}\mu_{j}A\mathbf{y}_{j} = -A\xi_{2} - b\qquad\text{for all}\;\lambda_{1},...,\lambda_{d-1},\mu_{2},...,\mu_{d}\in\R. 
		\end{align}
		Based on~\ref{item:pinegrove2} and potentially relabelling the indices, we may assume that $\{\mathbf{x}_{1},...,\mathbf{x}_{d-1},\mathbf{y}_{d}\}$ is a basis of $\R^{d}$. Setting $\mu_{2}=...=\mu_{d-1}=0$ and $\lambda_{d}\coloneqq\mu_{d}$, we have 
		\begin{align*}
			\Big(\sum_{j=1}^{d-1}\lambda_{j}A\mathbf{x}_{j}\Big) + \lambda_{d}A\mathbf{y}_{d} = -A\xi_{1}-A\xi_{2}-2b\eqqcolon \mathbf{c}\qquad \text{for all}\;\lambda_{1},...,\lambda_{d}\in\R. 
		\end{align*}
		The choice $\mathbf{c}\neq \b0$ is impossible by linearity of $A$. 
		If $\mathbf{c}=\b0$, then $A=0$ and so~\eqref{eq:pino} gives us $\pi=0$, completing the proof. 
	\end{proof}
	
	To conclude, we note that the above results motivate the study of related inequalities involving more general differential operators, conditions in the spirit of~\eqref{eq:trampolin} and their geometric interplay, both in the so-called compatible and incompatible scenarios (\`{a} la~\cite{GmeinederSpector,ContiGmeineder2022,GmeinederLewintanNeff2022,GmeinederLewintanNeff2023}). 
	However, to keep our exposition at a reasonable length, this will be deferred to future work.

	\section{Approximation and discretisation}
	\label{sec:prel}

	In this section, we gather  various results to set up the Nitsche scheme and to show convergence in Section~\ref{sec:un-NS}. 
	In Section~\ref{sec:graphapprox} we address the approximation of implicit constitutive relations. 
	Then, in Section~\ref{sec:fct-spaces} we introduce the spaces required for the weak formulation and collect background facts on traces and embeddings.
	Section~\ref{sec:prelim-fem} introduces the finite element discretisation, and
	Section~\ref{sec:prel-timediscr} gathers auxiliary material on time discretisation and compactness in time. 
	
	\subsection{Approximation of monotone graphs}\label{sec:graphapprox}
	To obtain existence of numerical approximations we need to work with explicit relations. 
	For this reason, we consider an approximation of the zero set $\mathcal{A}$ of $\gbd$, that is the graph of a single-valued mapping $\bv \mapsto \bs$. 
	Let us collect general assumptions on such approximations, which we verify for specific examples in the sequel. 
	
	\begin{assumption}[regularisation]
		\label{assump:gbd-reg}
		Assume that for $\gbd$ satisfying Assumption~\ref{assump:gbd-mon} with $r \in [1,\infty)$, and for any $\epsilon \in (0,1)$ there is a  continuous functions $\Seps\colon \setR^d \to \setR^d$ such that the following properties hold for any $\epsilon \in (0,1)$: 
		\begin{enumerate}[label = (a\arabic*)]
			\item \label{itm:gbd-eps-0}
			$\Seps(\b0) = \b0$;
			\item \label{itm:gbd-eps-q} each $\Seps$ satisfies $r$-growth and coercivity in the sense that there is a constant $c>0$ independently of $\epsilon>0$ such that for any $\bv \in \setR^d$ one has
			\begin{alignat*}{3}
				\Srel_{\epsilon}(\bv)  \cdot \bv &\geq c (\abs{	\Srel_{\epsilon}(\bv) }^{r'} + \abs{\bv}^r - 1)&& \qquad \qquad &\text{ if } r >1, \\
				\Srel_{\epsilon}(\bv)  \cdot \bv &\geq c ( \abs{\bv} - 1) \quad   \text{ and }\;\; & &\abs{\Srel_{\epsilon}(\bv) } \leq c \qquad 
				&\text{ if } r = 1;
			\end{alignat*}
			\item \label{itm:gbd-eps-mon}
			$\Seps$ is monotone, i.e., 
			\begin{align*}
				(\Seps(\bv_1) - \Seps(\bv_2)) \cdot (\bv_1 - \bv_2) \geq 0 \qquad \text{ for all } \bv_1, \bv_2 \in \setR^d;
			\end{align*} 
			\item \label{itm:gbd-eps-approx} 
            Let $M \subset \setR^n$ be an open set, and let $A$ be defined as in~\eqref{def:A}. 
			For any~$\bv \in L^{r}(M)^d$ and any $\bs\in L^{r'}(M)^d$ such that $\gbd(\bs,\bv) = \b0$  a.e.~in $M$, i.e., 
			$(\bv,\bs) \in A$, there exists a sequence $(\bv_{\epsilon})_{\epsilon>0}$ such that 
			\begin{alignat*}{3}
				\bv_{\epsilon} &\to \bv \qquad &&\text{ strongly in } L^{r}(M)^d,\\
				\Seps(\bv_{\epsilon}) &\to \bs \quad &&\text{ strongly  in }  L^{r'}(M)^d,
			\end{alignat*}
			as $\epsilon \to 0$. 
		\end{enumerate} 
	\end{assumption}
	
	\begin{example}[generalised Yosida approximation]\label{ex:gen-Yosida-r} \hfill 
		\begin{enumerate}[label = (\alph*)]
			\item ($r \in (1,\infty)$) \label{itm:gen-Yosida-r}
			For $\gbd$ as in Assumption~\ref{assump:gbd-mon} with some $r \in (1, \infty)$ we define the continuous function $\gbd_{\epsilon} \colon \setR^{d \times d} \to \setR$ via 
			\begin{align*}
				\gbd_{\epsilon} (\bs, \bv) \coloneqq \gbd\left(\bs,\bv - \epsilon \abs{\bs}^{r'-2}
				\bs\right) \quad \text{ for } \bs, \bv \in \setR^d.
			\end{align*}
			Note that for $A \subset L^{r}(M)^d\times L^{r'}(M)^d$ as defined  in~\eqref{def:A} in terms of $\gbd$, the corresponding ${A}_\epsilon \subset L^{r}(M)^d\times L^{r'}(M)^d$ defined analogously in terms of $\gbd_{\epsilon}$ is the generalised Yosida approximation of ${A}$, as  in~\cite[Sec.~4]{Francfort2004}. 
			The mapping~$\Seps \colon \setR^d \toto \setR^d$, defined by 
			\begin{align*}
				\gbd_\epsilon(\bs, \bv) = \b0 \quad \Leftrightarrow : \quad   \bs \in  \Seps(\bv) \qquad \text{ for } \bs, \bv \in \setR^d, 
			\end{align*}
			is a single-valued map which satisfies~\ref{itm:gbd-eps-0}--\ref{itm:gbd-eps-mon}, see~\cite[Lem.~3.29]{Tscherpel2018}. 
			To verify~\ref{itm:gbd-eps-approx} let $\bv \in L^{r}(M)^d$ and $\bs \in L^{r'}(M)^d$ be such that $\gbd(\bs,\bv) = \b0$ a.e.~in $M$. 
			By definition, for $\bv_{\epsilon} \coloneqq \bv + \epsilon \abs{\bs}^{r'-2} \bs $,  we have that 
			\begin{align*}
				\gbd_{\epsilon}(\bs,\bv_{\epsilon}) = \b0 \quad \text{ a.e.~in } M. 
			\end{align*}
			Clearly, we also obtain 
			\begin{alignat*}{3}
				\bv_{\epsilon} \coloneqq \bv + \epsilon \abs{\bs}^{\frac{1}{r-1}-1} \bs
				&\to \bv \quad &&\text{ strongly in } L^r(M)^d, \\
				\Seps(\bv_{\epsilon}) = \bs &\to \bs  \quad &&\text{ strongly in } L^{r'}(M)^d
			\end{alignat*}
			as $\epsilon \to 0$. 
			This proves~\ref{itm:gbd-eps-approx} for the generalised Yosida approximation. 
			\item ($r = 1$) 
			Similarly, for $\gbd$ as in Assumption~\ref{assump:gbd-mon} for $r= 1$ we define the continuous function $\gbd_{\epsilon} \colon \setR^{d \times d} \to \setR$ via 
			\begin{align*}
				\gbd_{\epsilon} (\bs, \bv) \coloneqq \gbd\left(\bs,\bv - \epsilon \bs\right) \quad \text{ for } \bs, \bv \in \setR^d.
			\end{align*}
			As before, one may check, that the set-valued map $\Seps \colon \setR^d \toto \setR^d$, defined by 
			\begin{align*}
				\gbd_\epsilon(\bs, \bv) = \b0 \quad \Leftrightarrow : \quad   \bs \in  \Seps(\bv) \qquad \text{ for } \bs,\bv \in \setR^d,
			\end{align*}
			is in fact single-valued and it satisfies~\ref{itm:gbd-eps-0}--\ref{itm:gbd-eps-mon}.  
			
			To verify also~\ref{itm:gbd-eps-approx} let $\bs \in L^\infty(M)^d$ and $\bv \in L^1(M)$ be such that $\gbd(\bs,\bv) = \b0$ a.e.~on $M$. 
			By definition we have $\gbd_\epsilon(\bs,\bv+ \epsilon \bs) = \b0$, and hence we set $\bs_{\epsilon} = \bs$ and $\bv_{\epsilon} \coloneqq \bv + \epsilon \bs$. 
			Then, both $\bs_{\epsilon} \to \bs$ in $L^\infty(M)^d$ as well as $\bv_\epsilon \to \bv$ converges  strongly in $L^1(M)^d$, as $\epsilon \to \b0$. 
			This proves that~\ref{itm:gbd-eps-approx} holds. 
		\end{enumerate}
	\end{example}
	
	\begin{example}[regularised Tresca relation]\label{ex:reg-rel}
		The Tresca slip boundary condition, as in Example~\ref{ex:Tresca} for $\gamma_\star = 0$, is an example of an implicit relation with $r=1$. 
		A simple regularisation $\Seps\in C^\infty(\R^d)^d$, employed in the computational experiments below, is the following:\begin{equation}\label{eq:tresca_regularised}
			\Seps(\bv) \coloneqq \mu_\star \frac{\bv}{\sqrt{|\bv|^2 + \varepsilon^2}}
			\quad \text{ for } \varepsilon \in (0,1).
		\end{equation} 
		This relation clearly satisfies~\ref{itm:gbd-eps-0}, and one can check that also~\ref{itm:gbd-eps-q} and~\ref{itm:gbd-eps-mon} hold.  
		The proof of \ref{itm:gbd-eps-approx} is contained in Appendix~\ref{app:ex-reg-Tresca}.
	\end{example}
	
	In the convergence proof for $r >2$ we require a convergence lemma of Minty-type for weakly converging {sequences} in corresponding dual spaces. 
	The special feature is, that the weakly converging sequences are related by approximate relations (as in Assumption~\ref{assump:gbd-reg}) rather than by a single maximal monotone graph. 
	A suitable Minty-type lemma for the generalised Yosida approximation is due to the last author~\cite[Lem.~3.31]{Tscherpel2018}.  
	Alternatively, one may use a regularisation of $\gbd$ with a $2$-coercive function, as in~\cite[Lem.~4.1]{Bulicek2021}, for which an analogous convergence result is available. 
	The benefit of such a result is that the regularisation limit may be taken simultaneously with the remaining limits, such as, e.g.,~discretisation limits. 
	
	We begin by collecting a convergence lemma in the simpler case where the arguments of the nonmonotone relation converge strongly. 
	This will be applied in the case of $r\leq 2$ in  Section~\ref{sec:un-NS}, and to keep our presentation self-contained, we include a proof: 
	
	\begin{lemma}[convergence lemma]\label{lem:conv-simple}
		Let $M \subset \setR^d$ be measurable, and suppose that $\gbd$ satisfies Assumption~\ref{assump:gbd-mon} with some $r \in [1,\infty)$. 
		Given a regularisation as in Assumption~\ref{assump:gbd-reg}~\ref{itm:gbd-eps-0}--\ref{itm:gbd-eps-approx}, let $(\bw_{\varepsilon})_{\varepsilon >0}$ and $(\Seps(\bw_{\varepsilon}))_{\varepsilon >0}$ be sequences such that 
		\begin{align}\label{conv:simple-1}
			\bw_{\varepsilon} \to  \bw \qquad & \text{ strongly in } L^{r}(M)^d,\\
			\label{conv:simple-2}
			\Seps(\bw_{\varepsilon}) \wsconv  \bsigma \qquad & \text{ weakly$^*$ in } L^{r'}(M)^d
		\end{align}
		as $\varepsilon \to 0$. 
		Then we have $\gbd(\bsigma,\bw) = \b0$ a.e.~in $M$. 
		
	\end{lemma}
	\begin{proof}
		Let $\bv \in L^r(M)^d$ and $\bs \in L^{r'}(M)$ be arbitrary functions such that $\gbd(\bs,\bv)= \b0 $ a.e.~in $M$. 
		Then, by Assumption~\ref{assump:gbd-reg}~\ref{itm:gbd-eps-approx}, there exists a sequence $(\bv_{\varepsilon})_{\varepsilon>0}$ such that
		\begin{alignat}{2}\label{conv:simple-3}
			\bv_{\varepsilon} &\to \bv \quad &&\text{ strongly in } L^{r}(M)^d,\\ \label{conv:simple-4}
			\Seps(\bv_{\varepsilon}) &\wsconv \bs && \text{weakly$^*$ in } L^{r'}(M)^d, 
		\end{alignat}
		as $\varepsilon \to 0$. 
		By the monotonicity of $\Seps$ thanks to Assumption~\ref{assump:gbd-reg}~\ref{itm:gbd-eps-mon} and the convergences displayed in~\eqref{conv:simple-1}--\eqref{conv:simple-4} we obtain 
		\begin{align*}
			0&  \leq \limsup_{\varepsilon \to 0} \int_{M} (\Seps(\bw_{\varepsilon}) - \Seps(\bv_{\varepsilon})) \cdot (\bw_{\varepsilon} - \bv_{\varepsilon})   \dx = 
			\int_{M}  (\bsigma - \bs) \cdot (\bw - \bv)  \dx.
		\end{align*}
		By the maximal monotonicity of $A$ from Assumption~\ref{assump:gbd-mon}, see~\eqref{def:A}, this shows that $\gbd(\bsigma,\bw) = \b0$ a.e.~in $M$. 
		This completes the proof.  
	\end{proof}
	
	If only weak convergence in $L^r$ is available,  a Minty-type result is essential.  
	This will be the case for a skew-symmetric variant of the Nitsche method with $r >2$ in Section~\ref{sec:extensions_rgeq2} below, where we make use of the following variant: 
	
	\begin{lemma}[Minty-type convergence lemma]\label{lem:minty}
		
		Let $M \subset \setR^d$ {be measurable}, and suppose that $\gbd$ satisfies Assumption~\ref{assump:gbd-mon} with some $r \in [1,\infty)$. 
		Given a regularisation as in Assumption~\ref{assump:gbd-reg}~\ref{itm:gbd-eps-0}--\ref{itm:gbd-eps-approx}, let  $(\bw_{\varepsilon})_{\varepsilon >0}$ and $(\Seps(\bw_{\varepsilon}))_{\varepsilon >0}$ be sequences such that 
		\begin{align}\label{conv:minty-1}
			\bw_{\varepsilon} \wconv  \bw \qquad & \text{ weakly in } L^{r}(M)^d,\\
			\label{conv:minty-2}
			\Seps(\bw_{\varepsilon}) \wsconv  \bsigma \qquad & \text{ weakly$^*$ in } L^{r'}(M)^d
		\end{align}
		as $\varepsilon \to 0$, 
		and assume additionally that 
		\begin{align}\label{est:minty-lsc}
			\limsup_{\varepsilon \to 0}	\int_{M}  \Seps(\bw_{\varepsilon}) \cdot \bw_{\varepsilon}  \dx \leq\int_{M} \bsigma \cdot \bw  \dx. 
		\end{align}
		Then we have $\gbd(\bsigma,\bw) = \b0$ a.e.~in $M$. 
		
	\end{lemma}
	\begin{proof}
		Let $\bv \in L^r(M)^d$ and $\bs \in L^{r'}(M)^{d}$ be arbitrary with $\gbd(\bs,\bv)= \b0 $ a.e.~in $M$. 
		Then, by 
        Assumption~\ref{assump:gbd-reg}~\ref{itm:gbd-eps-approx}, there exists a sequence $(\bv_{\varepsilon})_{\varepsilon>0}$ such that
		\begin{alignat}{2}\label{conv:minty-3}
			\bv_{\varepsilon} &\to \bv \quad &&\text{ strongly in } L^{r}(M)^d,\\ \label{conv:minty-4}
			\Seps(\bv_{\varepsilon}) &\to \bs && \text{strongly in } L^{r'}(M)^d
		\end{alignat}
		as $\varepsilon \to 0$. 
		By monotonicity of $\Seps$ due to Assumption~\ref{assump:gbd-reg}~\ref{itm:gbd-eps-mon} we obtain 
		\begin{align*}
			0&  \leq \limsup_{\varepsilon \to 0} \int_{M} (\Seps(\bw_{\varepsilon}) - \Seps(\bv_{\varepsilon})) \cdot (\bw_{\varepsilon} - \bv_{\varepsilon})  \dx \\
			&= \limsup_{\varepsilon \to 0}  
			\left( 
			\int_{M}  \Seps(\bw_{\varepsilon}) \cdot \bw_{\varepsilon}  \dx 
			+ 	\int_{M}  \Seps(\bv_{\varepsilon}) \cdot  \bv_{\varepsilon}  \dx 
			- 	\int_{M}   \Seps(\bv_{\varepsilon}) \cdot \bw_{\varepsilon} \dx 
			- 	\int_{M}  \Seps(\bw_{\varepsilon}) \cdot  \bv_{\varepsilon} \dx 
			\right).
		\end{align*}
		On the first term we shall use~\eqref{est:minty-lsc}. 
		On each of the remaining terms we shall employ the fact that by~\eqref{conv:minty-1}, \eqref{conv:minty-2} and~\eqref{conv:minty-3}, \eqref{conv:minty-4} at least one of the factors converges strongly, and the other factor converges weakly. 
		Thus, we arrive at 
		\begin{align*}
			0  \leq \int_{M}(\bsigma - \bs) \cdot  (\bw - \bv) \dx.
		\end{align*}
		Recall that  $\bv\in L^r(M)^d$ and $\bs \in L^{r'}(M)^d$ were arbitrary with $\gbd( \bs,\bv) = \b0$ a.e.~in $M$.   
		Because $ A \subset L^r(M)^d \times L^{r'}(M)^d$ as in~\eqref{def:A} is maximal monotone by Assumption~\ref{assump:gbd-mon}~\ref{itm:gbd-mon-lp}, this allows us to conclude that $\gbd(\bsigma,\bw) = \b0$ a.e.~in $M$. 
		The proof is complete. 
	\end{proof}

	\subsection{Function spaces and traces}\label{sec:fct-spaces} 
 We now give a detailed account of our function space setup, with a focus on traces and spaces of divergence-free functions as required in the sequel. 
	\subsubsection{Trace estimates}
	
	We collect some background facts on the trace operator restricted to $\Gamma \coloneqq \partial\Omega$, where $\Omega\subset\R^{d}$ is open and bounded domain with Lipschitz boundary. 
	To this end, we define 
	\begin{align}\label{def:2-sharp}
		2^\sharp \coloneqq 
		\frac{2d-2}{d-2} =  \begin{cases}
			4 \quad &\text{ if } d = 3,\\
			\infty 	\quad &\text{ if } d = 2.
		\end{cases}
	\end{align}
	Since $\Gamma$ is bounded and Lipschitz, we have the embedding  $H^{\frac{1}{2}}(\Gamma) \hookrightarrow L^{p}(\Gamma)$ for all $p \in [1, \infty)$ with $p \leq 2^{\sharp}$. 
	Combining this observation with the usual boundedness of the trace operator $\mathrm{tr}\colon H^{1}(\Omega)\to H^{\frac{1}{2}}(\Gamma)$, we arrive at the trace inequality 
	\begin{align}\label{est:trace}
		\norm{\tr(v)}_{L^{p}(\Gamma)} \lesssim \norm{v}_{H^1(\Omega)} \qquad \text{ for all } v \in H^1(\Omega),
	\end{align}
	for all $p< 2^{\sharp}$ (and for $p = 2^\sharp$, if $2^\sharp < \infty$). 
	
	Due to this trace inequality, we will mostly focus  on boundary conditions with $r \in [1,2^\sharp)$ in this work.
	
\subsubsection{Function space setup}\label{sec:prel-fctsp}
To obtain compactness on the trace terms in the dynamic case $\beta >0$, we need a special function space setup, similar to the one in~\cite[Sec.~3]{Abbatiello2021}, see also~\cite{Kovacs2017}. 
For divergence-free functions we proceed as in~\cite{Abbatiello2021}. 
For the numerical approximation, however, we also need the spaces without the divergence constraint and the zero normal trace conditions.  
This is because most finite element spaces are not conforming with respect to the divergence constraint and in a Nitsche method the impermeability is not imposed directly but by penalisation. 
Let us define 
\begin{align}
	\mathcal{W} 
	&\coloneqq \{(\bv, \btheta) \in C^{0,1}(\overline\Omega)^{d} \times C^{0,1}(\Gamma)^d\colon  \bv = \btheta \text{ on } \Gamma\},\\
	\mathcal{W}_{\divergence} 
	&\coloneqq \{(\bv, \btheta) \in C^{0,1}(\overline\Omega)^{d} \times C^{0,1}(\Gamma)^d\colon \divergence \bv = 0\; \text{ in } \Omega,\,\bv \cdot \bn = 0 \text{ and } \bv = \btheta \text{ on } \Gamma\}.
\end{align} 
For $(\bv, \btheta), (\tilde \bv, \tilde \btheta) \in \mathcal{W}$ let us define the inner product and the corresponding norm 
\begin{align}\label{def:ip-B}
	((\bv,\btheta),(\tilde \bv, \tilde \btheta))_B &\coloneqq \int_{\Omega} \bv \cdot \tilde \bv \dx +  \beta \int_{\Gamma} \btheta \cdot \tilde \btheta \ds,\\
	\label{def:norm-B}
	\norm{(\bv,\btheta)}_{B} &\coloneqq  \left(\norm{\bv}_{L^2(\Omega)}^2 +  \beta \norm{\btheta}_{L^2(\Gamma)}^2 \right)^{1/2},
\end{align}
and the norm 
\begin{align}
	\norm{(\bv,\btheta)}_{W} &\coloneqq  \norm{\bv}_{H^1(\Omega)} + \norm{\btheta}_{L^2(\Gamma)}. 
	\label{def:norm-W}
\end{align}
Then, we may introduce the spaces $W, \Wdiv$ and $B$, $\Bdiv$ as the closures 
\begin{alignat}{5} \label{def:sp-W}
	W
	&\coloneqq \overline{\mathcal{W}}^{\norm{\cdot}_{W}}, \qquad
	&\Wdiv 
	&\coloneqq \overline{\mathcal{\Wdiv}}^{\norm{\cdot}_{W}},
	\\
	\label{def:sp-B}
	B 
	&\coloneqq \overline{\mathcal{W}}^{\norm{\cdot}_{B}},
	\qquad
	&\Bdiv 
	&\coloneqq \overline{\mathcal{\Wdiv}}^{\norm{\cdot}_{B}}.
\end{alignat}
Note that all those spaces  are reflexive and separable. 

By virtue of the trace inequality~\eqref{est:trace} with $r = 2 < 2^\sharp$, for any $(\bv,\btheta)\in W$ one has that $\btheta = \tr(\bv)$. 
In particular, this means that $\norm{(\bv,\btheta)}_{W}$ is equivalent to $\norm{\bv}_{H^1(\Omega)}$, and hence we obtain  that 
\begin{align}\label{est:W-norm}
	\norm{(\bv,\btheta)}_{W}
	= \norm{(\bv,\tr(\bv))}_{W}
	\lesssim \norm{\bv}_{H^1(\Omega)} 
	\qquad \text{ for any } \bv \in H^1(\Omega)^d. 
\end{align}
For this reason, in the following we shall not distinguish between $(\bv,\tr(\bv))\in W$ and $\bv \in H^1(\Omega)$, and with slight misuse of notation we write, e.g., ~$\bv \in W$ and $(\bv,\cdot)_B$. 
Note that for $(\bv, \btheta)\in B$ in general there is no relation between $\bv$ and $\btheta$, since $\bv$ may not even have a trace. 
By the Sobolev embedding we have the following continuous, dense and compact embeddings
\begin{align}\label{eq:WB-emb}
	W \hookrightarrow\hookrightarrow B
	\qquad \text{ and } \quad 
	\Wdiv  \hookrightarrow\hookrightarrow \Bdiv. 
\end{align}
Consequently, also the embedding $(\Bdiv)' \hookrightarrow \hookrightarrow (\Wdiv)'$ is dense. 
We identify $\Bdiv \equiv \Bdiv'$, and denote the corresponding inner product by $(\cdot,\cdot)_{\Bdiv}$. 
This means in particular, that 
\begin{align}\label{eq:B-Bdiv}
	((\bv, \btheta),(\tilde \bv, \tilde \btheta))_{B}  = 	((\bv, \btheta),(\tilde \bv, \tilde \btheta))_{\Bdiv} \quad \text{ for any } (\bv, \btheta),(\tilde \bv, \tilde \btheta) \in \Bdiv. 
\end{align}
By the dense embeddings we have the Gelfand triplet 
\begin{align}\label{def:Gelfand}
	\Wdiv \hookrightarrow \Bdiv \equiv (\Bdiv)'  \hookrightarrow (\Wdiv)'. 
\end{align}
In particular, the duality pairing between $(\Wdiv)'$ and $\Wdiv$, denoted by $\skp{\cdot}{\cdot}_{\Wdiv}$  can be defined by extension of the inner product 
$(\cdot,\cdot)_B$, see~\cite[Sec.~3]{Abbatiello2021}, and we have 
\begin{align}
	\skp{ \bpsi}{ \bphi}_{\Wdiv} 
	= ( \bpsi, \bphi)_{\Bdiv} \quad \text{ for any }  \bpsi \in \Bdiv,\,  \bphi \in \Wdiv.
\end{align}

In case $\beta = 0$, there are no extra trace terms in the equation. 
To unify the notation in this case,  we denote 
\begin{alignat}{3}\label{def:sp-B-v}
	B &\coloneqq L^2(\Omega)^d, \qquad &&B_{\diver} &\coloneqq L^2_{\diver}(\Omega)^d,
	\\ \label{def:sp-W-v}
	W &\coloneqq H^1(\Omega)^d, \qquad &&W_{\diver} &\coloneqq H^{1}_{\diver}(\Omega)^d,
\end{alignat}
with the corresponding norms. 
Also in this case the embeddings in~\eqref{eq:WB-emb} and in~\eqref{def:Gelfand} are available. 

\subsubsection{Bochner spaces and interpolation}\label{sec:prel-bochner} 

For $T>0$, let $I \coloneqq (0,T)$. 
For $1\leq p \leq \infty$ and a Banach space $X$, we denote the usual Bochner space by $L^p(I;X)$, and write
\begin{align}
	C_w(\overline I;X) \coloneqq \{ v \colon \overline I \to X \colon t \mapsto \skp{w}{v(t)}_{X} \in C(\overline I) \,\; \forall w \in X' \}
\end{align}
for the space of weakly continuous functions with values in $X$. 
For future purposes, we record two results on weak continuity and interpolation.  
\begin{lemma}[weak continuity]\label{lem:w-cont}
	Let $B,Z$ be reflexive Banach spaces with continuous embedding $B \hookrightarrow Z$. 
	If $v \in L^1(I;Z) \cap L^{\infty}(I;B)$ and its distributional time derivative satisfies $\partial_t v \in L^1(I;Z)$, then  $v \in C_{w}(\overline{I};B)$. 
\end{lemma}
The preceding lemma is a direct consequence of ~\cite[Lem.~1.1, Lem.~1.4, Ch.~III, \S 1]{Temam1984}. Secondly, an interpolation result: 
\begin{lemma}[interpolation {\cite[Prop.~3.4]{DiBenedetto1993}}]\label{lem:interp}	
	Let $d\geq 2$,  let $s \in [2, \frac{2d}{d-2})$ and let 
	\begin{align*}
		r\coloneqq  \frac{4s}{d(s-2)} \in (2, \infty], 
		\qquad \text{ and } \qquad 
		\theta \coloneqq  \frac{d(s-2)}{2s} \in [0,1).
	\end{align*}
	Then there exists a constant $c>0$ such that 
	\begin{align*}
		\norm{v}_{L^{r}(I;L^s(\Omega))} &\leq c \norm{v}_{L^{\infty}(I;L^{2}(\Omega))}^{1-\theta} \norm{v}_{L^2(I;H^1(\Omega))}^{\theta}
	\end{align*}
	holds for any $v \in L^{\infty}(I;L^2(\Omega))\cap L^2(I;H^1(\Omega))$. 
\end{lemma}

In particular, for $d \in \{2,3\}$, we may choose $s = 4$, $r = \frac{8}{d}$ and $\theta = \frac{d}{4}$ to obtain 
\begin{align}\label{est:interp}
	L^\infty(I;L^2(\Omega)) \cap L^2(I;H^1(\Omega)) \hookrightarrow L^{8/d}(I;L^4(\Omega)). 
\end{align}

\subsection{Finite element approximation}
\label{sec:prelim-fem}

In the following let $\Omega\subset \setR^d$ be a bounded polyhedral Lipschitz (open) domain, and let $\Gamma \coloneqq \partial \Omega$. 
We consider a conforming triangulation $\mathcal{T}$ of $\Omega$ into closed $d$-simplices, and we denote by $\mathcal{F}$ the set of faces in $\mathcal{T}$. 
By $h_K \coloneqq \diam(K)$ we denote the diameter of a $d$-simplex $K \in \mathcal{T}$, and let $h \coloneqq \max_{K \in \mathcal{T}} h_K$ be the maximal mesh size in $\mathcal{T}$. 
Similarly, let~$h_F$ denote the diameter of a face~$F \in \mathcal{F}$, and let~$h_{\Gamma} \colon \Gamma \to [0,\infty)$ be the piecewise constant function with~$h_{\Gamma}|_{F} = h_F$ for any $F \in \mathcal{F}$ with $F \subset \Gamma$. 

We consider a family of triangulations $(\mathcal{T}_k)_{k \in \mathbb{N}}$ with maximal mesh size $h_k$, with $h_k \to 0$ as $k \to \infty$. 
This family of triangulations is assumed to be shape regular, which implies that locally all mesh sizes are equivalent with constants independent of $k$, meaning that
\begin{align*}
	\abs{K}^{\frac{1}{d}} \eqsim h_K \eqsim h_F
	\quad \text{ for any face $F \in \mathcal{F}_k$ of an arbitrary $d$-simplex $K \in \mathcal{T}_k$ for any $k \in \mathbb{N}$.} 
\end{align*}
Note that we do not assume quasi-uniformity. 
In the remaining section we shall write $h \to 0$ to represent the limit $k \to \infty$, and we  denote the respective triangulations by~$\tria$ and sets of faces $\tria$ by $\faces$ with index $h$.

\subsubsection*{Discrete spaces} Let $X(K) \subset C(K)^d$ and $Q(K)\subset C(K)$ be some finite dimensional 
function spaces on $K \in \tria$ with dimension independent of $K$. 
For $\mathbb{V}\subset C(\Omega)^d$ and $\mathbb{Q}\subset L^2_0(\Omega)$ we consider the finite element spaces on $\tria$ defined by
\begin{equation}\label{eq:fem}
	\begin{aligned}
		X_h &\coloneqq \{ \bv_h\in  \mathbb{V} \colon  \bv_h|_{K} \in X(K) \text{ for all } K \in \tria\},\\
		Q_h& \coloneqq \{ q_h \in  \mathbb{Q} \colon q_h|_{K} \in Q(K) \text{ for all } K \in \tria\}. 
	\end{aligned}
\end{equation}

Let us stress that we do not assume $X_h$ to be a subset of  $H^1_{\bn}(\Omega)$, i.e., there is no conformity with respect to the impermeability boundary condition~\eqref{eq:imperm}. 
Due to the assumption on the domain the traces of finite element function in $X_h$ and in $Q_h$ are well-defined a.e.~on $\Gamma$, and they are piecewise polynomial. 
For convenience of notation in the following we refrain from denoting traces of finite element functions by $\tr$. 

We denote the space of discretely divergence-free functions and the corresponding subspace of $W$ defined in~\eqref{def:sp-W} and~\eqref{def:sp-W-v}, respectively, by 
\begin{align}\label{def:Xdiv}
	\Xdiv &\coloneqq \{\bv_h \in \Xh \colon 	\skp{ \diver \bv_h}{ q_h } - \skp{ \bv_h \cdot \bn}{ q_h }_{\Gamma} = 0\; \text{ for all } q_h \in Q_h \},
	\\
	\label{def:Bh}
	\Bh &\coloneqq \begin{cases} 
		(X_{h,\divergence},\tr(X_{h,\divergence}))
		& \text{ if } \beta > 0,\\
		X_{h,\divergence}
		&\text{ if } \beta = 0.
	\end{cases}
\end{align}
Thus, we have the embedding $W_h\hookrightarrow W \hookrightarrow B$. 
In general $\Xdiv$ is not exactly divergence-free. 
For this reason in general we have~$\Xdiv\not\subset H^1_{\divergence}(\Omega)^d$, and $\Bh \not\subset \Wdiv$, i.e., conformity with respect to the divergence constraint is not necessarily given. 
Also conformity with respect to the impermeability is not satisfied. 

\subsubsection*{Discrete norms}
Due to the penalisation terms in the Nitsche method it is useful to work with the following $h$-dependent norm
\begin{alignat}{3}\label{def:norm-h}
	\norm{\bv}_{X_h} 
	&
	\coloneqq
	\left(  
	\norm{\Dv}_{L^2(\Omega)}^2 
	+ \norm{h_{\Gamma}^{-1/2} \tr(\bv)\cdot \bn}_{L^2(\Gamma)}^2 \right)^{1/2} 
	\quad 
	&&\text{ for } \bv \in H^1(\Omega)^d.
\end{alignat}
It is a consequence of the Korn type inequality presented in Section~\ref{sec:korn} that $\norm{\cdot}_{X_h}$ is a norm on $H^1(\Omega)$. 
In particular, the whole trace is controlled, if $\norm{\cdot}_{X_h}$ is bounded. 

\begin{lemma}\label{lem:h-norm}
	Let $\Omega \subset \setR^d$ be polyhedral domain as above, with  $\Gamma = \partial \Omega $ and let $(\tria)_{h>0}$ be a family of conforming triangulations.  
	Then one has 
	\begin{align}\label{est:H1-h}
		\norm{ \bv}_B 
		\lesssim 
		\norm{ \bv}_W  
		\lesssim
		\norm{\bv}_{H^1(\Omega)}
		\lesssim
		\norm{\bv}_{X_h} \qquad \text{ for all } \bv \in H^1(\Omega)^d \text{ and for any } h>0. 
	\end{align}
	In particular, $\norm{\cdot}_{X_h}$ is a norm on $W_h$, as defined in~\eqref{def:Bh}. 
	Furthermore, for $p \in [1,2^\sharp]$ with $p < \infty$ we obtain the estimate 
	\begin{align}\label{est:trace-h}	
		\norm{\tr(\bv)}_{L^{p}(\Gamma)} \leq c \norm{\bv}_{X_h} \qquad \text{ for any } \bv \in H^1(\Omega)^d \text{ and any } h>0. 
	\end{align}
\end{lemma}
\begin{proof}	
	
	For a polyhedral domain $\Omega$ and $\Gamma = \partial \Omega$ one has that $\bv \mapsto \norm{\tr(\bv) \cdot \bn}_{L^2(\Gamma)}$ is a norm on the rigid body deformations $\mathcal{R}(\Omega)$, and hence Assumption~\ref{assump:dom} is satisfied for $q = 2$. 
	Thus, the Korn-type inequality in Theorem~\ref{thm:korn-n} for $p = 2$ yields 
	\begin{align*}
		\norm{\bv}_{H^1(\Omega)} 
		\leq c (\norm{\BD\bv}_{L^2(\Omega)} + \norm{\tr(\bv) \cdot \bn}_{L^2(\Gamma)}) \leq c \norm{\bv}_{X_h}
	\end{align*}
	for any $\bv \in H^1(\Omega)$ with constant independent of $\bv$ and of $h$.  
	In combination with the continuous embedding $W \hookrightarrow B$ and the estimate~\eqref{est:W-norm} this yields 
	\begin{align}\label{est:norm-B-h}	
		\norm{ \bv}_B \leq c \norm{ \bv}_W  \leq c  \norm{\bv}_{H^1(\Omega)} \leq c \norm{\bv}_{X_h} \qquad \text{ for any } \bv \in H^1(\Omega)^d \text{ and any } h>0. 
	\end{align}
	For this reason $\norm{\cdot}_{X_h}$ is a norm on $W_h$, as defined in~\eqref{def:Bh}. 
	The second statement results from combining this estimate with the trace estimate in~\eqref{est:trace}. 
	This proves the claim. 
\end{proof}

\subsubsection{Scott--Zhang interpolation operator}
We now summarise some properties of the trace-preserving version of the Scott--Zhang interpolation operator~\cite{SZ.1990}.
To this end, we define the simplex neighbourhood $\omega_h(K)$ of a simplex $K \in \tria$ by 
\begin{align}\label{def:ngbh}
	\omega_h(K) \coloneqq \bigcup \{K' \in \tria \colon K \cap K' \neq \emptyset \}.  
\end{align}
In the following, we denote the Lagrange finite element space of polynomial degree $\ell \in \mathbb{N}$ by 
\begin{align}\label{def:Lagrange}
	\mathcal{L}^1_{\ell}(\tria) \coloneqq \{v_h \in H^1(\Omega) \colon  v_h|_{K} \in \mathcal{P}_{\ell}(K) \;\,\forall K \in \tria\},
\end{align}
where $\mathcal{P}_{\ell}(K)$ denotes the space of all polynomials of maximal degree $\ell$ on $K$. 
Besides the well-known local approximation properties, the vector-values version of the Scott--Zhang operator preserves zero normal traces. 	
Since the latter is not recorded in the literature, for the sake of completeness we present a proof. 

\begin{lemma}[Scott--Zhang interpolation operator {\cite{SZ.1990}}]\label{lem:SZ}
	There exists a linear interpolation operator operator $\PiSZ \colon H^1(\Omega)^d \to \mathcal{L}^1_1(\Omega)^d$, which is a projection and which satisfies the following:  
	\begin{enumerate}
		\item (local approximation) \label{itm:SZ-loc-approx}
		There exists a constant $c>0$ such that for any $j \in \{0,1,2\}$ and $\max(1,j) \leq  s\leq 2$ one has 
		\begin{alignat*}{3}
			\norm{\nabla^j(\PiSZ \bv - \bv)}_{L^{2}(K)} 
			&\leq c h_{K}^{s-j} \norm{\nabla^s \bv}_{L^2(\omega_h(K))}, 
		\end{alignat*}
		for all $\bv \in H^s(\Omega)$ and all $K \in \tria$, uniformly in $h>0$.
		\item (zero normal trace preservation) \label{itm:SZ-normal}
		$\PiSZ(H^1_{\bn}(\Omega)) = X_h \cap H^1_{\bn}(\Omega) $, i.e., 
		for any $\bv \in H^1_{\bn}(\Omega)$ the normal trace  satisfies 
		\begin{align*}
			\tr(\PiSZ \bv) \cdot \bn = 0 
			\quad \text{ on } \Gamma = \partial \Omega.
		\end{align*}
	\end{enumerate}
\end{lemma}
\begin{proof} 
	Applying the version of the Scott--Zhang interpolation that preserves discrete traces in $\tr(\mathcal{L}^1_1(\tria))$ componentwisely, the local approximation properties~\ref{itm:SZ-loc-approx} hold, see~\cite[Eq.~(4.3)]{SZ.1990}.
	Furthermore, by the fact that it is linear and it  preserves zero traces, one has 
	\begin{alignat}{3}\label{eq:st-conv-testf}
		(\PiSZ \bv) \cdot \bn = \PiSZ (\bv \cdot \bn) = 0 \quad \text{ on } \partial \Omega.
	\end{alignat}
	Consequently, we have~$\PiSZ \bv  \in \mathcal{L}^1_1(\tria)^d \cap H^1_{\bn}(\Omega)$, and~\ref{itm:SZ-normal} holds. 
\end{proof}

Note that Lemma~\ref{lem:SZ}~\ref{itm:SZ-loc-approx} implies that for sufficiently smooth functions $\bv$ one has that 
\begin{align*}
	\PiSZ \bv \to  \bv \quad \text{ strongly in }  H^{1}(\Omega)^d, \text{ as } h \to 0;
\end{align*}

For the convergence proof, we shall use a stability property of the Scott--Zhang operator which follows from the local approximation properties in Lemma~\ref{lem:SZ} by standard arguments. 

\begin{corollary}\label{cor:SZ-stab}
	Let $\PiSZ \colon H^1(\Omega)^d \to \mathcal{L}^1_1(\Omega)^d$ be the operator in Lemma~\ref{lem:SZ}. 
	There exists a constant $c>0$ such that for any face $F \in \faces$ and any adjacent $K \in \tria$, i.e., $F \subset K$, one has that 
	\begin{align*}
		\norm{\nabla (\PiSZ \bv - \bv)}_{L^2(F)} \leq c h_{K}^{1/2} \norm{\nabla^2 \bv}_{L^2(\omega_h(K))}  \qquad \text{ for any } \bv \in H^2(\Omega)^d,
	\end{align*}
	uniformly in $h>0$. 
	Consequently, the following local and global stability estimates hold
	\begin{align*}
		\norm{\nabla \PiSZ \bv }_{L^2(F)} &\leq c  \norm{ \bv}_{H^2(\omega_h(K))}, 
		\quad \text{ for any } F \in \faces \text{ and  } K \in \tria \text{ with } F \subset K,	\\
		\norm{\nabla \PiSZ \bv }_{L^2(\partial \Omega)} &\leq c  \norm{ \bv}_{H^2(\Omega)},
	\end{align*}
	for any $\bv \in H^2(\Omega)^d$, uniformly in $h>0$. 
\end{corollary}

\subsubsection{The inf-sup condition}
\label{sec:prelim-infsup}	
Well-posedness of the Stokes equations is a consequence of the well-known inf-sup stability, see, e.g.,~\cite{BBF.2013}: 
It states that there is a constant $c_s>0$ such that 
\begin{align}\label{est:inf-sup-const} 
	\inf_{q \in L^{2}_0(\Omega)\setminus \{0\}} \sup_{ \bv \in H^1_{0}(\Omega)^d \setminus\{\b0\}} \frac{\ll \diver \bv, q \rr}{\norm{ \nabla \bv}_{L^{2}(\Omega)} \norm{q}_{L^{2}(\Omega)}} \geq c_{s}. 
\end{align}

Throughout, we shall make the following assumptions on the finite element spaces. 
\begin{assumption}\label{assump:fem}
We assume that $\Xh, Q_h$ as in~\eqref{eq:fem} have the following properties: 
	\begin{enumerate}[label = (\roman*)]
		\item 
		(approximation properties of $X_h$)
		\label{itm:fem-X-approx} $\mathcal{L}^1_1(\tria)^d \subset \Xh$; in particular
		for each $\bv \in H^1(\Omega)^d$ there exists a sequence $(\bv_{h})_{h >0}$ with $\bv_h \in \Xh$ for any~$h>0$ such that 
		\begin{align*}
			\norm{\bv - \bv_h}_{H^1(\Omega)}& \to 0 \quad \text{ as } h \to 0;
		\end{align*}  
		\item (approximation properties of $Q_h$) 
		\label{itm:fem-Q-approx} 
		for each $q \in C^\infty(\overline \Omega)$ there exists a sequence $(q_{h})_{h >0}$ with $q_h \in Q_h$ for any~$h>0$ such that 
		\begin{align*}
			\norm{q - q_h}_{L^2(\Omega)} + \norm{q - q_h}_{L^2(\Gamma)}
			& \to 0 \quad \text{ as } h \to 0;
		\end{align*} 
		\item \label{itm:d-inf-sup-H10} 
		(discrete inf-sup stability) 
		there exists a constant $\overline{c}_s>0$ such that 
		\begin{align*} 
			\inf_{q_h \in Q_h \setminus \{0\}} \sup_{\bv_h \in \Xh \cap H^1_0(\Omega)^d  \setminus \{\b0\}} \frac{\ll \diver \bv_h, q_h \rr }{\norm{\nabla \bv_h}_{L^2(\Omega)} \norm{q_h}_{L^2(\Omega)}} \geq \overline{c}_s \quad \text{ for any } h > 0.
		\end{align*}
	\end{enumerate} 
\end{assumption}

Many classical mixed finite element pairs for the Stokes problem satisfy those assumptions, including the (generalised) Taylor-Hood elements~\cite{TaylorHood1973,GiraultRaviart1986}, the Bernardi--Raugel element of second order~\cite{BR.1985}, the MINI element~\cite{ArnoldBrezziFortin1984}, the conforming Crouzeix--Raviart element~\cite{CrouzeixRaviart1973}, the second order Guzm\'an--Neilan elements~\cite{GuzmanNeilan2014,GuzmanNeilan2014a}, and the Scott--Vogelius element for certain polynomial degrees and meshes~\cite{ScottVogelius1985,GuzmanScott2019}. 

\subsubsection{Numerical convective term}\label{sec:convterm}
The convective term in~\eqref{eq:NS-unst} in weak formulation can be represented as $b(\bu,\bu,\bv)$, defining the trilinear form~$b$ for any $\bu,\bv,\bw \in C^{\infty}(\overline \Omega)^d$ by 
\begin{align}\label{def:conv-term}
	b(\bu,\bv,\bw) 
	\coloneqq  
	- \int_{\Omega} \bv \otimes \bu : \nabla \bw \, \mathrm{d} \bx
	+ \int_{\Gamma} (\bu \cdot \bn) (\bv \cdot \bw) \ds
	,
\end{align}
where $(\nabla \bw)_{i,j} \coloneqq \partial_{x_j} w_i$. 
For $\bu$ satisfying 
$\bu \cdot \bn = 0$ on $\Gamma$ the last term vanishes and we arrive at the formulation as used in the homogeneous case~\cite[Ch. II, \S 1, eq. (1.12)]{Temam1984}.  
Integrating by parts one can see that 
\begin{align*}
	b(\bu,\bv,\bw) = - b(\bu,\bw,\bv) 
	+ \skp{\bu \cdot \bn}{\bv \cdot \bw}_{\Gamma}
	+  \skp{\divergence \bu}{\bv \cdot \bw}. 
\end{align*}
Hence, it follows that $b(\bu,\cdot,\cdot)$ is skew-symmetric in the sense that 
\begin{align*}
	b(\bu,\bv,\bw) = - b(\bu,\bw,\bv) \quad \text{ provided that } \divergence \bu = 0 \text{ and } \bu \cdot \bn = 0 \text{ on } \Gamma.
\end{align*}

We want to preserve this property on the discrete level, despite the discrete velocity functions $\bu_h \in X_h$ in the Nitsche method in general neither satisfying $\bu_h \cdot \bn = 0$  on $\Gamma$ nor $\diver \bu_h = 0$. 
Hence, as usual we define the trilinear form 
\begin{equation}\label{def:conv-term-num}
	\begin{aligned}
		\widetilde{b}(\bu_h, \bv_h, \bw_h) &\coloneqq \tfrac{1}{2} \left(b(\bu_h, \bv_h, \bw_h) - b(\bu_h, \bw_h, \bv_h) \right)\\
		& = \tfrac{1}{2}
		\left( - \skp{\bv_h \otimes \bu_h}{\nabla \bw_h} 
		+ \skp{\bw_h \otimes \bu_h}{\nabla \bv_h} \right) 
		\quad \text{ for } \bu_h, \bv_h, \bw_h \in \Xh,
	\end{aligned}
\end{equation}
and $\widetilde{b}(\bu_h, \cdot, \cdot)$ is skew-symmetric by definition. 
This coincides with the trilinear form for the case of homogeneous Dirichlet boundary conditions~\cite{Temam1984}, since the boundary term in~\eqref{def:conv-term} is skew-symmetric for fixed $\bu$ and hence cancels in $\widetilde{b}$. 
For this reason, the standard estimates apply. 
Integrating by parts we find that 
\begin{align} \label{eq:conv-id}
	\widetilde{b}(\bu, \bv, \bw)
	=-\skp{\bv \otimes \bu}{ \nabla \bw}
	+ \tfrac{1}{2} \skp{(\bu \cdot \bn)}{ (\bv \cdot \bw)}_{\Gamma}
	- \tfrac{1}{2} \skp{\diver \bu}{(\bv \cdot \bw)}.  
\end{align}
Thus, we have that
\begin{align}\label{eq:conv-id-2}
	\widetilde{b}(\bu, \bv, \bw) = b(\bu, \bv, \bw) \quad \text{ if } \divergence  \bu = 0 \text{ and } \bu \cdot \bn = 0 \text{ on } \Gamma. 
\end{align} 
The following estimates allow us to extend the trilinear forms $b, \widetilde{b}$ in~\eqref{def:conv-term}, \eqref{def:conv-term-num} to certain Sobolev spaces: 
By Hölder's inequality and the fact that $2^* =  \frac{2d}{d-2} > 4$ for $d \in \{2,3\}$, whereby $H^1(\Omega) \hookrightarrow L^4(\Omega)$, we obtain 
\begin{align}\label{est:conv-a1}
	\abs{\skp{\bv\otimes \bu}{\nabla\bw}} 
	&\leq 
	\norm{\bv}_{L^4(\Omega)} 
	\norm{\bu}_{L^4(\Omega)} 
	\norm{\nabla \bw}_{L^2(\Omega)}
	\leq 
	c \norm{\bu}_{H^1(\Omega)} \norm{\bv}_{H^1(\Omega)} \norm{\bw}_{H^1(\Omega)}.  
\end{align}
For future reference, we moreover note that for the numerical convective term~\eqref{def:conv-term-num} the estimate~\eqref{est:conv-a1}  implies that 
\begin{equation}
	\begin{aligned}\label{eq:traccovanish}
		|\widetilde{b}(\bu,\bu,\bw)| &\leq
		\norm{\bu}_{L^4(\Omega)}^2 \norm{\nabla \bw}_{L^2(\Omega)} + \norm{\bu}_{L^4(\Omega)} \norm{\nabla \bu}_{L^2(\Omega)} \norm{\bw}_{L^4(\Omega)}
		\\
		&
		\lesssim  \norm{\bu}_{H^1(\Omega)}^2 \norm{ \bw}_{H^1(\Omega)}.
	\end{aligned}
\end{equation}
In view of~\eqref{est:conv-a1}, \eqref{eq:traccovanish}, we conclude that $b(\cdot,\cdot,\cdot)$ and $\widetilde b(\cdot,\cdot,\cdot)$ can be extended to $H^1(\Omega)^d \times H^1(\Omega)^d \times H^1(\Omega)^d$ as bounded trilinear mappings. 

For the time-dependent case, similarly, with Hölder's inequality and $\frac{d}{8} + \frac{d}{8} + \frac{4-d}{4}  = 1$ we estimate 
\begin{equation}
	\begin{aligned}\label{est:conv-t-a1}
		\int_0^T \abs{\skp{\bv\otimes \bu}{\nabla\bw}}  \dt 
		&\leq 
		\int_0^T \norm{\bv}_{L^4(\Omega)} \norm{\bu}_{L^4(\Omega)} \norm{\nabla \bw}_{L^2(\Omega)} \dt\\
		&\leq 
		c \norm{\bv}_{L^{8/d}(I;L^4(\Omega))}
		\norm{\bu}_{L^{8/d}(I;L^4(\Omega))}  \norm{\bw}_{L^{\frac{4}{4-d}}(I;H^1(\Omega))}.
	\end{aligned}
\end{equation}
Note that $\frac{4}{4-d} = 2$ if $d = 2$, and $\frac{4}{4-d} = 4$ if $d = 3$. 
Again with Hölder's inequality, with 
$\frac{1}{2} + \frac{d}{8} + \frac{4-d}{8}  = 1$, we find that
\begin{equation}
	\begin{aligned}\label{est:conv-t-a2}
		\int_0^T \abs{\skp{\bw\otimes \bu}{\nabla\bv}}  \dt 
		&\leq \int_0^T \norm{\bw}_{L^4(\Omega)} \norm{\bu}_{L^4(\Omega)} \norm{\nabla \bv}_{L^2(\Omega)} \dt\\
		&\leq 
		c\norm{\bw}_{L^{\frac{8}{4-d}}(I;H^1(\Omega))} 
		\norm{\bu}_{L^{8/d}(I;L^4(\Omega))}
		\norm{\nabla \bv}_{L^{2}(I;L^2(\Omega))}.
	\end{aligned}
\end{equation}
Note that $\frac{8}{4-d} = 4$ if $d = 2$, and $\frac{8}{4-d} = 8$ if $d = 3$. 
From the exponents we see that admissibility of 
a solution as test function 
is lost in dimension $d = 3$. 

With the interpolation estimate in~\eqref{est:interp} we may extend the trilinear forms $b,\widetilde{b}$ to bounded mappings 
\begin{align*}
	L^\infty(I;L^2(\Omega)^d) \cap L^2(I;H^1(\Omega)^d) \times L^\infty(I;L^2(\Omega)^d) \cap L^2(I;H^1(\Omega)^d) \times L^\frac{8}{4-d	}(I;H^1(\Omega)^d) \to \setR.
\end{align*}

\subsubsection{$L^2$-projection mapping to $W_h$}

Recall that we have that $\Bh \subset W\hookrightarrow B$, with the spaces $W,B$ as defined in~\eqref{def:sp-W}, \eqref{def:sp-B}, and~\eqref{def:sp-B-v}, \eqref{def:sp-W-v}, respectively, and with $W_h$ as defined in~\eqref{def:Bh}. 
However, $W_h$ is not a subspace of the divergence-free space $W_{h,\divergence}$. 

We introduce the $B$-orthogonal projection operator $\PiBh \colon B \to \Bh $ defined by 
\begin{align}\label{def:PiBh}
	(\PiBh  \bv, \bw_h)_B = 	( \bv, \bw_h)_B \quad \text{ for any } \bw_h \in \Bh,
\end{align}
for $\bv \in B$.  
By definition, one has the stability estimate 
\begin{align}\label{def:PiBh-stab}
	\norm{\PiBh \bv}_B \leq	\norm{\bv}_B \qquad \text{ for any } \bv \in B.
\end{align} 

\subsection{Time discretisation and compactness}\label{sec:prel-timediscr} 

For simplicity we consider a time-descretisation with uniform time step size. 
For $m\in \mathbb{N}$ and $\delta  \coloneqq \delta_m \coloneqq \frac{T}{m}$ we consider the time grid points $t_j \coloneqq  j \delta$ 
for $j \in \{0,\ldots, m\}$
, so that 
$0 = t_0 < t_1 < \ldots< t_m = T$. 
Denoting the time intervals $I_j  = (t_{j-1},t_j]$ we work with the  partition 
\begin{align*}
	J_\delta \coloneqq \{ I_j \colon j \in \{1, \ldots, m\}\}
\end{align*} 
of the interval $(0,T]$. 
In the following we shall employ a backward Euler time-stepping. 
For a Banach space $X$ and a sequence $(\phi_j)_{j \in \mathbb{N}_0} \subset X$ we denote the difference quotient as
\begin{align*}
	\difft \phi_j \coloneqq \tfrac{1}{\delta}(\phi_j - \phi_{j-1}) \qquad\text{ for } j \in \mathbb{N}. 
\end{align*}
Note that for  $(\phi_j)_{j \in \mathbb{N}} \subset H$ for a Hilbert space $H$ with inner product $(.,.)_H$ we have
\begin{align}\label{eq:difft}
	(\difft \phi_j,\phi_j)_H = \tfrac{1}{2\delta} \left(\norm{\phi_j}_H^2 - \norm{\phi_{j-1}}_H^2 + \norm{\phi_j - \phi_{j-1}}_H^2 \right) \qquad \text{ for any } j \in \mathbb{N}. 
\end{align}
This follows from the identity 
$\tfrac{1}{2}\left(a^2 - b^2 + (a-b)^2\right) = 	(a-b)a.$

For a Banach space $X$ we denote the space of (right-continuous) piecewise constant in time functions with values in $X$ by 
\begin{align*}
	\mathcal{L}_{0}^0(J_\delta; X)\coloneqq \{\psi \in L^{\infty}(0,T;X)\colon \psi(t,\cdot)|_{I_j} \in X \text{ is constant in time for any } j \in \{1,\ldots, m \}\}. 
\end{align*}
Note that for any $\psi^\delta \in \mathcal{L}_{0}^0(J_\delta; X)$ with $\psi^\delta_j \coloneqq \psi^\delta|_{I_{j}}$ one has that 
\begin{align*}
	\norm{\psi^\delta}_{L^p(I;X)} 
	&= \left( \delta \sum_{j= 1}^m  \norm{\psi^\delta_j}_X^p \right)^{1/p} 
	\quad \text{ if } p \in [1,\infty),
	\\
	\norm{\psi^\delta}_{L^\infty(I;X)} 
	&= \max_{j= 1, \ldots, m} 
	\norm{\psi^\delta_j}_X.
\end{align*}

\subsubsection{$L^2$-projection}\label{def:Pid}

For a Hilbert space $H$ we denote by $\Pi_\delta \colon L^2(I;H) \to \mathcal{L}_{0}^0(J_\delta; H)$ the $L^2(I;H)$-orthogonal projection mapping to $\mathcal{L}_{0}^0(J_\delta; H)$. 
By definition it satisfies 
\begin{align}\label{est:Pd-stab-L2}
	\norm{\Pid \psi}_{L^2(I;H)} \leq \norm{\psi}_{L^2(I;H)} \quad \text{ for any } \psi \in L^2(I;H). 
\end{align}
Furthermore, one can show that 
\begin{align}\label{est:Pid-conv}
	\Pid \psi \to \psi \quad \text{ strongly in } L^2(I;H), \quad \text{ as } \delta \to 0. 
\end{align}

\subsubsection{Compactness tools}

The standard Aubin--Lions compactness result for some Banach spaces $V,B,Y$ with compact embedding $V \hookrightarrow B$ and  continuous embedding $B \hookrightarrow Y$, $I = (0,T)$ and $q \in [1,\infty)$ ensures the following: 
Assuming that a sequence of functions is bounded in $L^q(I;V)$ and the sequence of its time derivatives is bounded in $L^{q}(I;Y)$, it allows us to conclude strong convergence of a subsequence in $L^{q}(I;B)$. 
In the convergence proof we shall employ a  discrete version by~\cite{GallouetLatche2012}. 
{Note that the Nitsche method uses the spaces $X_{h,\divergence}$ as in~\eqref{def:Xdiv}, which are not $H^1_{\bn}(\Omega)$-conforming.
	For this reason, we use a non-conforming discrete version of the Aubin--Lions lemma which allows for $h$-dependent norms. 
	More specifically, we employ the following special case of a  result}  due to~\cite{GallouetLatche2012}, see also~\cite{DreherJuengel2012,CJL.2014} for similar versions:  

\begin{lemma}[discrete Aubin--Lions Lemma, {\cite[Thm.~3.4, Rmk.~6]{GallouetLatche2012}}]\label{lem:d-AL-0}
	Let $(B,\norm{\cdot}_B)$ be a Hilbert space, and let $(V_k)_{k \in \mathbb{N}}$ be a sequence of finite dimensional subspaces of $B$. 
	Moreover, assume that for each $k \in \mathbb{N}$ there are norms $\norm{\cdot}_{V_k}$ and $\norm{\cdot}_{Y_k}$ on $V_k$ such that the following conditions are satisfied: 	 
	\begin{enumerate}[label = (\roman*)]
		\item \label{itm:norm-Xk-0}
		If $(v_k)_{k \in \mathbb{N}}\subset B$ is a sequence with $v_k \in V_k$ for all $k \in \mathbb{N}$ which satisfies 
		\begin{align*}
			\norm{v_k}_{V_k} \leq c \quad \text{ for any } k \in \mathbb{N},
		\end{align*} 
		for a constant $c>0$ independent of $k$, then there exists a (non-relabelled) subsequence 
		{and $v\in B$} such that $v_k \to v$ in $B$ as $k \to \infty$;
		\item \label{itm:norm-Yk-0} The norm $\norm{\cdot}_{Y_k}$  on $V_k$ is dual to $\norm{\cdot}_{V_k}$
		with respect to the inner product in $B$, that is, 
		\begin{align*}
			\norm{v_k}_{Y_k} \coloneqq \sup_{\phi_k \in V_k} \frac{(v_k,\phi_k)_B}{\norm{\phi_k}_{V_k}} \qquad \text{  for }v_k \in V_k.
		\end{align*}	
	\end{enumerate}	
	For a given $T>0$ and a sequence $(m_k)_{k \in \mathbb{N}}$ with $m_k \to \infty$, 
	let $\delta_k \coloneqq T/m_k$ and consider the corresponding equidistant partition $J_{\delta_k}$ of $I = (0,T)$. Then the following holds: 
    
    \noindent If $(\psi^{k})_{k\in\mathbb{N}}$ is a sequence such that 
	\begin{itemize}
		\item[(iii)] $\psi^k \in \mathcal{L}^0_0(J_{\delta_k};V_k)$ for any $k \in \mathbb{N}$, and, 
		\item[(iv)] denoting $\psi^k_j \coloneqq \psi^k|_{I_j}$, there exists a constant $c>0$ and an exponent $q \in [1,\infty)$ such that 
		\begin{align}\label{est:d-Aubin-Lions}
			\delta_k \sum_{{j = 1}}^{m_k} \norm{\psi_j^k}^q_{V_k} + \delta_k \sum_{{j = 2}}^{m_k} \norm{\frac{1}{\delta_k}(\psi^k_j -\psi^k_{j-1})}_{Y_k}^q
			\leq c \qquad \text{ for all } k \in \mathbb{N},
		\end{align}
	\end{itemize}
	then, there exists a (non-relabelled) subsequence of $(\psi^k)_{k \in \mathbb{N}}$ which converges strongly to a function $\psi$ in $L^q(I;B)$ as $k \to \infty$. 
\end{lemma}

\begin{remark}
	The assumption in Lemma~\ref{lem:d-AL-0}~\ref{itm:norm-Xk-0} replaces compactness of the embedding $W \hookrightarrow B$, and condition~\ref{itm:norm-Yk-0} replaces the continuous embedding $B \hookrightarrow Y$. 
\end{remark}

\begin{remark}\label{rmk:AL-appl}
	We shall apply the discrete Aubin--Lions compactness result in Proposition~\ref{lem:d-AL-0} in both cases $\beta > 0$ and $\beta = 0$ as follows: 
	As a Hilbert space, we choose $B$ as in~\eqref{def:sp-B} or~\eqref{def:sp-B-v}, and as a sequence of $(V_k,\norm{\cdot}_{V_k})$ we choose $(W_h, \norm{\cdot}_{X_h})$ as defined in~\eqref{def:Bh} and~\eqref{def:norm-h}. 
	We denote
	\begin{align}\label{def:Yh}
		\norm{ \bv_h}_{Y_h} 
		\coloneqq \sup_{\bphi_h  \in W_h} \frac{( \bv_h, \bphi_h)_B}{\norm{ \bphi_h}_{X_h}} \qquad \text{ for } \bv_h \in W_h. 
	\end{align}
	Recall that $W_h$ is based on discretely divergence-free finite element functions, and this circumstance is reflected in the dual norm.  
	In order to verify~\ref{itm:norm-Xk-0}, we note that by Lemma~\ref{lem:h-norm} we have 
	\begin{align*}
		\norm{ \bv_h}_{W} \leq  \norm{ \bv_h}_{X_h} \leq c,
	\end{align*}
	uniformly in $h>0$. 
	Thus, by compactness of the embedding $W \hookrightarrow\hookrightarrow B $, condition~\ref{itm:norm-Yk-0} in Proposition~\ref{lem:d-AL-0} is satisfied.  
	Therefore, to apply Proposition~\ref{lem:d-AL-0}, it suffices to ensure the corresponding estimate in~\eqref{est:d-Aubin-Lions} for some $q \in [1,\infty)$. 
\end{remark}

\section{Convergence of the Nitsche method}
\label{sec:un-NS}	

We now come to the main result of the present paper and its proof. 
To this end, let $\Omega\subset\R^{d}$, for $d\in\{2,3\}$, be a bounded, polyhedral and  (open) Lipschitz domain, its boundary $\Gamma\coloneqq\partial\Omega$ being oriented by the outer unit normal $\bn\colon\Gamma\to\mathbb{S}^{d-1}$. As in Section~\ref{sec:model-results}, we consider the Navier--Stokes equations~\eqref{eq:NS-unst}--\eqref{eq:NS-law}, which we recall to read 
\begin{subequations}\label{eq:NS-unst-full}
	\begin{alignat}{2}
		\label{eq:NS-unst-A}
		\partial_t \bu + 	\diver(\bu \otimes \bu) - 2 \nu \diver \BD \bu  + \nabla \pi &= \bf \qquad \quad &&\text{ in } Q,\\
		\diver\bu &= 0 \qquad \quad &&\text{ in } Q,\\
		\bu(0,\cdot) &= \bu_0 \quad && \text{ on } \Omega, 
	\end{alignat}
	subject to the impermeability~\eqref{eq:imperm} and dynamic boundary condition~\eqref{eq:genbc}:
	\begin{alignat}{2}
		\label{eq:NS-unst-bc-1}
		\bu \cdot \bn &= 0 \quad &&\text{ on } I \times \Gamma,\\
		\label{eq:NS-unst-bc-2}
		-2\nu ( \Du\,\bn)_\tau &=  \bsigma 
		+ \beta \partial_t \bu
		\quad &&\text{ on } I\times\Gamma, 
	\end{alignat}
\end{subequations}
where the parameter $\beta\geq  0$ is as in~\eqref{eq:genbc}. 
Here, as usual, $T>0$ is a given final time, $I\coloneqq (0,T)$ and $Q\coloneqq I\times Q$. Moreover, $\bu_{0}$ is an initial velocity, $\bf$ an external force, and $\nu>0$ represents the viscosity. 

	\noindent We suppose  that one of the following conditions holds for some $r \in [1,\infty)$: 
	\begin{enumerate}[label = (C\arabic*)]
		\item \label{itm:case-expl-noncoerc} 
		$\bsigma= \Srel(\bu_\tau)$ with $\Srel$ satisfying Assumption~\ref{assump:s-expl} with exponent $r$ and constant $\lambda \geq 0$; 
		\item \label{itm:case-impl-coerc} 
		$\gbd(\bsigma  + \lambda \bu_\tau, \bu_\tau)= \b0$ with $\gbd$ satisfying Assumption~\ref{assump:gbd-mon} with exponent $r$, and with $\lambda \geq 0$; 
	\end{enumerate}
\noindent 
In case~\ref{itm:case-expl-noncoerc},  the relation between $\bsigma$ and $\bu_\tau$ is explicit, 
possibly nonmonotone and noncoercive, 
but bounded with constant $\lambda \geq 0$. 
In case~\ref{itm:case-impl-coerc} the relation between $\bsigma$ and $\bu_\tau$ is possibly implicit, $\lambda$-monotone, and coercive. 

In view of the underlying weak formulation of~\eqref{eq:NS-unst-full}, we denote the tangential trace of a Sobolev function $\bv\in W^{1,1}(\Omega)^{d}$ by
\begin{align}\label{def:tr-tau}
	\tr_{\tau}	(\bv) \coloneqq \tr(\bv) - (\tr(\bv) \cdot \bn) \bn, 
\end{align}
and recall that the spaces 
$B, \Bdiv, \Wdiv$ have been introduced in~\eqref{def:sp-B}--\eqref{def:sp-W-v}. Based on these conventions, we follow~\cite{Abbatiello2021,BulicekMalekMaringova2023} and introduce the notion of weak solutions as follows: 
\begin{definition}[weak solution]\label{def:unst-w-sol}
	In the above situation, let $\bf \in L^2(0,T;B)$ and let $\bu_0 \in \Bdiv$. 
	Furthermore, let $\lambda \geq 0$ be constant and assume that~\ref{itm:case-expl-noncoerc} or~\ref{itm:case-impl-coerc} holds. 
	Then we call a  pair of functions $(\bu, \bsigma)$ with
	\begin{alignat*}{3}
		\bu	&\in L^2(I;\Wdiv) \cap C_w(\overline I;\Bdiv)\quad &&\text{ with } \;
		\partial_t \bu \in L^{1}(I;(\Wdiv)'), 
		\\
		\bsigma &\in L^{p}(I;L^{p}(\Gamma)^d) \qquad &&\text{ for } p = \min (2,r'),
	\end{alignat*} 
	a \emph{weak solution} to~\eqref{eq:NS-unst-full}, if they satisfy
	\begin{alignat*}{3}
		\skp{\partial_t \bu}{{\bphi}}_{\Wdiv} 
		+2\nu  \skp{ \BD \bu}{ \BD \bphi }	+ b(\bu, \bu, \bphi) 
		+\skp{ \bsigma}{ \tr_{\tau}(\bphi) }_{\Gamma}  
		&= ({\bf}, {\bphi})_B
	\end{alignat*}
	for all  $\bphi \in \Wdiv$ and a.e.~in $I$ as well as  $\lim_{t \to 0_+}	\norm{{\bu}(t) - {\bu}_0}_B = 0$, 
	and 
	\begin{alignat*}{3}
		\bsigma &= \Srel(\tr_{\tau} (\bu))  \qquad &&\text{ a.e.~on } I\times \Gamma \qquad \text{ in case~\ref{itm:case-expl-noncoerc},}\\
		\gbd(\bsigma + \lambda \tr_{\tau} (\bu), \tr_{\tau}(\bu))&=\b0 \qquad &&\text{ a.e.~on } I\times \Gamma \qquad \text{ in case~\ref{itm:case-impl-coerc}.}
	\end{alignat*}
\end{definition}
Here, we choose a pressure-free formulation to avoid difficulties with non-integrable pressure. Moreover, we point out that if $\beta >0$, then the first term of the variational formulation also includes $\partial_t  \tr(\bu)$; the divergence constraint and the impermeability are contained in the definition of $\Wdiv$. 
\medskip 

In Section~\ref{sec:unst-Nitsche} we introduce the general Nitsche method to approximate solutions to system~\eqref{eq:NS-unst-full}. 
Section~\ref{sec:unst-main} contains the main result on convergence of subsequences of approximate solutions to weak solutions for $r\leq 2$. 
The following sections deal with the convergence proof, starting in Section~\ref{sec:unst-apriori} with the a priori estimates. 
Section~\ref{sec:unst-conv} contains the convergence of the approximate functions, Section~\ref{sec:unst-limit} verifies that the limiting functions satisfy the formulation, and Section~\ref{sec:unst-id} contains the identification of the nonlinear boundary terms for $r\leq 2$. 

The main convergence result for $r\leq2$ is contained in Theorem~\ref{thm:main-unsteady} and extensions for $r>2$ in special cases are presented in Section~\ref{sec:extensions_rgeq2}. 
Our results in particular constitute an existence proof of weak solutions, which is new in this general setting 
for the noncoercive case~\ref{itm:case-expl-noncoerc}, the nonmonotone implicit case with $\lambda >0$ in~\ref{itm:case-impl-coerc} and for $r = 1$. 

\subsection{The Nitsche method}
\label{sec:unst-Nitsche}
In the following, we employ an implicit time stepping and a finite element discretisation with a Nitsche method in space. 
The feature of this method is that the boundary condition $\bu \cdot \bn$ is not strongly enforced by including it in the finite element space. 
Instead, it is weakly imposed by penalisation in the weak formulation. 

We assume that, for a sequence of triangulations $(\tria)_{h>0}$ of the polyhedral domain $\Omega$, we have pairs of inf-sup stable finite element spaces $(X_h,Q_h)$ as defined in~\eqref{eq:fem}, satisfying  Assumption~\ref{assump:fem}. 
With the $B$-orthogonal projection $\PiBh\colon B\to\Bh$, see~\eqref{def:PiBh}, we discretise $ \bu_0 \in \Bdiv \subset B$ as 
\begin{align}\label{def:u0-discr}
	\bu_{0}^{h} \coloneqq \PiBh  \bu_0. 
\end{align}
Consequently, we obtain 
\begin{align}\label{est:PiBh-stab-u0}
	\norm{ \bu^h_0}_B  = \norm{\PiBh  \bu_0}_B &\leq	\norm{ \bu_0}_B,\\
	\label{eq:conv-u0}
	\bu_0^h 
	= \PiBh  \bu_0 & \to   \bu_0 \quad \text{ strongly in } B 
	\quad \text{ as } h \to 0. 
\end{align}
For the time discretisation we use a uniform time step size $\delta>0$ to obtain a sequence of partitions $(J_\delta)_{\delta>0}$ of $I$, see Section~\ref{sec:prel-timediscr}. 
Then, with the $L^2(I;B)$-orthogonal projection $\Pid$ mapping to $\mathcal{L}^0_0(J_{\delta};B)$ as defined in~\eqref{def:Pid}, we discretise $ \bf \in L^2(I;B)$ as 
\begin{align}\label{def:f-discr}
	\bf^{\delta} \coloneqq \Pid( \bf) \qquad \text{ with } \;
	\bf^{\delta}_j \coloneqq  \bf^{\delta} |_{I_j} 
	=  \dashint_{I_j}  \bf(t) \dt  \in B \quad \text{ for  } j \in \{1, \ldots, m\}. 
\end{align}
In consequence, we have that
\begin{align}\label{est:L2-stab-f}
	\norm{\Pid  \bf }_{L^2(I;B)}& \leq 	\norm{  \bf }_{L^2(I;B)},\\
	\label{eq:f-conv}
	\Pid  \bf &\to \bf \qquad \text{ strongly in } L^2(I;B)  \quad \text{ as } \delta \to 0.
\end{align}

For the definition of the discrete problem, we replace the convective term $b$ by the numerical convective term $\widetilde{b}$  as defined in~\eqref{def:conv-term-num}. 
Moreover, in case~\ref{itm:case-impl-coerc}, let~$\Seps$ with $\varepsilon>0$  be an approximation as in Assumption~\ref{assump:gbd-reg}. Lastly, let $\theta\in \{-1,0,1\}$ and let
$\alpha>0$ be a parameter to be specified later. 

\subsubsection*{The discrete problem.} 
Let $\epsilon, h, \delta >0$, and denote $\discr \coloneqq (\epsilon, h, \delta) \in \mathbb{R}^3_{>0}$. 
Then the sequences $(\overline{\bu}_j^{\discr})_{j = 0,\ldots, m}$ and $(\bsigma_j^{\discr})_{j = 1, \ldots, m}$ are defined by the following iteration: For
\begin{itemize}
	\item $j=0$, put  ${\bu}_{0}^\discr \coloneqq  \bu_0^h =  \PiBh {\bu}_0$. 
	\item $j\in\{1,...,m\}$, so in the $j$-th time step, let ${\bu}_{j-1}^{\discr}\in \Xdiv$ be the solution from the $(j-1)$-st time step and  find $({\bu}_{j}^{\discr}, \pi_{j}^{\discr})\in (\Xh,Q_h)$ such that 
	\begin{subequations}\label{eq:unst-discr}
		\begin{equation}	\label{eq:unst-discr-a}
			\begin{aligned}
				(\difft {\bu}^{\discr}_j,{\bv}_h)_B
				&+ 2\nu \skp{ \BD \bu^{\discr}_j}{ \BD \bv_h } 
				+ \widetilde{b}(\bu^{\discr}_j, \bu^{\discr}_j, \bv_h)
				- \skp{ \pi_j^{\discr}}{ \diver \bv_h } 
				+ \skp{q_h}{\diver \bu_j^{\discr}}  
				&&
				\\
				&
				+ \skp{\pi_j^\discr}{\bv_{h} \cdot \bn}_{\Gamma}
				-  \skp{q_h}{\bu_{j}^{\discr} \cdot \bn}_{\Gamma} 
				&&\\
				& 
				+ \skp{ \bsigma^\discr_j}{\tr_{\tau}(\bv_h) }_{\Gamma}
				+  \nu \alpha   \skp{h_{\Gamma}^{-1} \bu^{\discr}_j \cdot \bn}{  \bv_h \cdot \bn}_{\Gamma}&& 
				\\ 
				& 
				- 2\nu \skp{ \left( (  \BD \bu_j^{\discr}) \bn \right) \cdot \bn }{ \bv_h \cdot \bn }_{\Gamma}
				- 2 \nu \theta \skp{  \left(( \BD \bv_h) \bn \right) \cdot \bn }{ \bu_j^{\discr} \cdot \bn }_{\Gamma}
				&&= ({\bf}_j^{\delta} ,\bv_h )_B
			\end{aligned}
		\end{equation}
		for all $(\bv_h,q_h)\in X_h \times Q_h$, where 
		\begin{align}\label{eq:unst-discr-b}
			\bsigma^\discr_j \coloneqq \begin{cases}
				\Srel(\tr_\tau(\bu^\discr_j)) \qquad &\text{ in case~\ref{itm:case-expl-noncoerc},}\\
				\Seps(\tr_\tau(\bu^{\discr}_j)) - \lambda \tr_\tau(\bu^{\discr}_{j})
				&\text{ in case~\ref{itm:case-impl-coerc}.} 
			\end{cases}
		\end{align}
	\end{subequations}
\end{itemize}
As usual, we note that testing~\eqref{eq:unst-discr-a} with $(\b0,q_h)$ yields that $\bu_j^{\discr} \in X_{h,\divergence}$, and a pressure-free formulation is available by testing with $(\bv_h,0)$ for $\bv_h \in X_{h,\divergence}$. 

All terms in the first line of~\eqref{eq:unst-discr-a} are contained in a mixed method for Navier--Stokes equations subject to homogeneous boundary conditions. 
The boundary terms involving $\tr_\tau$ are due to the general class of boundary condition, and the remaining boundary terms are due to the use of the Nitsche method. 

We denote the right-continuous, piecewise constant-in-time interpolants of the velocity functions $(\bu^{\discr}_j)_{j \in \{1,\ldots, m\}}$, pressure functions  $(\pi^{\discr}_j)_{j \in \{1,\ldots, m\}}$ and stress functions  $(\bsigma^\discr_j)_{j \in \{1,\ldots, m\}}$  by $\bu^{\discr}$, $\pi^{\discr}$, and 
$\bsigma^\discr$, respectively. 

\begin{remark}\label{rem:TRockt} We comment on several aspects of the above discrete formulation: 
	\begin{enumerate}[label = (\alph*)]
		\item\label{itm:T1}
		All functions in $ \Xh\subset H^1(\Omega)^{d}$ have a trace in $L^2(\Gamma)^{d}$. Hence, with ${\bu}^{\discr}_j,\overline{\bv}_h \in W$, we have 
		\begin{align*} 
			(\difft {\bu}^{\discr}_j,{\bv}_h)_B 
			= 	(\difft {\bu}^{\discr}_j,{\bv}_h) + \beta(\difft( \tr(\bu^{\discr}_j)),\tr(\bv_h))_{\Gamma} \quad \text{ for all } \bv_h \in \Xh. 
		\end{align*}
		This formulation covers both cases $\beta> 0$ and $\beta = 0$. 
		\item\label{itm:T3} 
		As to the coercivity of the discrete formulation, the penalisation parameter~$\alpha$ has to be chosen sufficiently large, see Assumption~\ref{as:param},  unless $\lambda=0$ and the skew-symmetric variant of Nitsche's method is employed ($\theta=-1$). 
		Since it affects the condition number of the system, however, it should not be too large. 
		In principle, one may want to choose the parameter locally as  in~\cite{BringmannCarstensenStreitberger2024}. 
		In our setting this would be possible in the case of monotone relations; in nonmonotone relations, it depends on a constant appearing in a global Korn inequality.
	\end{enumerate}
\end{remark}

\subsection{Main convergence result for $r\leq 2$}
\label{sec:unst-main}
In the following we assume that, in each of the cases~\ref{itm:case-expl-noncoerc} and~\ref{itm:case-impl-coerc}, the corresponding constant $\lambda\geq 0$ is sufficiently small and $\alpha>0$ is sufficiently large. 
Let $c_{\tr,K}>0$ be the constant in the estimate 
\begin{align}\label{est:tr-K}
	\norm{\bv_h}_{L^2(\Gamma)} \leq c_{\tr,K} \norm{\bv_h}_{X_h} \qquad \text{ for any } \bv_h \in X_h,\;\text{uniformly in}\;h, 
\end{align}
see Lemma~\ref{lem:h-norm}.  
We denote the constant in a standard inverse trace inequality by $c_{\tr}>0$. 
Pertaining to Remark~\ref{rem:TRockt}~\ref{itm:T3}, we then make the following choice for the parameter $\alpha$, depending on $\lambda$. 

\begin{assumption}[parameters]\label{as:param}
	We assume that the constant $\lambda \geq 0$ in~\ref{itm:case-expl-noncoerc} and in~\ref{itm:case-impl-coerc}, respectively, satisfies 
	\begin{align}\label{est:lambda}
		\lambda < \frac{2 \nu}{c_{\mathrm{tr,K}}^2} \eqqcolon \overline c. 
	\end{align}
	Furthermore, we choose the parameter $\alpha>0$ in the discretisation~\eqref{eq:unst-discr} sufficiently large that
	\begin{align}\label{est:alpha}
		\alpha >  
		\frac{2\lambda + \frac{1}{2}\overline c (1+\theta)^2 d c_{\tr}^2}{\overline c- \lambda}
	\end{align}
\end{assumption}

In the special case of monotone (coercive) relations in~\ref{itm:case-impl-coerc} with $\lambda = 0$, condition~\eqref{est:lambda} is trivially satisfied. For~\eqref{est:alpha} to be satisfied in this case, it suffices to choose $\alpha$ such that 
\begin{align}\label{est:alpha-mon}
	\alpha>	\tfrac{1}{2}(1+\theta)^2 d c_{\mathrm{tr}}^2. 
\end{align} 
Note that~\eqref{est:alpha-mon} is trivially satisfied for any $\alpha>0$ whenever $\theta=-1$, i.e., for the skew-symmetric Nitsche method.

\begin{assumption}[general setting]\label{assump:general}
    For $T>0$ let $I = (0,T)$, and let $\Omega \subset \setR^d$ for $d \in \{2,3\}$, be a bounded and open Lipschitz domain with polyhedral boundary $\Gamma$. 	
	Let $\nu>0$ and $\beta \geq  0$ be given constants. 
	Let $ \bf \in L^2(I;B)$ and $ \bu_0 \in \Bdiv$, with spaces $B,\Bdiv$ in~\eqref{def:sp-B}, or in~\eqref{def:sp-B-v}, respectively. 
	Let
    \begin{align*}
    r \in \begin{cases} \;\;
    [1,2] &\text{ in case } \ref{itm:case-expl-noncoerc}, \quad \text{see~Assumption~\ref{assump:s-expl}},\\
 \;\; [1,\infty) &\text{ in case } \ref{itm:case-impl-coerc},  \quad \text{see~Assumption~\ref{assump:gbd-mon}},
    \end{cases}
    \end{align*}
    and let $\lambda\geq 0$, $\alpha>0$ be such that Assumption~\ref{as:param} holds. 
	Assume that $(\tria)_{h>0}$ is a sequence of shape-regular triangulations of $\Omega$ and let $(X_h,Q_h)$ be pairs of mixed finite element spaces satisfying Assumption~\ref{assump:fem}. 
	Further, for $\delta>0$  let $J_{\delta} = \{I_1, \ldots, I_{m_{\delta}}\}$ be uniform partitions of $I$. 
	Furthermore, in case~\ref{itm:case-impl-coerc} let Assumption~\ref{assump:gbd-reg}~\ref{itm:gbd-eps-0}--\ref{itm:gbd-eps-approx} be satisfied.
\end{assumption}

\begin{theorem}[convergence result for $r\leq 2$]\label{thm:main-unsteady}
In the general setting of Assumption~\ref{assump:general} let $r \in [1,2]$. 
    
Then, for each $\oldepsilon = (\varepsilon, h,\delta)\in \setR^{3}_{>0}$ there exist (approximate) solutions
	\begin{align}
		\bu^{\oldepsilon} \in \mathcal{L}^0_0(J_{\delta};X_{h,\divergence}), 
		\quad 
		\pi^{\oldepsilon} \in \mathcal{L}^0_0(J_{\delta};Q_{h}),
		\quad \text{ and } \quad 
		\bsigma^\discr 
		\in \mathcal{L}^0_0(J_{\delta};L^{\min(2,r')}(\Gamma)^d),
	\end{align}
	to the Nitsche method~\eqref{eq:unst-discr}, with $ \bu^\oldepsilon_0 = \bu_0^h \in \Bh$ as in~\eqref{def:u0-discr}, and $\bf^\delta \in \mathcal{L}^0_0(J_{\delta};B)$ as in~\eqref{def:f-discr}. 
	There exists functions $\bu, \bsigma$ and a (non-relabelled) subsequence such that 
		\begin{alignat*}{5}
			\bu^\oldepsilon 
			&\to  \bu \quad 
			&&\text{ strongly in } L^q(I;B) \text{ and } L^{s}(I;L^4(\Omega)^d), &&\text{ for any } q \in [1,\infty), s \in [1,8/d),
			\\
			\bu^\oldepsilon 
			&\overset{(*)}{\rightharpoonup}  \bu \quad 
			&&\text{ weakly* in } L^\infty(I;B) \text{ and weakly in } L^2(I;W),\\
			\tr (\bu^\oldepsilon) \cdot \bn &\to \b0 \qquad &&\text{ strongly in } L^2(I;L^2(\Gamma)^d), \\
			\tr_\tau (\bu^\oldepsilon) &\to \tr_\tau(\bu) \qquad &&\text{ strongly in } L^{2}(I;L^{p}(\Gamma)^d) \cap L^r(I;L^r(\Omega)^d),  \quad &&\text{ for any  } p \in [1, 2^\sharp),\\
			\bsigma^\discr &\wconv \bsigma \qquad &&\text{ weakly in } L^{p}(I;L^{p}(\Gamma)^d), && \text{ for } 	p = \min(2,r')
		\end{alignat*}	
	as $\oldepsilon = (\varepsilon, h, \delta) \to (0,0,0)$. 
	Furthermore, the couple $(\bu, \bsigma)$ is a weak solution to~\eqref{eq:NS-unst-full} in the sense of Definition~\ref{def:unst-w-sol}, and hence a weak solution exists.  
\end{theorem}

The proof of Theorem~\ref{thm:main-unsteady} is a consequence of Lemmas~\ref{lem:apriori-unst}--\ref{lem:unst-conv-s-c2} to be stated and proved below. 
Extensions to $r>2$ are presented in Section~\ref{sec:extensions_rgeq2}.

\begin{remark}[Assumptions on $\Gamma$ and extension to mixed boundary conditions]\label{rmk:mixed-bc}Thanks to Corollary~\ref{cor:polyhedral}, 
	Assumption~\ref{assump:dom} holds for~$\Gamma = \partial \Omega$ for a polyhedral domain~$\Omega$ as in Theorem~\ref{thm:main-unsteady}. 
	Hence, Korn's inequality is available without extra assumption. 
	
	The convergence result can be extended to mixed boundary conditions, where slip-type boundary conditions~\eqref{eq:NS-unst-bc-1}--\eqref{eq:NS-unst-bc-2} are imposed only on part of the boundary $\Gamma \subset \partial \Omega$. 
	On other parts of the boundary one may impose
	\begin{itemize} 
		\item periodic boundary conditions; if this is the only other boundary condition  Assumption~\ref{assump:dom} is assumed to be satisfied for $\Gamma \subset\partial \Omega$;
		\item Dirichlet boundary conditions; if the corresponding part of the boundary has full $\mathcal{H}^{d-1}$-measure, then Korn's inequality is available without requiring Assumption~\ref{assump:dom};
    \item natural boundary conditions when dropping the convective term; 
      if the convective term is included, our analysis will most likely still apply, but additional care needs to be taken to ensure coercivity, 
    for example via directional \emph{do-nothing boundary conditions} as in~\cite{BraackMucha2014}. 
	\end{itemize}
	In our framework also generalised Robin boundary conditions and dynamic versions thereof can be handled, where the full trace $\tr(\bu)$ and $\BS \bn$ are related implicitly.  
		This replaces~\eqref{eq:NS-unst-bc-1}, \eqref{eq:NS-unst-bc-2}, and without the impermeability condition~\eqref{eq:NS-unst-bc-1} a Nitsche penalisation is not needed. 
\end{remark}

In particular, we note that the arguments in the convergence proof include plain convergence for methods on polyhedral domains with strongly imposed impermeability condition. 

\begin{remark}[pressure formulation]
	The pressure is a distribution in time with $\pi \in \mathcal{D}'(\overline I;L^2(\Omega))$ and $t \mapsto \int_0^t \pi(s) \ds  \in L^{\frac{8}{d+4}}(I;L^2(\Omega))$. 
	In the convergence proof,  the discrete pressure functions will be shown to converge in a suitable sense, and in the limit one has 
	\begin{equation}\label{eq:unst-lim-eq-1}
		\begin{aligned}
			- \int_0^T ( \bu, \bv)_{B}\, \partial_t \phi \dt  
			&+ 2 \nu \int_{0}^T \skp{\BD \bu}{ \BD \bv} \phi \dt 
			+ \int_0^T b(\bu,\bu,\bv) \phi \dt 
			+ \int_0^T \skp{\pi}{\divergence \bv}  
			{\phi} \dt 
			\\
			& + \int_{0}^T \skp{\bsigma}{\tr_\tau(\bv)}_{\Gamma} \, \phi \dt  
            =
			\int_{0}^T ( \bf, {\bv})_B \,\phi  \dt. 
		\end{aligned}
	\end{equation}
	for any $\bv \in C^{\infty}(\overline \Omega)^d$, with $\tr(\bv) \cdot \bn  = 0$ on $\Gamma$, and $\phi \in C^{\infty}(\overline I)$ with $\phi(T) = \phi(0) =  0$. 
	If one identifies $B \equiv B'$, and defines the duality pairing between $W'$ and $W$, the first term may be rewritten, cf.~\cite{BulicekMalekMaringova2023}, as 
	\begin{align*}
		- \int_0^T ( \bu, \bv)_{B}\, \partial_t \phi \dt   = 	 \int_0^T \skp{\partial_t  \bu}{ \bv}_{W}\,  \phi \dt.
	\end{align*}
\end{remark}

The convergence proof proceeds along the following standard strategy: 
Starting from a~priori estimates we obtain weakly converging subsequences. 
By means of compactness even strongly converging subsequences may be extracted. 
Then, the limiting equation is deduced for the limiting functions, and in the end one has to identify the nonlinear terms. 
In case $r\leq 2$ we have that the tangential traces converge strongly in $L^r(I\times \Gamma)$, and thus the identification is simpler than in case $r>2$. 
To avoid repetition the a priori estimates and convergence results in Sections~\ref{sec:unst-apriori}--\ref{sec:unst-limit} are formulated for the largest possible range of the exponent~$r$, see Assumption~\ref{assump:general}. 
Section~\ref{sec:unst-id} contains the identification of the nonlinear terms in the limit for $r \in [1,2]$. 
The corresponding extensions for $r>2$ are presented in Section~\ref{sec:extensions_rgeq2}. 

\subsection{A priori estimates}
\label{sec:unst-apriori}	

We consider sequences $(\epsilon_k)_{k \in \mathbb{N}}$, $(h_k)_{k \in \mathbb{N}}$, and $(\delta_k)_{k \in \mathbb{N}}$ such that 
\begin{align*}
	\discr_k \coloneqq(\epsilon_k, h_k, \delta_k) \to 0 \qquad \text{ as } k \to \infty. 
\end{align*}
We will consider several subsequences, but shall not indicate this in the notation. 
For the sequences of approximate solutions to~\eqref{eq:unst-discr} we denote for 	$j \in \{1, \ldots, m_k\}$ and  $k \in \mathbb{N}$
\begin{equation} \label{def:discr-short}
	\begin{aligned}
		\bu^k_0 &\coloneqq  \bu^{h_k}_0 
		=  \Pi_{h_k}  \bu_0,    
		& \qquad 	 \bf^k_j &\coloneqq  \bf^{\delta_k}_j  =  \Pi_{\delta_k}  \bf |_{I_j},\\
		\bu^k_j 
		&\coloneqq \bu^{\discr_k}_j
		= 
		\bu^{(\epsilon_k,h_k,\delta_k)}_j, 
		& 
		\pi^k_j
		&\coloneqq \pi^{\discr_k}_j
		= 
		\pi^{(\epsilon_k,h_k,\delta_k)}_j,
		\\
		\bsigma^{k}_j 
		&\coloneqq \bsigma^{\discr_k}_j,
		\qquad  \text{ and } &	\widehat \bsigma^{k}_j 
		&\coloneqq \Seps(\tr_\tau(\bu^k_j)) \;\; \text{ in case \ref{itm:case-impl-coerc},}
	\end{aligned}
\end{equation}
and analogously the right-continuous, piecewise constant interpolations are $\bu^k, \pi^k, \bf^k$, $\bsigma^k$ and $\widehat \bsigma^k$. For notational brevity, 
we sometimes write  $h$ and $\delta$ instead of $h_k$ and $\delta_{k}$ in what follows.

\begin{lemma}[a~priori estimates]\label{lem:apriori-unst}
In the situation of Assumption~\ref{assump:general} there exists a constant $c>0$ depending only on 
{$\norm{ \bf}_{L^2(I;B)}^2$}, $\norm{ \bu_0}_B$, and on the constants in Assumption~\ref{assump:gbd-reg}~\ref{itm:gbd-eps-q} such that if $(\bu_j^k)_{j=1}^{m_k}$ solves \eqref{eq:unst-discr-a}, then the a~priori estimate
	\begin{align}\label{eq:Tabowser}
		\max_{j \in \{ 1, \ldots, m_k \}} \norm{{\bu}^k_j}_B^2  + \sum_{j = 1}^{m_k} \norm{{\bu}^k_j - {\bu}^k_{j-1}}_B^2 
		+  \delta_k \sum_{j = 1}^{m_k}  \norm{\bu_j^k}_{X_{h_k}}^2 
		& \leq c,
	\end{align}
	holds for all $k \in \mathbb{N}$. 
	Moreover, we have the following uniform estimates on the piecewise constant interpolants:
	\begin{align} \label{est:ap-pwconst}
		\begin{split}
			\norm{ \bu^k}_{L^\infty(I;B)}  
			+ 	\norm{ \bu^k}_{L^2(I;X_{h_k})}
			&\leq c,\\
			\norm{ \bu^k}_{L^\infty(I;L^2(\Omega))} 
			+ 	\norm{ \bu^k}_{L^2(I;H^1(\Omega))}
			+	\norm{\bu^k}_{L^{\frac{8}{d}}(I;L^4(\Omega))}
			+ \norm{\tr(\bu^k)}_{L^2(I;L^p(\Gamma))}  
			&\leq c,
		\end{split}
	\end{align}
	
	for any $k \in \mathbb{N}$ for any $p \in [1, \infty)$ with $p \leq 2^\sharp$. 
	Furthermore, we have 
	\begin{alignat*}{2}
		\norm{\bsigma^k}_{L^{2}(I;L^{2}(\Gamma))}  & \leq c \qquad && \text{ in case~\ref{itm:case-expl-noncoerc}, } 
		\\
		\norm{\tr_{\tau}(\bu^k)}_{L^{r}(I;L^{r}(\Gamma))} + 		\norm{\widehat{\bsigma}^k}_{L^{r'}(I;L^{r'}(\Gamma))}  & \leq c \qquad && \text{ in case \ref{itm:case-impl-coerc}, }
	\end{alignat*}
	for any $k \in \mathbb{N}$. 
\end{lemma}
\begin{proof}
	
	Testing~\eqref{eq:unst-discr} with $(\bu^{k}_j,\pi_j^{k}) \in \Xh\times Q_h$ and using the notation in~\eqref{def:discr-short} yields
	\begin{equation}\label{est:unst-ap-1}
		\begin{aligned}
			({\difft {\bu}^k_j},{\bu}^k_j)_B
			&+  2 \nu \norm{\Du^k_j}_{L^2(\Omega)}^2 
			- 2\nu(1+\theta) \int_{\Gamma} ((\Du^k_j \, \bn)\cdot \bn) (\bu^k_j \cdot \bn) \ds \\
			& 
			+ \skp{\bsigma_{j}^k
			}{\tr_{\tau}(\bu^k_j) }_{\Gamma}
			+ \nu \alpha  \int_{\Gamma}h_{\Gamma}^{-1} \abs{\bu^k_j \cdot \bn}^2 \ds = ({\bf}_j^k,{\bu}^k_j)_B,
		\end{aligned}
	\end{equation}
	thanks to the skew-symmetry of the numerical convective term $\widetilde{b}$, cf.~\eqref{def:conv-term-num}.  
	Note that all terms containing the pressure vanish. 
	By~\eqref{eq:difft} we find that 
	\begin{align}\label{est:unst-ap-2}
		(\difft{\bu}^k_j,{\bu}^k_j)_B 
		= \frac{1}{2\delta_{k}} \left(\norm{{\bu}^k_j}_B^2 - \norm{{\bu}^k_{j-1}}_B^2 + \norm{{\bu}^k_j - {\bu}^k_{j-1}}_B^2 \right). 
	\end{align} 
	By a standard inverse trace estimate with constant  $c_{\mathrm{tr}}>0$, we obtain 
	\begin{equation}
		\begin{aligned}\label{est:unst-ap-3}
			\abs{\int_{\Gamma} (\Du^k_j \, \bn) \cdot \bn (\bu_j^k \cdot\bn) \ds }  
			& \leq  
			\norm{h_\Gamma^{1/2} \Du^k_j }_{L^2(\Gamma)} \norm{h_{\Gamma}^{-1/2} \bu^k_j \cdot \bn}_{L^2(\Gamma)} \\ 
			& \leq d^{\frac{1}{2}}c_{\mathrm{tr}}\norm{\Du^k_j }_{L^2(\Omega)} \norm{h_{\Gamma}^{-1/2} \bu^k_j \cdot \bn}_{L^2(\Gamma)},
		\end{aligned}
	\end{equation}
	where the factor $d$ arises as upper bound on the number of boundary faces adjacent to one $d$-simplex. We record from~\cite[Ex.\ 37.2]{EG.2021.2} that 
	\begin{align}\label{eq:Tabinatora}
		\xi^{2} - 2\beta \xi\eta + \omega\eta^{2} \geq \frac{\omega-\beta^{2}}{1+\omega}(\xi^{2}+\eta^{2})\qquad\text{for all}\;\xi,\eta,\beta,\omega\geq 0. 
	\end{align}
	Applying this inequality with the particular choice
	\begin{align*}
		\xi=\norm{\Du^k_j }_{L^2(\Omega)},\;\;\eta=\norm{h_{\Gamma}^{-1/2} \bu^k_j \cdot \bn}_{L^2(\Gamma)},\;\;\omega=\frac{\alpha}{2}\;\;\text{and}\;\;\beta=\tfrac{1}{2}(1+\theta)d^{\frac{1}{2}}c_{\mathrm{tr}}
	\end{align*}
	and recalling the definition of the norm $\norm{\cdot}_{X_h}$ in~\eqref{def:norm-h} gives us
	\begin{equation}
		\begin{aligned}
			\label{est:unst-ap-4}
			\nu & \left(	2\norm{\Du^k_j}_{L^2(\Omega)}^2    - 2(1+\theta) \int_{\Gamma } ((\Du^k_j \, \bn)\cdot \bn) (\bu^k_j \cdot \bn) + \alpha \int_{\Gamma}h_{\Gamma}^{-1} \abs{\bu^k_j \cdot \bn}^2 \ds\right) \\ & \stackrel{\eqref{est:unst-ap-3}}{\geq}  2\nu\Big(\norm{\Du^k_j }_{L^2(\Omega)}^{2}- {(1+\theta)}d^{\frac{1}{2}}c_{\mathrm{tr}}\norm{\Du^k_j }_{L^2(\Omega)} \norm{h_{\Gamma}^{-1/2} \bu^k_j \cdot \bn}_{L^2(\Gamma)} + \frac{\alpha}{2}\norm{h_{\Gamma}^{-1/2} \bu^k_j \cdot \bn}_{L^2(\Gamma)} ^{2}\Big)\\ 
			& \stackrel{\eqref{eq:Tabinatora}}{\geq} 
			2 \nu { \frac{\alpha - \frac{1}{2}(1+\theta)^2dc_{\mathrm{tr}}^2 }{2 + \alpha} }
			\left( \norm{\Du^k_j}_{L^2(\Omega)}^2 + \norm{h_{\Gamma}^{-1/2}\bu^k_j \cdot \bn}_{L^2(\Gamma)}^2
			\right) \\ &
			\stackrel{\eqref{def:norm-h}}{=}  2 \nu { \frac{\alpha -\frac{1}{2}(1+\theta)^2dc_{\mathrm{tr}}^2 }{2 + \alpha} }
			\norm{\bu^k_j}_{X_h}^2. 
		\end{aligned}
	\end{equation}
	In case~\ref{itm:case-expl-noncoerc}, by the boundedness of $\Srel$ in Assumption~\ref{assump:s-expl} with $r \leq 2$, see~\eqref{est:case1-bd}, we have that
	\begin{align}\label{est:unst-ap-5-c1}
		\begin{split}
			\skp{\bsigma^k_j}{\tr_{\tau}(\bu^k_j)}_{\Gamma}	& \stackrel{\eqref{eq:unst-discr-b}}{=}
			\skp{\Srel(\tr_{\tau}(\bu^k_j))}{\tr_{\tau}(\bu^k_j)}_{\Gamma} \\ & 
			\;\;\geq - \abs{\skp{\Srel(\tr_{\tau}(\bu^k_j))}{\tr_{\tau}(\bu^k_j)}_{\Gamma} }
			\geq - \lambda \left(c +   \norm{\tr_{\tau}(\bu^k_j)}_{L^2(\Gamma)}^2\right).  \end{split}
	\end{align}
	In case~\ref{itm:case-impl-coerc}, using the coercivity of $\Seps$ due to Assumption~\ref{assump:gbd-reg}, we obtain 
	\begin{align}\label{est:unst-ap-5-c2}
		\begin{split}
			\skp{\bsigma^k_j}{\tr_{\tau}(\bu^k_j)}_{\Gamma} & \stackrel{\eqref{eq:unst-discr-b}}{=} 
			\skp{\Seps(\tr_{\tau}(\bu^k_j)) - \lambda \tr_{\tau}(\bu^k_j) }{\tr_{\tau}(\bu^k_j)}_{\Gamma} \\ & 
			\;\;\geq  - c -  \lambda \norm{\tr_{\tau}(\bu^k_j)}_{L^2(\Gamma)}^2.
		\end{split}
	\end{align}
	In both cases, using the trace and Korn's inequality~\eqref{est:tr-K} we find that 
	\begin{align}\label{est:unst-ap-5}
		\skp{\bsigma^k_j}{\tr_{\tau}(\bu^k_j)}_{\Gamma}
		\geq 	
		- c \,\max(\lambda,1)	- \lambda \norm{\tr_\tau (\bu_{j}^k)}_{L^2(\Gamma)}^2 
		\geq 
		- c\,\max(\lambda,1)	- \lambda c_{\mathrm{tr},K}^2 \norm{ \bu_j^k}_{X_h}^2. 
	\end{align}	
	Combining~\eqref{est:unst-ap-1}--\eqref{est:unst-ap-5} and employing the conditions~\eqref{est:lambda} and~\eqref{est:alpha} on $\lambda$ and $\alpha$, there exist constants $c, C>0$ such that 
	\begin{align}\label{est:unst-ap-6}
		\frac{1}{2\delta_{k}} \left(\norm{{\bu}^k_j}_B^2 - \norm{{\bu}^k_{j-1}}_B^2 + \norm{{\bu}^k_j - {\bu}^k_{j-1}}_B^2 \right)
		+ 
		c \norm{ \bu_j^k}_{X_h}^2 
		\leq
		( \bf_j^k, \bu_j^k)_B + C.
	\end{align}
	Estimating the right-hand side of~\eqref{est:unst-ap-6} by use of the Cauchy--Schwarz inequality, employing the estimate~\eqref{est:norm-B-h} and Young's inequality with $\rho>0$, and we arrive at 	
	\begin{equation}\label{est:unst-ap-7}
		\begin{aligned}
			({\bf}^k_j,{\bu}^k_j)_B 
			&  \leq 
			\norm{ \bf_j^k}_{B} \norm{  \bu_j^k}_{B}	
			\leq 
			c \norm{ \bf_j^k}_{B} \norm{ \bu_j^k}_{X_h}
			\leq 
			c(\rho) \norm{ \bf_j^k}_{B}^2 +  \rho \norm{ \bu_j^k}_{X_h}^2. 
		\end{aligned}
	\end{equation}
	For sufficiently small $\rho>0$, this allows us to absorb the $\norm{\cdot}_{X_h}$-term on the right-hand side of~\eqref{est:unst-ap-7} into the left-hand side of~\eqref{est:unst-ap-6}, i.e., we have
	\begin{align}\label{est:unst-ap-8}
		\tfrac{1}{2\delta_k} \left(\norm{{\bu}^k_j}_B^2 - \norm{{\bu}^k_{j-1}}_B^2 
		+ \norm{{\bu}^k_j - {\bu}^k_{j-1}}_B^2 \right) 
		+ 	c\norm{ \bu_j^k}_{X_h}^2 
		\leq C \left(\norm{ \bf_j^k}_{B}^2 + 1\right),
	\end{align}
	uniformly in $k$.  
	Multiplying by $\delta_k$ and summing over $j\in\{ 1, \ldots, i\}$ yields for an arbitrary $i  \in \{1, \ldots, m_k\}$ 
	\begin{align}\label{est:unst-ap-9}
		\norm{{\bu}^k_i}_B^2 
		+ \sum_{j = 1}^i \norm{{\bu}^k_j - {\bu}^k_{j-1}}_B^2 
		+ 	c  \delta_k \sum_{j = 1}^i  \norm{ \bu_j^k}_{X_h}^2 
		\leq C \left(
		\delta_k \sum_{j = 1}^{m_k} \norm{ \bf_j^k}_{B}^2  + \norm{{\bu}^k_{0}}_B^2 + 1 \right). 
	\end{align}
	Applying the $L^2(I;B)$-stability of $\Pi_{\delta} \colon L^2(I;B) \to \mathcal{L}^0_0(I_\delta;B)$ from  see~\eqref{est:L2-stab-f}, and the~$B$-stability of $\Pi_{h_k} \colon B \to \Bh$ from~\eqref{est:PiBh-stab-u0}, we arrive at the  a priori-estimate 
	\begin{equation}\label{est:unst-ap-10}
		\begin{aligned} 
			\max_{j = 1, \ldots, m_k} \norm{{\bu}^k_j}_B^2  + \sum_{j = 1}^{m_k} \norm{{\bu}^k_j - {\bu}^k_{j-1}}_B^2 
			+ 	c  \delta_k \sum_{j = 1}^{m_k}  \norm{ \bu_j^k}_{X_h}^2  
			& \leq 	 C \left( \norm{ \bf}_{L^2(I;B)}^2 + \norm{ \bu_0}_B^2 + 1 \right),
		\end{aligned} 
	\end{equation}
	with constants independent of $k$. 
	This establishes~\eqref{eq:Tabowser}. 
	
	By virtue of Lemma~\ref{lem:h-norm}, ~\eqref{est:unst-ap-10} particularly yields that the sequence $({\bu^k})_{k \in \mathbb{N}}$ is bounded in the space $L^\infty(I;L^2(\Omega)^d) \cap L^2(I;H^1(\Omega)^d)$. Thus, 
	by interpolation via~\eqref{est:interp}, we obtain uniform bounds on the sequence $(\bu^{k})_{k \in \mathbb{N}}$ in $L^{\frac{8}{d}}(I;L^4(\Omega)^d)$ provided that $d \leq 3$. 
	Moreover, based on the bounds on $\bu^k$ in $L^2(I;H^1(\Omega)^d)$ and the trace inequality~\eqref{est:trace}, there exists a constant $c>0$ independent of $k$ such that 
	\begin{align}\label{est:unst-ap-11}
		\norm{\tr(\bu^k)}_{L^2(I;L^p(\Gamma))} 
		\leq c	\norm{\bu^k}_{L^2(I;H^1(\Omega))}
		\leq c
	\end{align} 
	holds for any $p \in [1,\infty)$ with $p \leq 2^\sharp$, uniformly in $k$. 
	
	In case~\ref{itm:case-expl-noncoerc} we obtain by~\eqref{est:case1-bd2} and the bound on $\tr_{\tau}(\bu^k)$ in~\eqref{est:unst-ap-11} with $2 < 2^\sharp$ that 
	\begin{align}\label{est:unst-ap-12}
		\norm{\bsigma^k}_{L^2(I;L^2(\Gamma))} 
		\leq c
	\end{align} 
	uniformly in $k \in \mathbb{N}$. 
    In case~\ref{itm:case-impl-coerc} using the coercivity 
	of $\Seps$ due to Assumption~\ref{assump:gbd-reg}~\ref{itm:gbd-eps-q} we obtain 
	\begin{alignat}{3}
		\label{est:unst-ap-5a}
		c\left(	 \norm{\Seps(\tr_{\tau}(\bu_{j}^k))}_{L^{r'}(\Gamma)}^{r'} 
		+ \norm{\tr_\tau (\bu_{j}^k)}_{L^r(\Gamma)}^r
		- 1 \right)
		& \leq 	\skp{\Seps(\tr_{\tau}(\bu_{j}^k)) }{\tr_{\tau}(\bu_{j}^k)}_{\Gamma}
		\qquad && \text{ if } r > 1,\\
		\label{est:unst-ap-5b}		 
        c\left(\norm{\tr_\tau (\bu_{j}^k)}_{L^r(\Gamma)}^r - 1 \right) &\leq 
			\skp{\Seps(\tr_{\tau}(\bu_{j}^k)) }{\tr_{\tau}(\bu_{j}^k)}_{\Gamma} 
            & &\text{ and }
            \\
			\norm{\Seps(\tr_{\tau}(\bu_{j}^k))}_{L^{\infty}(\Gamma)}& \leq c 
		\qquad &&\text{ if } r  =  1.
	\end{alignat} 
	This means, that~\eqref{est:unst-ap-6} also holds with	 $\norm{\tr_{\tau}(\bu^k_j)}_{L^{r}}^{r}   + \norm{\widehat{\bsigma}^k_j}_{L^{r'}}^{r'}$ on the left-hand side, and hence we have
	\begin{align}\label{est:unst-ap-13}
		\norm{\tr_{\tau}(\bu^k) }_{L^{r}(I;L^{r}(\Gamma))} + 	\norm{\widehat \bsigma^k}_{L^{r'}(I;L^{r'}(\Gamma))}  \leq c,
	\end{align}
	with constant independent of $k$. 
	This proves the claim.
\end{proof}

\begin{lemma}[existence of approximate solutions]\label{lem:ex} 
	Under the conditions of Assumption~\ref{assump:general} for each $k \in \mathbb{N}$ there exists a sequence of pairs $( \bu^k_j, \pi^k_j)_{j \in \{0, \ldots, m_k\}} \subset   X_{h_k,\divergence} \times Q_{h_k}$  solving~\eqref{eq:unst-discr}.  
\end{lemma}
\begin{proof}
	Testing with $(0,q_{h_k})$ for $q_{h_k} \in Q_{h_k}$ arbitrary in~\eqref{eq:unst-discr} yields that $\bu_j^k \in X_{h_k,\divergence}$ for any $j \in \{1, \ldots, m_k\}$.
  The existence of $\bu^k_j$ follows from the a priori estimates and a corollary of Brouwer's
	fixed point theorem, see~\cite[§~5.7, (G.7), p.~104]{Granas2003}. 
	The existence of $\pi^k_j$ then is a routine  consequence of the inf-sup condition. 
\end{proof}
Next, we give a slightly improved estimate on the time increment: 

\begin{lemma}[time increment]\label{lem:dt-est} 
Under the conditions of Assumption~\ref{assump:general} with $ \norm{\cdot}_{Y_{h_k}}$ as in~\eqref{def:Yh} there exists a constant $c>0$ such that 
	\begin{align}\label{def:sigma-k}
		\delta_k\sum_{j = 1}^{m_k} \norm{\difft {\bu}^k_j}_{Y_{h_k}}  
		\leq c \qquad \text{ for all } k \in \mathbb{N}. 
	\end{align}
\end{lemma}
\begin{proof}
	We divide the proof into two main steps.
	
	\textit{1. Step (preliminary estimates)}: 
	Let $j\in\{1,...,m_{k}\}$. 
	Let us first estimate 
	\begin{equation}
		\begin{aligned}\label{def:z-form}
			\wz(\bu^k_j, \bv_h) \coloneqq & 
			- 2\nu \skp{ \BD \bu^{k}_j}{ \BD \bv_h } 
			- \widetilde{b}(\bu^{k}_j, \bu^{k}_j, \bv_h)
			\\
			& 
			- \skp{\bsigma^k_j}{\tr_{\tau}(\bv_h) }_{\Gamma}
			-   \nu \alpha \skp{ h_{\Gamma}^{-1} \bu^{k}_j \cdot \bn}{  \bv_h \cdot \bn}_{\Gamma}
			\\ 
			& 
			+ \skp{ \left( (2\nu  \BD \bu_j^{k}) \bn \right) \cdot \bn }{ \bv_h \cdot \bn }_{\Gamma}
			+ {\theta}\skp{   \left((2 \nu \BD \bv_h) \bn \right) \cdot \bn }{ \bu_j^{k} \cdot \bn }_{\Gamma}
			+ ({\bf}_j^{k} ,{\bv}_h )_B \\ 
			 \eqqcolon & \mathrm{I}_{1}+...+\mathrm{I}_{7},
		\end{aligned}
	\end{equation}
	representing all but the time derivative and the pressure terms in the discrete equation~\eqref{eq:unst-discr} when tested with $(\bv_h, 0) \in (X_{d,\divergence},Q_h)$. 
	On the first term $\mathrm{I}_{1}$ we have with Lemma~\ref{lem:h-norm} that
	\begin{align}\label{eq:unst-incr-2}
		2 \nu \abs{\skp{ \BD \bu^{k}_j}{ \BD \bv_h }} 
		\lesssim 
		\norm{\bu_j^k}_{H^1(\Omega)} \norm{\bv_h}_{H^1(\Omega)} 
		\lesssim 
		\norm{ \bu_j^k}_{X_h} \norm{ \bv_h}_{X_h}. 
	\end{align}
	By~\eqref{eq:traccovanish} and Lemma~\ref{lem:h-norm}, the numerical convective term~\eqref{def:conv-term-num} in $\mathrm{I}_{2}$   can be estimated as
	\begin{equation}\label{eq:unst-incr-3}
		\begin{aligned}
			\abs{ \widetilde b(\bu^k_j,\bu^k_j,\bv_h)}
			& \lesssim 
			\norm{ \bu^k_j}_{H^1(\Omega)}^2
			\norm{ \bv_h}_{H^1(\Omega)}
			\lesssim 	\norm{ \bu^k_j}_{H^1(\Omega)}^2
			\norm{ \bv_h}_{X_h}. 
		\end{aligned}
	\end{equation}
	In case~\ref{itm:case-expl-noncoerc} using Hölder's inequality and estimate~\eqref{est:trace-h} in Lemma~\ref{lem:h-norm} with $p = 2 \leq 2^\sharp$ we estimate  term $\mathrm{I}_{3}$ as 
		\begin{align}\label{eq:unst-incr-4-c1}
			\abs{\skp{\bsigma^k_j}{\tr_{\tau}(\bv_h) }_{\Gamma}} 
			\leq  \norm{\bsigma^k_j}_{L^2(\Gamma)} \norm{ \bv_h}_{L^2(\Gamma)}
			\lesssim \norm{\bsigma^k_j}_{L^2(\Gamma)} \norm{ \bv_h}_{X_h}. 
		\end{align}
		On the other hand, in case~\ref{itm:case-impl-coerc}, we have $\bsigma_j^k = \widehat \bsigma^k_j - \lambda \tr_{\tau} (\bu_j^k) $, and hence term  $\mathrm{I}_{3}$ is estimated with the same arguments as follows
		\begin{equation}
			\begin{aligned}
				\label{eq:unst-incr-4-c2}
				\abs{\skp{\bsigma^k_j}{\tr_{\tau}(\bv_h) }_{\Gamma}} 
				&\leq  \norm{\widehat \bsigma^k_j}_{L^{r'}(\Gamma)} \norm{ \bv_h}_{L^{r}(\Gamma)} + \lambda 
				\norm{ \tr_{\tau}(\bu^k_j)}_{L^2(\Gamma)} \norm{ \bv_h}_{L^2(\Gamma)}\\
				&\lesssim \left(\norm{\widehat \bsigma^k_j}_{L^{r'}(\Gamma)}  +  
				\norm{\bu^k_j}_{X_h} \right) 
				\norm{ \bv_h}_{X_h}. 
			\end{aligned}
		\end{equation}
		Denoting 
		\begin{align}\label{eq:def-I3k}
			(\mathrm{I}_{3})^k_j \coloneqq \begin{cases}
				\norm{\bsigma^k_j}_{L^2(\Gamma)} \quad &\text{ in case } \ref{itm:case-expl-noncoerc},\\
				\norm{\widehat \bsigma^k_j}_{L^{r'}(\Gamma)}  \quad &\text{ in case } \ref{itm:case-impl-coerc},
			\end{cases}
		\end{align}
		we may summarize both estimates~\eqref{eq:unst-incr-4-c1} and~\eqref{eq:unst-incr-4-c2} as 
		\begin{align}\label{eq:unst-incr-4}
			\abs{\skp{\bsigma^k_j}{\tr_{\tau}(\bv_h) }_{\Gamma}} 
			\lesssim 
			\left(  (\mathrm{I}_{3})^k_j +  \norm{\bu^k_j}_{X_h}  \right) 
			\norm{ \bv_h}_{X_h}. 
		\end{align}
	For the next trace term $\mathrm{I}_{4}$ in~\eqref{def:z-form}, we employ the definition of $\norm{\cdot}_{X_h}$, see~\eqref{def:norm-h}, to find
	\begin{equation}
		\begin{aligned}\label{eq:unst-incr-6}
			\skp{h_{\Gamma}^{-1} \tr(\bu^{k}_j) \cdot \bn}{\bv_h\cdot \bn }_{\Gamma}
			&\leq 
			\norm{h_{\Gamma}^{-1/2}\bu^k_j \cdot \bn}_{L^2(\Gamma)} \norm{h_{\Gamma}^{-1/2} \bv_h \cdot \bn}_{L^2(\Gamma)} 
			&
			\leq 
			\norm{ \bu^k_j}_{X_h}
			\norm{ \bv_h}_{X_h}.
		\end{aligned}
	\end{equation}
	As to $\mathrm{I}_{5}$, we proceed as in~\eqref{est:unst-ap-3} and obtain 
	\begin{equation}
		\begin{aligned}\label{eq:unst-incr-7}
			\abs{\skp{ \left( (
					\BD \bu_j^{k}) \bn \right) \cdot \bn }{ \bv_h \cdot \bn }_{\Gamma} } 
			& \lesssim  
			\norm{ \BD \bu_j^{k} }_{L^2(\Omega)}
			\norm{h_{\Gamma}^{-1/2} \bv_h \cdot \bn }_{L^2(\Gamma)}
			\lesssim   
			\norm{ \bu_j^{k} }_{X_h}
			\norm{ \bv_h}_{X_h},
		\end{aligned}
	\end{equation}
	and the next term $\mathrm{I}_{6}$ is estimated by analogous means noting that $|\theta|\leq 1$. 
	Finally, as in~\eqref{est:unst-ap-7} we have 
	\begin{equation}
		\begin{aligned}\label{eq:unst-incr-8}
			\abs{({\bf}_j^{k} ,{\bv}_h )_B } 
			\lesssim  
			\norm{{\bf}_j^{k} }_{B}
			\norm{\bv_h}_{X_h}. 
		\end{aligned}
	\end{equation}
	for $\mathrm{I}_{7}$. 
	Altogether, \eqref{def:z-form}--\eqref{eq:unst-incr-8} yield that for any $\bv_h \in X_h$,  
	\begin{equation}\label{eq:unst-est-2}
		\begin{aligned}
			\abs{\wz(\bu^k_j, \bv_h)} & \lesssim  
			\left( 
			\norm{ \bu_j^k}_{X_h} 
			+
			\norm{\bu^k_j}_{H^1(\Omega)}^2 
			+ (\mathrm{I}_{3})^k_j
			+ \norm{{\bf}_j^{k} }_{B}
			\right) \norm{ \bv_h}_{X_h}, 
		\end{aligned}
	\end{equation}
	uniformly in $k$ and $h_k$. 
	
	\textit{2. Step (estimate of time increment):}
	Testing~\eqref{eq:unst-discr} with $(\bv_h,0)\in X_{h,\divergence} \times Q_h$ 	yields with $\wz$ as defined in~\eqref{def:z-form} that 
	\begin{equation}
		\label{eq:unst-incr-1}
		(\difft {\bu}^{k}_j,{\bv}_h)_B
		= \wz( \bu^k,\bv_h) \qquad \text{ for all } \bv_h \in  X_{h,\divergence}. 
	\end{equation}
	Applying the estimate~\eqref{eq:unst-est-2} on $\wz$ yields for any $ \bv_h \in \Bh$, see~\eqref{def:Bh}, that
	\begin{align}\label{eq:unst-incr-9}
		\frac{\abs{(\difft {\bu}^{k}_j,{\bv}_h)_B}}{\norm{ \bv_h}_{X_h}} 
		\lesssim 
		\norm{ \bu_j^k}_{X_h} 
		+ 
		\norm{ \bu^k_j}_{H^1(\Omega)}^2 
		+ (\mathrm{I}_{3})^k_j
		+ \norm{{\bf}_j^{k} }_{B}. 
	\end{align}
	Therefore, in the dual discrete norm $\norm{\cdot}_{Y_h}$ as defined in~\eqref{def:Yh}, we have the bound
	\begin{align}\label{eq:unst-incr-10}
		\norm{\difft {\bu}^{k}_j}_{Y_h} 
		\lesssim
		\norm{ \bu_j^k}_{X_h} 
		+ \norm{ \bu^k_j}_{H^1(\Omega)}^2 
		+ (\mathrm{I}_{3})^k_j		
		+ \norm{{\bf}_j^{k} }_{B}. 
	\end{align}
	Summing over $j \in\{1, \ldots, m_{k}\}$ and multiplying with $\delta_k$ yields 
	\begin{align*}
		\delta_k \sum_{j = 1}^{m_k}	\norm{\difft {\bu}^{k}_j}_{Y_h} 
		&\lesssim  
		\norm{ \bu^k}_{L^1(I;X_h)}
		+ \norm{\bu^k}_{L^2(I;H^1(\Omega))}^2 
		+(\mathrm{I}_{3})^k	
		+ \norm{ \bf^k}_{L^{1}(I;B)}. 
	\end{align*}
	where
		\begin{align}\label{eq:def-I3k-norm}
			(\mathrm{I}_{3})^k
			\coloneqq \begin{cases}
				\norm{\bsigma^k}_{L^1(I;L^2(\Gamma))} \quad &\text{ in case \ref{itm:case-expl-noncoerc}},\\
				\norm{\widehat \bsigma^k}_{L^1(I;L^{r'}(\Gamma))}  \quad &\text{ in case ~\ref{itm:case-impl-coerc}}.
			\end{cases}
	\end{align}
	Employing the estimates from Lemma~\ref{lem:apriori-unst} we find that the right-hand side is bounded uniformly in $k \in \mathbb{N}$. 
    This is~\eqref{def:sigma-k}, and the proof is complete. 
\end{proof}
We now come to estimates on the pressure functions, which are a consequence of refined bounds on the convective terms. 
\begin{lemma}[pressure estimates]\label{lem:p-est} 
Under the conditions of Assumption~\ref{assump:general} there  exists a constant $c>0$ such that 
	\begin{align}\label{def:p-k}
		\delta_k^2 \sum_{j = 1}^{m_k} \norm{\pi^k_j}_{L^2(\Omega)}^{8/(d+4)} 
		\leq c \qquad \text{ for all } k \in \mathbb{N}. 
	\end{align}
\end{lemma}
\begin{proof}	
Since $\bu^k_j \in X_{h,\divergence}$, we may rewrite~\eqref{eq:unst-discr} tested with $\bv_h \in X_h \cap H^1_0(\Omega) ^d$ and with $q_h = 0$ as
	\begin{align}\label{eq:est-press-1}
		\skp{ \pi_j^{k}}{ \diver \bv_h } 
		= -
		(\difft {\bu}^{k}_j,{\bv}_h)_B + z_2( \bu^k_j,\bv_h),
	\end{align}
	with
 \begin{equation}
		\begin{aligned}\label{def:z-form-2}
			\vz(\bu^k_j, \bv_h) &\coloneqq 
			- 2\nu \skp{ \BD \bu^{k}_j}{ \BD \bv_h } 
			- \widetilde{b}(\bu^{k}_j, \bu^{k}_j, \bv_h)
			+ 2 \nu \skp{   \left(( \BD \bv_h) \bn \right) \cdot \bn }{ \bu_j^{k} \cdot \bn }_{\Gamma}
			+ ({\bf}_j^{k} ,\bv_h )_B. 
		\end{aligned}
	\end{equation}
	Indeed, in view of~\eqref{eq:unst-discr}, we note that all except the  remaining boundary terms vanish, since $\bv_h  \in H^1_0(\Omega)^d$.
	In order to obtain $L^p$-estimates on the pressure with more than integrability  in time better than $L^1$ we have to resort to a the finer estimate on the convective term  in~\eqref{eq:traccovanish}. We obtain
	\begin{equation}\label{eq:est-press-2}
		\begin{aligned}
			\abs{ \widetilde b(\bu^k_j,\bu^k_j,\bv_h)}
			& \leq 
			\norm{\bu^k_j}_{L^4(\Omega)}^2
			\norm{\nabla \bv_h}_{L^2(\Omega)}
			+ \norm{\bu^k_j}_{L^4(\Omega)} 	
			\norm{\nabla \bu^k_j}_{L^2(\Omega)}
			\norm{\bv_h}_{L^4(\Omega)}\\
			& \lesssim 
			\norm{\bu^k_j}_{L^4(\Omega)} 	
			\norm{ \bu^k_j}_{H^1(\Omega)}
			\norm{ \bv_h}_{X_h}. 
		\end{aligned}
	\end{equation}
	The remaining terms are estimated as in the proof of Lemma~\ref{lem:dt-est}, which yields 
	
	\begin{equation}\label{eq:est-press-3}
		\begin{aligned}
			\abs{\vz(\bu^k_j, \bv_h)} & \lesssim c
			\left( 
			\norm{ \bu_j^k}_{X_h} 
			+
			\norm{\bu^k_j}_{L^4(\Omega)} 	
			\norm{\bu^k_j}_{H^1(\Omega)}
			+ \norm{{\bf}_j^{k} }_{B}
			\right) \norm{ \bv_h}_{X_h}. 
		\end{aligned}
	\end{equation}
	By the inf-sup condition in Assumption~\ref{assump:fem}~\ref{itm:d-inf-sup-H10}, identity~\eqref{eq:est-press-1}, and estimate~\eqref{eq:est-press-3} on $z_2$ we find 
	\begin{equation} 
		\begin{aligned}\label{eq:est-press-2a}
			c_s \norm{\pi^k_j}_{L^{2}(\Omega)} 
			& \leq
			\sup_{\bv_h \in \Xh \cap H^1_0(\Omega)^d \setminus \{\b0\}} \frac{\skp{ \diver \bv_h}{ \pi_j^k } }{\norm{\bv_h}_{X_h} }\\
			&	\leq  
			\sup_{\bv_h \in \Xh\setminus \{\b0\}} \norm{\bv_h}_{X_h}^{-1} \left( 
			\abs{(\difft {\bu}^{k}_j,{\bv}_h)_B } + \abs{ \vz ( \bu^k_j,\bv_h)}\right)
			\\
			& \lesssim  \tfrac{1}{\delta_k} \norm{ \bu^k_{j} -  \bu^k_{j-1}}_B +   
			\norm{ \bu_j^k}_{X_h} 
			+ \norm{\bu^k_j}_{L^4(\Omega)} 	
			\norm{\nabla \bu^k_j}_{L^2(\Omega)} 
			+ \norm{{\bf}_j^{k} }_{B}. 
		\end{aligned}
	\end{equation}
	Raising both sides to the power $\frac{8}{d+4}$, summing over $j$ and multiplying by $\delta_k^2$, allows us to bound the right-hand side using the a priori estimates in Lemma~\ref{lem:apriori-unst} and estimate~\eqref{est:L2-stab-f}, and thus we find 
	\begin{align}\label{eq:est-press-5}
		\delta_k^2 \sum_{j = 1}^{m_k} \norm{\pi_j^k}_{L^2(\Omega)}^{8/(d+4)} \leq c,
	\end{align}
	uniformly in $k$. 
	This shows the claim. 
\end{proof}

\subsection{Convergence}\label{sec:unst-conv}

In the following convergence proof we avoid the use of  strong convergence in $B$ to extract convergence {of} the traces. 
In particular, the proof is therefore also valid for $\beta = 0$. 

\begin{lemma}[convergence of velocities and boundary terms]\label{lem:unst-conv-u}
In the situation of Assumption~\ref{assump:general}, let $ \bu^k \in \mathcal{L}^0_0(J_{\delta_k};W_{h_k})$ be the piecewise constant interpolant of the sequence $( \bu^k_j)_{j \in {1, \ldots, m_k}}$ of approximate solutions to~\eqref{eq:unst-discr}. 
	Then, there is a (non--relabelled) subsequence and a function
	\begin{align*}
		\bu & \in L^\infty(I;B) \cap L^{2}(I;W) \cap L^2(I;H^1_{\bn,\divergence}(\Omega)) 
	\end{align*}
	such that 
	\begin{alignat*}{5}
		\bu^k 
		&\to  \bu \quad 
		&&\text{ strongly in } L^q(I;B), &&\text{ for any } q \in [1,\infty),
		\\
		\bu^k 
		&\overset{*}{\rightharpoonup}  \bu \quad 
		&&\text{ weakly* in } L^\infty(I;B),
		\\
		\bu^k 
		&\rightharpoonup  \bu \quad 
		&&\text{ weakly in } L^2(I;W),
		\\
		\bu^k &\to \bu \quad &&\text{ strongly in } L^s(I;L^4(\Omega)^d), 
		\quad &&\text{ for any } s \in [1,8/d),\\
		\tr (\bu^k) \cdot \bn &\to \b0 \qquad &&\text{ strongly in } L^2(I;L^2(\Gamma)^d), \quad &&
		\\
		\tr_\tau (\bu^k)
		&\to \tr_\tau(\bu) \qquad
		&& \text{ strongly in } L^2(I;L^p(\Gamma)^d),  \quad
		&&  \text{ for any } 
	   p \in [1, 2^\sharp),\\
		\bsigma^k
		& 
		\wconv \bsigma \qquad
		&& 
		\text{ weakly in } L^{p}(I;L^{p}(\Gamma)^d), &&\text{ for } p = \min(2,r')
	\end{alignat*}	
	as $k \to \infty$, {where $r$ is as in Theorem~\ref{thm:main-unsteady}} and $2^\sharp$ as in~\eqref{def:2-sharp}.
	
	In case~\ref{itm:case-impl-coerc} one additionally has that 
	\begin{alignat*}{5}
		\tr_\tau (\bu^k) &\wconv \tr_\tau(\bu) \qquad &&\text{ weakly in } L^r(I;L^r(\Gamma)^d),
		\quad
		&& \\
		\widehat{\bsigma}^k &\wsconv \widehat{\bsigma} \qquad &&\text{ weakly* in } L^{r'}(I;L^{r'}(\Gamma)^d)&& 
	\end{alignat*}	
	as $k \to \infty$, 
	with 
	\begin{align}\label{id:sigma}
		\bsigma = \widehat{\bsigma} - \lambda \tr_{\tau}(\bu). 
	\end{align}
\end{lemma}
\begin{proof}
	We divide the proof into five major steps. 
	
	\textit{1. Step (compactness and strong convergence):}
	We apply the discrete Aubin--Lions Lemma~\ref{lem:d-AL-0} with $\norm{\cdot}_{W_h}$, $\norm{\cdot}_{Y_h}$ and $q = 1$ { as indicated in Remark~\ref{rmk:AL-appl}.}
	The respective uniform bounds on $( \bu^k_j)_{j \in \{1, \ldots, m_k\}}$ for $k \in \mathbb{N}$ are satisfied due to  Lemmas~\ref{lem:apriori-unst} and~\ref{lem:dt-est}. 
	Thus, there exists a (not relabelled) subsequence of $( \bu^k)_{k \in \mathbb{N}}$ such that 
	\begin{align}\label{eq:conv-unst-1}
		\bu^k \to  \bu \quad \text{ strongly in } L^1(I;B) \quad \text{ as } k \to \infty.
	\end{align}
	We combine this strong convergence with the uniform estimate in $L^\infty(I;B)$  and use interpolation between $L^1(I)$ and $L^\infty(I)$ to obtain  
		\begin{align}
			\label{eq:conv-unst-2a}
			\bu^k \to  \bu \qquad \text{ strongly in } L^q(I;B),  \;\;\text{ for any } q \in [1,\infty)
		\end{align}
		as $k \to \infty$.
		Similarly, we combine the uniform bounds in $L^{2}(I;H^1(\Omega)^d)$ thanks to~\eqref{est:ap-pwconst} with an embedding estimate and with the strong convergence in~\eqref{eq:conv-unst-2a}. 
		Then, by interpolation in space and time we also find that 					
  \begin{align}
			\label{eq:conv-unst-2b}
			\bu^k \to \bu \quad \text{ strongly in } L^s(I;L^4(\Omega)^d), 
			\;\; \text{ for any } s \in [1,8/d)
		\end{align}
		as $k \to \infty$. Note that the arguments are analogous to the ones proving~\eqref{est:interp}. 
	
	Note that for $\beta >0 $ due to the definition of the norm $\norm{\cdot}_B$ in~\eqref{def:norm-B} strong convergence of the full traces is a consequence of~\eqref{eq:conv-unst-2a}. 
		However, if $\beta = 0$, then this is not possible, and hence we have to take a detour to obtain strong convergence. 
		For this purpose, we shall use interpolation estimates and interpolation to obtain stronger convergence results on the full traces. And then we will use the convergence of the normal traces to zero due to the penalisation, to deduce the strong convergence for the tangential traces. 
	
 First, by function space interpolation, we have that
	$[L^2(\Omega)^{d},W^{1,2}(\Omega)^{d}]_{3/4,2} = W^{3/4,2}(\Omega)^{d}$. 
 Hence, there exists a constant $c>0$ such that, whenever $\bv\in L^{2}(I;L^{2}(\Omega)^{d})\cap L^{2}(I;W^{1,2}(\Omega)^{d})$,
	we have 
	\begin{align*}
		\|\bv(t,\cdot)\|_{W^{\frac{3}{4},2}(\Omega)} \leq c\,\|\bv(t,\cdot)\|_{L^{2}(\Omega)}^{\frac{1}{4}}\|\bv(t,\cdot)\|_{W^{1,2}(\Omega)}^{\frac{3}{4}}\qquad\text{for $\mathcal{L}^{1}$-a.e.~$t\in I$}, 
	\end{align*}
	whereby an integration with respect to $t$ and H\"{o}lder's inequality yield that 
	\begin{align}\label{eq:tabtabtab}
		\|\bv\|_{L^{2}(I;W^{3/4,2}(\Omega))} \lesssim \|\bv\|_{L^{2}(I;L^{2}(\Omega))}^{\frac{1}{4}}\|\bv\|_{L^{2}(I;W^{1,2}(\Omega))}^{\frac{3}{4}}. 
	\end{align}
	By~\eqref{eq:conv-unst-2a}, $(\bu^k)_{k \in \mathbb{N}}$ converges strongly in $L^2(I;L^2(\Omega)^d)$. On the other hand,  by boundedness of $(\bu^k)_{k \in \mathbb{N}}$ in $L^2(I;W^{1,2}(\Omega)^d)$ due to Lemma~\ref{lem:apriori-unst}, \eqref{eq:tabtabtab} implies that  
	\begin{align}\label{eq:conv-unst-3}		
		\bu^k \to \bu \quad \text{ strongly in }  L^2(I;W^{3/4,2}(\Omega)^d) \quad \text{ as } k \to \infty. 
	\end{align}
	By continuity of the trace operator $W^{s,p}(\Omega)^{d} \to W^{s-\frac{1}{p},p}(\Gamma)^{d}$, we have in particular that  
	\begin{align}\label{eq:conv-unst-4}
		\tr(\bu^k) \to \tr(\bu) \quad \text{ strongly in } L^2(I;L^2(\Gamma)^d)\; \quad \text{ as } k \to \infty. 
	\end{align}
	We now claim that~\eqref{eq:conv-unst-4} implies that 
	\begin{alignat}{3}\label{eq:conv-unst-5a}
		\tr (\bu^k) &\to \tr(\bu) \qquad \text{ strongly in } L^2(I;L^p(\Gamma)^d) 
		\text{ for any }  p < 2^\sharp. 
	\end{alignat}
	Indeed, if $p\leq 2$, then~\eqref{eq:conv-unst-5a} directly follows from~\eqref{eq:conv-unst-4}. Now, if $2<p<2^{\sharp}$, 
	interpolation yields with a suitable $0<\mu<1$ that 
	\begin{align}\label{eq:trace-interp}
		\|\mathrm{tr}(\bu^{k}-\bu)\|_{L^{2}(I;L^{p}(\Gamma))} \leq \|\mathrm{tr}(\bu^{k}-\bu)\|_{L^{2}(I;L^{2}(\Gamma))}^{{\mu}}\|\mathrm{tr}(\bu^{k}-\bu)\|_{L^{2}(I;L^{2^{\sharp}}(\Gamma))}^{1-\mu}
	\end{align}
	By Lemma~\ref{lem:apriori-unst} and continuity of the trace embedding, see~\eqref{est:trace}, $(\mathrm{tr}(\bu^{k}))_{k\in\mathbb{N}}$ is bounded in $L^{2}(I;L^{2^{\sharp}}(\Gamma)^{d})$. Hence, in view of~\eqref{eq:conv-unst-4}, \eqref{eq:trace-interp} implies~\eqref{eq:conv-unst-5a}. 
	
	\textit{2. Step (weak convergence):}
	By the uniform bounds of $({\bu}^k)_{k \in \mathbb{N}}$ in Lemma~\ref{lem:apriori-unst}, the Banach-Alaoglu theorem and Lemma~\ref{lem:h-norm}, there exist further (non-relabelled) subsequences such that 
	\begin{alignat}{5} \label{eq:conv-unst-5c}
		\bu^k &\overset{*}{\rightharpoonup}  \bu \qquad &&\text{ weakly* in } \;\; &&L^{\infty}(I;B),\\
		\label{eq:conv-unst-5d}
		\bu^k &\rightharpoonup    \bu \quad &&\text{ weakly in } &&L^{2}(I;W)
	\end{alignat}
	as $k \to \infty$. 
	
	\textit{3. Step (traces):}
	By the bound~\eqref{eq:Tabowser} on $( \bu^k)_{k \in \mathbb{N}}$, the definition of the norm $\norm{\cdot}_{X_h}$ in~\eqref{def:norm-h} implies that
	\begin{align}\label{eq:conv-unst-6a}
		\tr (\bu^k) \cdot \bn \to 0 \qquad \text{ strongly in } L^2(I;L^2(\Gamma))
	\end{align}
	as $k \to \infty$, and therefore~\eqref{eq:conv-unst-5a} gives us 
	\begin{align}\label{eq:Tabretto}
		\tr(\bu) \cdot \bn = 0\qquad\text{ as an identity in $L^2(I;L^2(\Gamma))$}. 
	\end{align}
	Since $\bu$ has a trace in $L^{2}(I;L^{2^\sharp}(\Gamma)^d)$, the identity also holds in this space, and in particular we have $\bu \in L^2(I;\Hdivn)$. 
		
	Note that by the convergence of the full trace in~\eqref{eq:conv-unst-5a} and of the normal trace in~\eqref{eq:conv-unst-6a}, also the tangential trace converges 
		\begin{align}\label{eq:conv-unst-6b1}
			\tr_{\tau}(\bu^k) = \tr(\bu^k) - (\tr(\bu^k) \cdot \bn) \bn  \to \tr(\bu) - 0 = \tr_{\tau}(\bu) \quad \text{ strongly in } L^2(I;L^2(\Gamma)^d),  
		\end{align}
		an $k \to \infty$. 
		By the fact that the full traces are bounded in $L^2(I;L^{2^\sharp}(\Gamma)^d)$ by Lemma~\ref{lem:apriori-unst}, using the fact that $\abs{\tr(\bu)}^2 = \abs{\tr_\tau(\bu)}^2 + \abs{\tr(\bu) \cdot \bn}^2$, also the tangential traces are bounded in $L^2(I;L^{2^\sharp}(\Gamma)^d)$. 
		Interpolating between $L^2(I;L^{2^\sharp}(\Gamma)^d)$ and $L^2(I;L^2(\Gamma)^d)$, we thus obtain that 
		\begin{align}\label{eq:conv-unst-6b2}
			\tr_{\tau}(\bu^k) \to \tr_{\tau}(\bu) \quad \text{ strongly in } L^2(I;L^p(\Gamma)^d),  \text{ for any } p < 2^{\sharp}
		\end{align}
		an $k \to \infty$. 
	
In case~\ref{itm:case-impl-coerc} from the uniform estimates of $(\tr_{\tau} (\bu^k))_{k \in \mathbb{N}}$ in $L^{r}(I;L^r(\Gamma)^d)$  due to Lemma~\ref{lem:apriori-unst}, the Banach--Alaoglu theorem allows us to extract a weakly converging subsequence, such that
		\begin{alignat}{5} \label{eq:conv-unst-6c1}
			\tr_{\tau}(\bu^k) & \wconv \tr_{\tau}(\bu) && \quad \text{ weakly in } L^{r}(I;L^r(\Gamma)^d) \;\; \text{ as } k \to \infty. 
		\end{alignat}
	
	\textit{4. Step (nonlinear trace term)}:
	With the definition of $\bsigma^k$ in~\eqref{eq:unst-discr-b} from the uniform estimates in Lemma~\ref{lem:apriori-unst} it follows that $(\bsigma^k)_{k \in \mathbb{N}}$ is bounded in $L^{p}(I;L^{p}(\Gamma)^d)$ for $p = \min(2,r')$. 
	Hence, there is a weakly converging subsequence such that 
		\begin{alignat}{5}
			\label{eq:conv-unst-6c3}
			\bsigma^k &\wconv  \bsigma && \quad \text{ weakly in } L^{p}(I;L^{p}(\Gamma)^d)
		\end{alignat}
		as $k \to \infty$. 
		Similarly, in case~\ref{itm:case-impl-coerc} from the uniform estimates of $(\widehat \bsigma^k)_{k \in \mathbb{N}}$ in $L^{r'}(I;L^{r'}(\Gamma)^d)$, we obtain a weakly star converging subsequence such that
		\begin{alignat}{5}
			\label{eq:conv-unst-6c2}
			\widehat \bsigma^k &\wsconv \widehat \bsigma && \quad \text{ weakly$^*$ in } L^{r'}(I;L^{r'}(\Gamma)^d) \;\; \text{ as } k \to\infty. 
		\end{alignat}
		
		Note that in case~\ref{itm:case-impl-coerc} we have $\bsigma^k = \widehat{\bsigma}^k - \lambda \tr_{\tau}(\bu^k)$. 
		Since by Lemma~\ref{lem:unst-conv-u} we have $\bsigma^k \wconv \bsigma$ weakly in $L^{\min(2,r')}(I;L^{\min(2,r')}(\Gamma)^d)$, $\widehat{\bsigma}^k \wsconv \widehat \bsigma$ weakly$^*$ in $L^{r'}(I;L^{r'}(\Gamma)^d)$, and $\tr_{\tau}(\bu^k) \to \tr_{\tau}(\bu)$ strongly in $L^2(I;L^2(\Gamma)^d)$ as $k \to \infty$,  with the uniqueness of limits we obtain that 
		\begin{align}
			\bsigma = \widehat{\bsigma} - \lambda \tr_{\tau}(\bu) \quad \text{ in } L^{\min(2,r')}(I;L^{\min(2,r')}(\Gamma)^d),
		\end{align}
		which proves~\eqref{id:sigma}. 
		
		\textit{5. Step (incompressibility):}	
		We now show that $\bu$ is divergence-free. 	
		Let $q \in C^{\infty}(\overline \Omega)\cap L^2_0(\Omega)$ and let $\phi \in C^\infty_0(I)$ be arbitrary. By  Assumption~\ref{assump:fem}~\ref{itm:fem-Q-approx}  there exists  a sequence $(q_h)_{h>0} \subset Q_h$  such that 
		\begin{alignat}{3}\label{eq:conv-unst-7a-1}
			q_h \to q \qquad \text{ strongly in } L^2(\Omega) \qquad \text{ and }  \quad 
			q_h|_{\Gamma} \to \tr(q) \qquad \text{ strongly in } L^2(\Gamma)
		\end{alignat}
		as $k \to \infty$.					Furthermore, let $\phi^k \coloneqq \Pi_{\delta_k} \phi \in \mathcal{L}^0_0(J_{\delta_k})$ be the projection from Section~\eqref{def:Pid}. 
		Then, for $\psi^k \coloneqq q_{h_k} \phi^k \in \mathcal{L}^0_0(J_{\delta_k};Q_{h_k})$ we have 
		\begin{alignat}{3}\label{eq:conv-unst-7a}
			\psi^k &\to  \phi q \qquad &&\text{ strongly in } L^2(Q), \\
			\label{eq:conv-unst-7b}
			\psi^k|_{\Gamma} &\to \phi q|_{\Gamma} &&\text{ strongly in } L^2(I \times \Gamma)
		\end{alignat}
		as $k \to \infty$. 
		
		Since $\bu^k \in \mathcal{L}^0_0(J_{\delta_k};X_{h,\diver})$ we have 
		\begin{align*}
			\int_0^T \skp{\divergence \bu^k}{\psi^k} - \skp{\bu^k \cdot \bn}{ \psi^k}_{\Gamma} \dt = 0\qquad \text{for all}\;k\in\mathbb{N}. 
		\end{align*}
		With~\eqref{eq:conv-unst-6a} and the fact that $(\psi^k)_{k\in\mathbb{N}}$ is bounded in $L^2(I \times \Gamma)$ by~\eqref{eq:conv-unst-7b}, the second term vanishes as $k \to \infty$. Hence, 
		\begin{align}\label{eq:conv-unst-8}
			\int_0^T \skp{\divergence \bu^k}{\psi^k } \dt \to 0 \qquad \text{ as } k \to \infty.  
		\end{align}
		To conclude the proof, we estimate 
		\begin{equation}
			\begin{aligned}\label{eq:conv-unst-9}
				\abs{\int_0^T \skp{\divergence \bu}{\phi q} \dt}  
				\leq \abs{\int_0^T \skp{\divergence (\bu - \bu^k)}{\phi q} \dt} 
				&+ \abs{\int_0^T \skp{\divergence \bu^k}{\phi q - \psi^k} \dt}
				\\
				&+ \abs{\int_0^T \skp{\divergence \bu^k}{\psi^k} \dt}.
			\end{aligned}
		\end{equation}
		The first term vanishes by the fact that $\bu^k \rightharpoonup \bu$ in $L^2(I;H^1(\Omega)^d)$ as $k \to \infty$. 
		Since $(\bu^{k})_{k\in\mathbb{N}}$ is bounded in $L^{2}(I;H^{1}(\Omega)^{d})$, \eqref{eq:conv-unst-7a} implies that the second term vanishes as $k\to\infty$. 
		Lastly, by~\eqref{eq:conv-unst-8}, the ultimate term vanishes as $k \to \infty$. 
		In consequence, $\divergence \bu = 0$ in the sense of distributions. 
		Because of $\divergence \bu \in L^2(Q)$ and~\eqref{eq:Tabretto},  $\bu \in L^2(I;H^1_{\bn,\divergence}(\Omega))$. 
		This completes the proof. 
	\end{proof}
	
	Next, we study the convergence of the pressure functions:	
	\begin{lemma}[convergence of the pressure] \label{lem:unst-conv-pi}
	In the situation  of Assumption~\ref{assump:general}, let $(\pi^k)_{k \in \mathbb{N}}$  be the sequence of pressure functions so that, in particular,  $\pi^k \in  \mathcal{L}^0_0(J_{\delta_k};Q_{h_k})$ for all $k\in\mathbb{N}$. 
		Then, there exists a (non-relabelled)  subsequence 
		and a function 
		\begin{align*}
			\xi \in L^{8/(d+4)}(I;L^2(\Omega))
		\end{align*}
		with the following property: If $\phi \in C^{\infty}(\overline I)$ satisfies $\phi(T) = 0$ and $\phi^k \in  \mathcal{L}^0_0(J_{\delta_k};Q_{h_k})$ denotes the piecewise constant interpolant of $(\phi^k_j)_{j = 1, \ldots, m_k}$ with $\phi^k_j = \phi(t_j)$ for $j\in\{1,...,m_{k}\}$, 
		then 
		\begin{align*}
			\int_0^T \pi^k \phi^k \dt \wconv -	\int_0^T \xi \partial_t \phi \dt \quad \text{ weakly in $L^2(\Omega)$ as   $k \to \infty$}.
		\end{align*}
	\end{lemma}
	\begin{proof}
		For the pressure functions $\pi^k \in \mathcal{L}^0_0(J_{\delta_k};Q_h)$ we define $\xi^k \in \mathcal{L}^1_1(J_{\delta_k};Q_h)$ by
		\begin{align}\label{def:phik-2}
			\xi^k(t) \coloneqq \int_0^t \pi^k(s) \ds \quad \text{ for } t \in I. 
		\end{align}
		This means that 
		\begin{align}
			\xi^k(0) = 0, 
			\quad 
			\pi^k = \partial_t \xi^k,
			\quad 
			\text{ and } \quad 
			{\pi^k_j = \difft \xi^k_j}. 
		\end{align}
		By Lemma~\ref{lem:p-est} we have that 
		\begin{align}\label{eq:est-press-4}
			\norm{\xi^k}_{L^{8/(d+4)}(I;L^2(\Omega))} \leq c,
		\end{align}
		uniformly in $k$.
		Thus, there is a function $\xi \in L^{8/(d+4)}(I;L^2(\Omega))$ and a (non-relabelled) subsequence such that 
		\begin{alignat}{3}\label{eq:conv-xi}
			\xi^k 
			&\wconv \xi \quad 
			&&\text{ weakly in } L^{8/(d+4)}(I;L^2(\Omega)) \quad \text{ as } k \to \infty.
		\end{alignat}	
		Noting that $\xi^k(0) = 0$ for any $\psi \in H^1(I)$ with $\psi(T) = 0$ we obtain 
		\begin{align}\label{eq:pressure-ibp}
			\int_0^T \pi^k \psi \dt 
			=
			\int_0^T \partial_t \xi^k \psi \dt 
			= 
			-	\int_0^T  \xi^k \partial_t \psi \dt. 
		\end{align}
		Let $\phi \in C^\infty(\overline I)$ with $\phi(T) = 0$ be arbitrary. 
		Defining $\widehat \phi^k \in \mathcal{L}^1_1(J_{\delta_k})$ as the Lagrange interpolation of $\phi$, we have $\widehat{\phi}^{k}(T)=0$ and 
		\begin{align}\label{eq:phi-conv}
			\widehat \phi^k \to \phi \quad \text{ strongly in } W^{1,\infty}(I)\;\text{as}\;k\to\infty.
		\end{align} 
		Now, applying~\eqref{eq:pressure-ibp}, and using both the strong convergence $\partial_t \widehat \phi^k \to \partial_t \phi$ in $W^{1,\infty}(I)$ from~\eqref{eq:phi-conv} and the weak convergence of $\xi^k \wconv \xi$ in $L^{8/(d+4)}(I;L^{2}(\Omega))$ from~\eqref{eq:conv-xi} we find that 
		\begin{align}\label{eq:pressure-1}
			\int_0^T \pi^k \widehat \phi^k \dt 
			= 	
			- \int_0^T \xi^k  \partial_t \widehat \phi^k  \dt
			\wconv 
			- \int_0^T \xi \partial_t \phi  \dt
		\end{align}
		weakly in $L^2(\Omega)$ as $k \to \infty$. 
		Since~$\phi^k \in \mathcal{L}^0_0(J_{\delta_k})$ is the piecewise constant interpolant of $\phi$, and therefore also of $\widehat \phi^k$, we obtain 
			\begin{align}\label{eq:pressure-2}
				\norm{\phi^k - \widehat{\phi}^k}_{L^\infty(I)} \leq  \delta_k \norm{\partial_t \widehat{\phi}^k}_{L^\infty(I)} 
				\leq \delta_k \norm{\partial_t {\phi}}_{L^\infty(I)}. 
			\end{align}
			Therefore, with Hölder's inequality we find that 
			\begin{equation}\label{eq:pressure-3}
				\begin{aligned}
					\norm{	\int_0^T \pi^k (\widehat \phi^k - \phi^k) \dt  }_{L^2(\Omega)}
					&\lesssim  \norm{\pi^k}_{L^{\frac{8}{d+4}}(I;L^2(\Omega))}	\norm{\phi^k - \widehat{\phi}^k}_{L^\infty(I)}  
					\\
					&\lesssim \left( \delta_k\sum_{ j = 1}^{m_k}  \norm{\pi^k_j}_{L^2(\Omega)}^{8/(d+4)} \right)^{\frac{d+4}{8}}  \delta_k \norm{\partial_t {\phi}^k}_{L^\infty(I)} 
					\\
					&\lesssim \delta_k^{\frac{4-d}{8}} 
					\left( \delta_k^2 \sum_{ j = 1}^{m_k}  \norm{\pi^k_j}_{L^2(\Omega)}^{8/(d+4)} \right)^{\frac{d+4}{8}}   \norm{\partial_t {\phi}^k}_{L^\infty(I)}. 
				\end{aligned}
			\end{equation}
			Thanks to the estimate in Lemma~\ref{lem:p-est},
			the right-hand side converges to $0$ as $k \to \infty$. 
			Thus, in combination with~\eqref{eq:pressure-1}, this proves the claim.  
	\end{proof}
	
	\subsection{Limit passage in equation}\label{sec:unst-limit}
	
	In the following we shall prove that the limiting functions satisfy the weak formulation in Definition~\ref{def:unst-w-sol}, as well as an energy inequality and the attainment of the initial data. 
	Afterwards it remains to identify the nonlinear boundary condition.  
	
	\begin{proposition}[limiting equation]\label{prop:limit-eq}
	In the situation of Assumption~\ref{assump:general} the limiting pair of functions~$(\bu, \bsigma)$ from Lemma~\ref{lem:unst-conv-u} satisfies that 
		\begin{align}
			\skp{\partial_t  \bu}{ \bv}_{\Wdiv} 
			+ 2 \nu \skp{\BD \bu}{ \BD \bv}   
			+  b(\bu,\bu,\bv)  
			+  \skp{
				\bsigma
			}{\tr_\tau(\bv)}_{\Gamma}  =  ( \bf, {\bv})_B,
		\end{align}
		for all $\bv \in \Wdiv$ and for a.e.~$t \in I$, where 
		$\bsigma = \widehat{\bsigma} - \lambda \tr_{\tau}(\bu)$ in case~\ref{itm:case-impl-coerc}. 
		Furthermore, we have that $\bu \in C_w(\overline I;B_{\divergence})$, and that the initial datum is attained in the sense as specified in Definition~\ref{def:unst-w-sol}. 
		
		The following energy identity 
		\begin{align*}
			\tfrac{1}{2} \norm{\bu(s_2)}_{B}^2  + 2 \nu \int_{s_1}^{s_2}  \norm{\BD \bu}_{L^2(\Omega)}^2 \dt   + \int_{s_1}^{s_2} \skp{\bsigma}{\tr_{\tau} (\bu)}_{\Gamma} \dt 
			= \int_{s_1}^{s_2} (\bf,\bu)_B \dt + 	\tfrac{1}{2} \norm{\bu(s_1)}_{B}^2  
		\end{align*}
		is satisfied for a.e.~$0<s_1<s_2< T$, and for $s_1 = 0$.  
	\end{proposition}
	\begin{proof}
		Using the convergence results we want to take the limit in the discrete equation~\eqref{eq:unst-discr}.  
		
		\textit{1. Step (distributional solution):}
		We first show that the limiting  functions $\bu, \bsigma$  and $\xi$  in Lemmas~\ref{lem:unst-conv-u}, \ref{lem:unst-conv-s-c1}, \ref{lem:unst-conv-s-c2}, and \ref{lem:unst-conv-pi} satisfy
		\begin{equation}\label{eq:unst-lim-eq-1a}
			\begin{aligned}
				- \int_0^T ( \bu, \bv)_{B}\, \partial_t \phi \dt  
				&+ 2 \nu \int_{0}^T \skp{\BD \bu}{ \BD \bv} \phi \dt 
				+ \int_0^T b(\bu,\bu,\bv) \phi \dt 
				+ \int_0^T \skp{\xi}{\divergence \bv} \partial_t \phi \dt 
				\\
				& + \int_{0}^T \skp{\bsigma}{\tr_\tau(\bv)}_{\Gamma} \, \phi \dt  
				=
				\int_{0}^T ( \bf, {\bv})_B \,\phi  \dt
				+ (\bu(0),  \bv)_B \, \phi(0),
			\end{aligned}
		\end{equation}
		for arbitrary $\bv \in C^{\infty}(\overline \Omega)^d$, with $\tr(\bv) \cdot \bn  = 0$ on $\Gamma$, and $\phi \in C^{\infty}(\overline I)$ with $\phi(T) = 0$. 
		This means that $(\bu,\bsigma)$ satisfies~\eqref{eq:NS-unst-full} in the sense of distributions. 
	
	For this purpose, let such test functions $\bv, \phi$ be arbitrary but fixed. 
	For $k \in \mathbb{N}$ we consider the sequence $\phi_j^k \coloneqq \phi(t_j)$ for $j = 0, \ldots, m_k$, and let $\phi^k \in\mathcal{L}^0_0(J_{\delta_k})$ be the piecewise constant interpolant of $(\phi_j^k)_{j = 1, \ldots, m_k}$. 
	Note that $\phi(T) = \phi(t_{m_k}) = 0 $. 
	Furthermore, we denote by $\widehat{\phi}^k \in \mathcal{L}^1_1(J_{\delta_k})$ the Lagrange interpolation of $\phi$, i.e., the continuous piecewise affine interpolant of  $(\phi_j^k)_{j = 0, \ldots, m_k}$. 
	Then $\partial_t \widehat \phi^k \in \mathcal{L}^0_0(J_{\delta_{k}})$  coincides with the piecewise constant interpolant of $(\difft \phi_j)_{j}$ and we have that $\widehat{\phi}^k(T) = 0$. 
	By standard approximation results we have the uniform convergence 
	\begin{align}\label{eq:unst-lim-phik}
		\norm{\phi - \phi^k}_{L^\infty(I)} + 
		\norm{\partial _t \phi - \partial_t \widehat{\phi}^k}_{L^\infty(I)} \to 0 \quad \text{ as } k \to \infty.
	\end{align}
	For $h = h_k$ let $\PiSZ$ be the Scott--Zhang operator, cf.~Lemma~\ref{lem:SZ}, and choose $\bv_h \coloneqq \PiSZ \bv \in X_{h} \cap H^1_{\bn}(\Omega)$. 
	Note that by Lemma~\ref{lem:SZ}~\ref{itm:SZ-normal} we have that $\bv_h \cdot \bn = 0$, since $\tr(\bv) \cdot \bn = 0$ on $\Gamma$.
	Then, by the approximation property in Lemma~\ref{lem:SZ} and the trace inequality~\eqref{est:trace}
		we obtain 
	\begin{alignat}{3}\label{eq:unst-lim-vh}
		\bv_h &\to \bv \quad &&\text{ strongly in } H^{1}(\Omega)^d,\\ \label{eq:unst-lim-vh-tr}
		\tr(\bv_h)  = \tr_\tau (\bv_h)&\to \tr(\bv) = \tr_\tau(\bv) \quad &&\text{ strongly in } L^{s}(\Gamma)^d
	\end{alignat}
	as $h \to 0$, for any $s \in [1,2^\sharp)$.  
	In particular, with Lemma~\ref{lem:h-norm} we have that  
	\begin{align}\label{eq:unst-lim-vh-WB}
		\bv_h \to  \bv \quad \text{ strongly in } W \text{ and in } B \quad \text{ as } h \to 0.
	\end{align}
	Because $\bv_h \cdot \bn = 0$ on $\Gamma$, testing~\eqref{eq:unst-discr-a} with $(\bv_h,0) \in X_h \times Q_h$ the terms including the normal trace of $\bv_h$ vanish and we have 
	\begin{equation}
		\begin{aligned}	\label{eq:unst-discr-v}
			(\difft {\bu}^{k}_j,{\bv}_h)_B
			&+ 2\nu \skp{ \BD \bu^{k}_j}{ \BD \bv_h } 
			+ \widetilde{b}(\bu^{k}_j, \bu^{k}_j, \bv_h)
			- \skp{ \pi_j^{k}}{ \diver \bv_h } 
			\\
			& 
			+ \skp{ \bsigma^k_j}{\tr_{\tau}(\bv_h) }_{\Gamma}
			- 2 \nu {\theta} \skp{   \left(( \BD \bv_h) \bn \right) \cdot \bn }{ \bu_j^{k} \cdot \bn }_{\Gamma}
			= ({\bf}_j^k ,{\bv}_h )_B.
		\end{aligned}
	\end{equation} 
	Due to the fact that $\phi(T) = \phi^k_{m_k} = 0$ we may use summation by parts and the fact that $ \bu^k_0 =  \Pi_h  \bu_0$, see~\eqref{eq:unst-discr},  to find 
	\begin{equation}\label{eq:unst-lim-1-a}
		\begin{aligned}
			\delta_k	\sum_{j= 1}^{m_k} (\difft  {\bu}^{k}_j,{\bv}_h)_B \,\phi_j^k 
			& =
			-	\delta_k \sum_{j= 1}^{m_k} (  {\bu}^{k}_{j-1},{\bv}_h)_B\, \difft  \phi_{j}^k 
			- ( \bu^k_0,  \bv_h)_B\,\phi_0^k \\
			& = -  \int_0^T ( \bu^k(t-\delta_k),  \bv_h)_B \,\partial_t \widehat{\phi}^k(t) \dt 
			- ( \bu^k_0,  \bv_h)_B\,\phi(0). 
		\end{aligned}
	\end{equation}
	Formally, we extend $\bu^k$ to $(-\delta_k,T]$ by $\bu^k_0$. 
	Multiplying the equation~\eqref{eq:unst-discr-v} with $\delta_k \phi_j$ and summing over $j \in \{1, \ldots, m_k\} $ we obtain 
	\begin{equation}
		\begin{aligned}	\label{eq:unst-discr-v2}
			&-  \int_0^T ( \bu^k(t-\delta_k),  \bv_h)_B \,\partial_t \widehat{\phi}^k(t) \dt 
			+ 2\nu \int_0^T \skp{ \BD \bu^{k}}{ \BD \bv_h }\, \phi^k \dt 
			+ \int_0^T \widetilde{b}(\bu^{k}, \bu^{k}, \bv_h) \, \phi^k \dt 
			\\
			& \qquad 
			- \int_0^T \skp{ \pi^{k}}{ \diver \bv_h }\, \phi^k \dt  
			+ \int_0^T \skp{
				\bsigma^k 
			}{\tr_{\tau}(\bv_h) }_{\Gamma} \, \phi^k \dt \\
			& \qquad
			- \int_0^T \skp{   \left((2 \nu {\theta} \BD \bv_h) \bn \right) \cdot \bn }{ \bu^{k} \cdot \bn }_{\Gamma}\, \phi^k \dt 
			= \int_0^T ({\bf}^k ,{\bv}_h )_B \, \phi^k \dt 
			+ 	( \bu^k_0,  \bv_h)_B\,\phi(0).
		\end{aligned}
	\end{equation} 
	We proceed taking the limit term by term. 
	
	For the first term, recall that by~\eqref{eq:unst-lim-phik} and by~\eqref{eq:unst-lim-vh-WB} we have $\partial_t \widehat{\phi}^k \to \partial_t \phi$ in $L^\infty(I)$, and that $ \bv_h \to  \bv$ in $B$ as $k \to \infty$.
	By the estimate in Lemma~\ref{lem:apriori-unst} we have 
	\begin{align}\label{eq:unst-lim-1-b}
		\norm{  \bu^k(\cdot) - \bu^k(\cdot-\delta_k) }_{L^2(I;B)}^2 = \delta_k\sum_{j = 1}^{m_k} 
		\norm{ {\bu}^{k}_j   -{\bu}^{k}_{j-1} }_B^2  \leq \delta_k c \to 0
	\end{align}
	as $k \to \infty$. 
    By the fact that $ \bu^k \to \bu$ strongly in $L^2(0,T;B)$, cf.~Lemma~\ref{lem:unst-conv-u}, we also have that 
	\begin{align}\label{eq:unst-lim-1c}
		\bu^k(\cdot-\delta_k) \to \bu \quad \text{ strongly in } L^2(0,T;B)
	\end{align}
	as $k \to \infty$. 
	In combination, we arrive at  
	\begin{align}\label{eq:unst-lim-1}
		-  \int_0^T ( \bu^k(t-\delta_k), \bv_h)_B \, \partial_t \widehat{\phi}^k \dt 
		\to 
		- \int_0^T ( \bu,  \bv)_B\, \partial_t\phi \dt  
		\qquad \text{ as } k \to \infty. 
	\end{align} 
	Taking the limit in the dissipation term and in the convective term uses standard arguments.
	
	By Lemma~\ref{lem:unst-conv-pi} and $\divergence \bv_h \to \divergence \bv$ strongly in $L^2(\Omega)$, cf.~\eqref{eq:unst-lim-vh}, as well as~\eqref{eq:unst-lim-phik} it follows that 
	\begin{align} \label{eq:unst-lim-4}
		\int_0^T  \skp{\pi^k }{\divergence \bv_h} \, \phi^k \dt \to  
    - \int_0^T  \skp{\xi}{\divergence \bv} \,\partial_t\phi \dt,  
		\qquad \text{ as } k \to \infty. 
	\end{align}
	By Lemma~\ref{lem:unst-conv-u} we know that $\bsigma^k \wconv \bsigma $ converges weakly in $L^{\min(2,r')}(I \times \Gamma)^d$ as $k \to \infty$.  
	Furthermore, by~\eqref{eq:unst-lim-vh-tr} we have that $\tr_\tau(\bv_h) \to \tr_{\tau}(\bv)$ strongly in $L^{\max(r,2)}(\Gamma)^d$ since 
	$r \in [1,2^\sharp]$ with $r < \infty$, and by~\eqref{eq:unst-lim-phik} we have that $\phi^k \to \phi$ strongly in $L^{\max(r,2)}(I)^d$ as $k \to \infty$. 
	This means, we can take the limit 
	\begin{align}\label{eq:unst-lim-5}
		\int_0^T \skp{\bsigma^k}{\tr_\tau(\bv_h)}_{\Gamma} \, \phi^k \dt  \to 
		\int_0^T \skp{\bsigma}{\bv}_{\Gamma} \, \phi \dt 	\qquad \text{ as } k \to \infty. 
	\end{align}
	To show that the last boundary term in~\eqref{eq:unst-discr-v2} vanishes we use the Hölder and the triangle inequality to find
	\begin{equation}\label{eq:unst-lim-7a}
		\begin{aligned}
			\abs{\int_{0}^T  \skp{  \left(( \BD \bv_h) \bn \right) \cdot \bn }{ \bu^{k} \cdot \bn }_{\Gamma} \,\phi^k \dt}  
			&\leq 
			\norm{\nabla \bv_h}_{L^2(\Gamma)} \norm{\bu^k \cdot \bn}_{L^2(I;L^2(\Gamma))} \norm{\phi^k}_{L^2(I)}. 
		\end{aligned}
	\end{equation}
	By the stability property of the Scott--Zhang operator in Corollary~\ref{cor:SZ-stab}, we have  
	\begin{align}\label{est:unst-stab}
		\norm{\nabla \bv_h}_{L^2(\Gamma)}  = \norm{\nabla \PiSZ \bv }_{L^2(\Gamma)} \leq c \norm{\bv}_{H^2(\Omega)},
	\end{align}
	which is bounded uniformly in $h$, since $\bv$ is smooth.  
	In combination with the fact that $(\phi^k)_{k \in \mathbb{N}}$ is bounded in $L^2(I)$ by~\eqref{eq:unst-lim-phik}, and that 
	$\tr(\bu^k) \cdot \bn \to 0$ converges in particular strongly in $L^2(I\times \Gamma)$ by Lemma~\ref{lem:unst-conv-u} we obtain 
	\begin{align}\label{eq:unst-lim-7}
		- \int_0^T \skp{   \left((2 \nu { \theta} \BD \bv_h) \bn \right) \cdot \bn }{ \bu^{k} \cdot \bn }_{\Gamma}\, \phi^k \dt \to 0
		\qquad \text{ as } k \to \infty.  
	\end{align}
	For the force term in~\eqref{eq:unst-discr-v2} by the strong convergence $ \bf^k = \Pi_{\delta_k}(
	\bf)  \to  \bf $  in $L^2(I;B)$ as in~\eqref{eq:f-conv}, and employing again~\eqref{eq:unst-lim-phik}, \eqref{eq:unst-lim-vh-WB} we obtain 
	\begin{align}\label{eq:unst-lim-8}
		\int_0^T ( \bf^k  ,{\bv}_h)_B \,\phi^k \dt 
		\to  
		\int_0^T ( \bf  ,{\bv})_B \,\phi \dt 
		\qquad  \text{ as } k \to \infty. 
	\end{align}
	For the initial velocity we know that  $\bu^k_0 =  \Pi_h  \bu_0 \to  \bu_0$ converges strongly in $B$ as $k \to \infty$, see~\eqref{eq:conv-u0}. 
	Thus, with~\eqref{eq:unst-lim-vh} we have 
	\begin{align}\label{eq:unst-lim-9}
		- ( \bu^k_0,  \bv_h)_B\,\phi(0) 
		\to 
		- ( \bu_0,  \bv)_B\,\phi(0) 
		\qquad  \text{ as } k \to \infty. 
	\end{align}
	
	Altogether, \eqref{eq:unst-discr-v2} as well as~\eqref{eq:unst-lim-1}, 
	\eqref{eq:unst-lim-4}, \eqref{eq:unst-lim-5}, \eqref{eq:unst-lim-7}, \eqref{eq:unst-lim-8} and~\eqref{eq:unst-lim-9} show that $\bu, \bsigma$ and $\xi$ satisfy the distributional formulation~\eqref{eq:unst-lim-eq-1a}. 
	
	\textit{2. Step (pressure-free weak formulation):}
	Testing with arbitrary divergence-free functions $\bv \in \mathcal{W}_{\divergence}$ and $\phi \in C^{\infty}(\overline I)$ with $\phi(T) = 0$, we obtain the pressure-free formulation 
	\begin{equation}\label{eq:unst-lim-distr-2}
		\begin{aligned}
			- \int_0^T ( \bu, \bv)_{B}\, \partial_t \phi \dt  
			&+ 2 \nu \int_{0}^T \skp{\BD \bu}{ \BD \bv} \phi \dt 
			+ \int_0^T b(\bu,\bu,\bv) \phi \dt 
			\\
			& + \int_{0}^T \skp{
				\bsigma
			}{\tr_\tau(\bv)}_{\Gamma} \, \phi \dt  = \int_{0}^T (\bf, {\bv})_B\, \phi  \dt + (\bu(0),  \bv)_B \, \phi(0).
		\end{aligned}
	\end{equation}
 By taking the limit in the corresponding equation for $\bu^k$, we may identify the initial value term replacing $\bu(0,\cdot)$ by $\bu_0$.  
	Noting, that $ \bu \in L^2(I;\Bdiv)$ and $ \bv \in \Bdiv$, we may rewrite a.e.~in $I$ 
	\begin{align}
		( \bu, \bv)_{\Bspace}  = ( \bu, \bv)_{\Bdiv},
	\end{align}
	see~\eqref{eq:B-Bdiv}. 
	Furthermore, by the regularity of each of the terms and by density of $\mathcal{W}_{\divergence}$ in $\Wdiv$ we have 
	\begin{equation}\label{eq:unst-lim-distr-3}
		\begin{aligned}
			- \int_0^T ( \bu, \bv)_{\Bdiv}\, \partial_t \phi \dt  
			&+ 2 \nu \int_{0}^T \skp{\BD \bu}{  \BD \bv} \phi \dt 
			+ \int_0^T b(\bu,\bu,\bv) \phi \dt 
			\\
			& + \int_{0}^T \skp{\bsigma}{\tr_\tau(\bv)}_{\Gamma} \,\phi \dt  = \int_{0}^T ( \bf, {\bv})_B\, \phi \dt + (\bu_0,  \bv)_B \, \phi(0),
		\end{aligned}
	\end{equation}
	for any $\bv \in \Wdiv$ and for any $\phi \in C^{\infty}(\overline I)$ with $\phi(T) = 0$. 
	
	Similarly as before, we may investigate the boundedness properties of 
	\begin{equation}
		\begin{aligned}\label{def:z-form-v}
			\yz(\bu, \bv) &\coloneqq 
			- 2\nu \skp{ \BD \bu}{ \BD \bv } 
			- {b}(\bu, \bu, \bv)
			- \skp{\bsigma}{\tr_{\tau}(\bv) }_{\Gamma}
			+ ({\bf} ,{\bv} )_B,
		\end{aligned}
	\end{equation}
	for $\bv \in \Wdiv$. 
	This allows us to conclude that the distributional time-derivative  of $\bu \in L^2(I;B)$ is $\partial_t \bu \in L^{4/d}(I;(\Wdiv)')$. 
	By the definition of the duality relation between $(\Wdiv)'$ and $\Wdiv$ we have that 
	\begin{equation}\label{eq:unst-lim-weak-1}
		\begin{aligned}
			\skp{\partial_t  \bu}{ \bv}_{\Wdiv} 
			+ 2 \nu \skp{\BD \bu}{ \BD \bv}   
			+  b(\bu,\bu,\bv)  
			+  \skp{
				\bsigma
			}{\tr_\tau(\bv)}_{\Gamma}  =  ( \bf, {\bv})_B,
		\end{aligned}
	\end{equation} 
	for all $\bv \in \Wdiv$ and for a.e.~$t \in I$. 
	This means that the weak formulation in Definition~\ref{def:unst-w-sol} holds. 
		
		\textit{3. Step (attainment of initial data):} 
Recall that $\Wdiv, \Bdiv$ and hence also $(\Wdiv)'$ are reflexive Banach spaces with $\Bdiv \hookrightarrow (\Wdiv)'$, see~\eqref{def:Gelfand}. 
	Thus, in particular we have that 
	\begin{align*}
		\bu \in L^1(I;(\Wdiv)')\cap L^\infty(0,T;\Bdiv) \quad \text{ and } \quad 
		\partial_t  \bu \in L^1(I;(\Wdiv)'). 
	\end{align*}
	Hence, by Lemma~\ref{lem:w-cont} it follows that $ \bu \in C_{w}([0,T];\Bdiv)$. 
	Furthermore, taking the limit in the formulation with time derivative on the test function shows that 
		\begin{align}\label{eq:ic-weak}
			\skp{\bu(s,\cdot) - \bu_0}{\bv} \to 0  \qquad \text{ as } s \to 0, \quad \text{ for any } \bv \in \Bdiv.  
		\end{align}
  
		To show that the initial value is attained, we start from estimate~\eqref{est:unst-ap-6}, which states that
		\begin{align}
			\frac{1}{2\delta_{k}} \left(\norm{{\bu}^k_j}_B^2 - \norm{{\bu}^k_{j-1}}_B^2 + \norm{{\bu}^k_j - {\bu}^k_{j-1}}_B^2 \right)
			+ 
			c \norm{ \bu_j^k}_{X_h}^2 
			\leq
			( \bf_j^k, \bu_j^k)_B + C,
		\end{align}	
		with constants independent of $j,k$. 
		Considering the piecewise constant extensions $\bu^k$, $\bsigma^k$ and $\bf^k = \Pi_{\delta_k} \bf $, cf.~\eqref{def:f-discr}, and the continuous, piecewise affine extension $\widetilde \bu^k$, and integrating over $(0,s)$ yields 
		\begin{align}\label{est:initcond-1}
			\tfrac{1}{2} \norm{\widetilde \bu^k(s,\cdot)}_{B}^2 &-  	\tfrac{1}{2}  \norm{\bu^k_0}_{B}^2
			\leq   \int_{0}^{s}\left(({\bf}^k,{\bu}^k)_B + c\right) \dt.
		\end{align}
		Since $\abs{\widetilde{\bu}^k_j - \bu^k_j}  = \abs{\bu^k_j - \bu^k_{j-1}}$, from the estimates in~\eqref{est:unst-ap-10} it follows that 
		\begin{align*}
			\widetilde{\bu}^k - \bu^k \to 0 \quad \text{ strongly in } L^2(I;B)
		\end{align*}
		as $k \to \infty$. 
        Hence, from $\bu^k \to \bu$ strongly in $L^2(I;B)$ , by Lemma~\ref{lem:unst-conv-u} it follows that $\widetilde{\bu}^k \to \bu$ strongly in $L^2(I;B)$ as $k \to \infty$.  
		Consequently, there is a subsequence, such that 
		\begin{align*}
			\widetilde{\bu}^k(s,\cdot) \to \bu(s,\cdot) \quad \text{ strongly in } B \;\; \text{ as } k \to \infty,
		\end{align*}
		for a.e.~$s \in I$. 
		Also, by~\eqref{eq:conv-u0} we have that 
		that 
		\begin{align*}
			{\bu}^k_0 \to \bu_0, \quad \text{ strongly in } B, \text{ as } k \to \infty. 		
		\end{align*}
		Thus, with the previous two convergence results and~\eqref{est:initcond-1} for a.e.~$s \in I$ we obtain 
		\begin{align*}
			0 \leq \norm{\bu(s,\cdot) - \bu_0}_B^2
			&= 	\lim_{k \to \infty} \norm{\widetilde{\bu}^k(s,\cdot) - \bu^k_0}^2_{B} \\
   & =  	\lim_{k \to \infty} \left(\norm{\widetilde{\bu}^k(s,\cdot)}^2_{B} - \norm{ \bu^k_0}^2_{B}  + 2 \skp{ \bu^k_0- \widetilde{\bu}^k(s,\cdot)}{\bu^k_0} \right)\\
			& \leq \limsup_{k \to \infty}  2 \left(
			\int_{0}^{s}\left(({\bf}^k,{\bu}^k)_B + c\right) \dt + \skp{\bu^k_0 - \bu^k(s,\cdot)}{\bu^k_0}\right)\\
			& = 2\left( \int_{0}^{s}\left(({\bf},{\bu})_B + c\right) \dt+ \skp{\bu_0 - \bu(s,\cdot)}{\bu_0} \right),
		\end{align*}
		where we have used that $\bf^\delta = \Pi_{\delta} \bf \to \bf$ strongly in $L^2(I;B)$, cf.~\eqref{eq:f-conv}, and $\bu^k \to \bu$ strongly in $L^2(I;B)$ by Lemma~\ref{lem:unst-conv-u}. 
		Then, taking $\lim_{s \to 0}$ shows with the absolute continuity of the integral and with~\eqref{eq:ic-weak}, that 
		\begin{align*}
			\lim_{s \to 0_+} \norm{\bu(s,\cdot) - \bu_0}_B^2  = 0. 
		\end{align*}
		
		\textit{4. Step (energy identity):} 	
		While in space we have full admissibility, in time some integrability is missing due to the convective term.  
		Note however, that by the choice of $r$ in all cases we have full admissibility in the trace term. 
		For this reason, one may employ a truncation and mollification in time to prove an energy identity for almost all times, as in~\cite[Ch.~2.5]{Li.1969}: 
		
		Let $0<s_1<s_2<T$ be arbitrary but fixed, and choose a mollification parameter $\oldepsilon$ such that $\oldepsilon < \tfrac{1}{3}\min(s_1, T - s_2)$. 
		Let $\rho_{\oldepsilon}$ be the standard mollifier in time with support $[-\oldepsilon, \oldepsilon]$. 
		For   a Banach space $Y$ and some function $\bv \in L^1(I;Y)$ extended by zero to $\setR$, we denote the mollification in time by 
		\begin{align*}
			(\rho_{\oldepsilon} * \bv )(t,\cdot) \coloneqq \int_{\setR} \rho_{\oldepsilon}(t-s) \bv(s,\cdot) \,\mathrm{d}s. 
		\end{align*}
		Furthermore, we use the piecewise affine cut-off function $\psi_{\delta} \in C(I)$ defined by its values 
		\begin{align*}
			\psi_{\delta}(0 ) = \psi_{\delta}(s_1 - \delta) = \psi_{\delta}(s_2 + \delta) = \psi_{\delta}(T) = 0
			\quad \text{ and }  \quad \psi_{\delta}(s_1 + \delta) = \psi_{\delta}(s_2 - \delta) = 1. 
		\end{align*}
		We consider the smoothed and truncated function $\bu_{\oldepsilon,\delta} \in C(\overline I;W_{\divergence})$  defined by 
		\begin{align*}
			\bu_{\oldepsilon,\delta} \coloneqq \psi_\delta (\rho_{\oldepsilon} * \rho_{\oldepsilon} * (\psi_{\delta} \bu)). 
		\end{align*} 
		As test function in~\eqref{eq:unst-lim-weak-1} we use $\bv \coloneqq \bu_{\oldepsilon,\delta}(t,\cdot) $, and we integrate in $I = (0,T)$. 
		Then, first taking the limit in $\oldepsilon \to 0$, and then $\delta \to 0$ one can show that the energy identity is satisfied for a.e.~$0 < s_1 < s_2 < T$. 
		By the attainment of the initial data, the energy identity in particular holds for $s_1 = 0$. 
		The proof is complete. 
\end{proof}
	
	\begin{remark}\label{rmk:local-fortin}  
		When showing that the term $\int_0^T \skp{   \left(( \BD \bv_h) \bn \right) \cdot \bn }{ \bu^{k} \cdot \bn }_{\Gamma}\, \phi^k \dt$ in~\eqref{eq:unst-lim-7} converges to zero as $k\to \infty$, we have employed the stability estimate on the Scott--Zhang operator, see~\eqref{est:unst-stab}. 
		This stability cannot be proved by use of a global trace estimate, because $\nabla \bv_h$ is not in $H^1$. 
		Instead, its proof is based on local approximation results in Corollary~\ref{cor:SZ-stab}, which is the reason we use the Scott--Zhang interpolation operator. 
		Alternatively, one may use a Fortin operator. 
		This would have the advantage that one can work with a pressure-free formulation. 
		However, for the Fortin operator to have  local stability properties, it is not enough to use an abstract Fortin operator which exists for any inf-sup stable pair of finite element spaces thanks to the Fortin Lemma,  cf.~\cite[p.~217]{GiraultRaviart1986}. 
		Instead, one would have to use a constructive approach, which is available for most low order finite elements but usually only for homogeneous Dirichlet boundary conditions, see~\cite{BR.1985,GiraultScott2003},~\cite[Ch.~8]{BBF.2013,} and~\cite{DieningStornTscherpel2022}, and~\cite[A.2.2]{Tscherpel2018} for an overview. 
		Fortin operators with trace-preservation properties for the Scott--Vogelius element in 2D are constructed  in~\cite{ParkerSueli2025,EickmannGuzmanNeilanEtAl2025}. 
	\end{remark}
	
	\subsection{Identification of the nonlinear boundary term for $r \in [1,2]$}\label{sec:unst-id}
	To prove Theorem~\ref{thm:main-unsteady} it remains to identify that 
	\begin{alignat*}{3}
		\bsigma &= \Srel(\tr_{\tau} (\bu))  \qquad &&\text{ a.e.~on } I\times \Gamma \qquad \text{ in case~\ref{itm:case-expl-noncoerc},}\\
		\gbd(\bsigma + \lambda \tr_{\tau} (\bu), \tr_{\tau}(\bu))&=\b0 \qquad &&\text{ a.e.~on } I\times \Gamma \qquad \text{ in case~\ref{itm:case-impl-coerc}.}
	\end{alignat*}
	Then, in combination with Proposition~\ref{prop:limit-eq} we have that $(\bu,\bsigma)$ is a weak solution as in Definition~\ref{def:unst-w-sol}. 
    Since for $r\leq 2$ we even have strong convergence of the trace terms, in this case the identification is simpler than in the case $r > 2$. 
	We now proceed separately in each of the cases~\ref{itm:case-expl-noncoerc} and~\ref{itm:case-impl-coerc}, see~\eqref{eq:unst-discr-b}. 
	
	\begin{lemma}[convergence of the nonlinear trace term in case~\ref{itm:case-expl-noncoerc}]\label{lem:unst-conv-s-c1}
	In the situation of Theorem~\ref{thm:main-unsteady}, case~\ref{itm:case-expl-noncoerc}, recall that Assumption~\ref{assump:s-expl} is in action, and $r \in [1,2]$, see Assumption~\ref{assump:general}. 
  Recalling from~\eqref{def:discr-short} the sequence $(\bsigma^{k})_{k\in\mathbb{N}}$ of piecewise constant interpolants of the finite sequences $(\Srel(\tr_\tau(\bu^k_j)))_{j\in\{1,...,m_{k}\}}$ with $	\bsigma^k  \in \mathcal{L}^0_0(J_{\delta_k}; L^{2}(\Gamma)^d) $ for all $k\in\mathbb{N}$, there exists a (non-relabelled)  subsequence such that  
  \begin{align}\label{eq:streetwalker}
			\bsigma^k 
			& \rightharpoonup  \Srel(\tr_{\tau}(\bu)) \eqqcolon \bsigma  \quad 
			&&\text{ weakly in } L^{2}(I;L^{2}(\Gamma)^d) \qquad \text{ as } k \to \infty, 
		\end{align}
		where $\bu$ is as in Lemma~\ref{lem:unst-conv-u}.
	\end{lemma}

	\begin{proof}
    By Lemma~\ref{lem:unst-conv-u} we in particular have that $\bsigma^k = \Srel(\tr_\tau(\bu^k)) \wconv \bsigma$ weakly in $L^2(I;L^2(\Gamma)^d)$, and that 
    $\tr_{\tau}(\bu^k) \to \tr_{\tau}(\bu)$ strongly in $L^2(I;L^2(\Gamma)^d)$ as $k \to \infty$. 
    From the latter it follows, that there is a (non-relabelled) subsequence of $(\tr_{\tau}(\bu^k))_{k \in \mathbb{N}}$, which converges to $\tr_{\tau}(\bu)$ pointwise a.e.~in $I \times \Gamma$. 
    Since by Assumption~\ref{assump:s-expl} the function $\Srel$ is continuous, we also have that $(\Srel(\tr_{\tau}(\bu^k)))_{k \in \mathbb{N}}$ converges to $\Srel(\tr_{\tau}(\bu))$ pointwise a.e.~in $I \times \Gamma$. 
    By boundedness of $(\Srel(\tr_{\tau}(\bu^k)))_{k \in \mathbb{N}}$ in $L^2(I;L^2(\Gamma)^d)$ it converges weakly to $\Srel(\tr_{\tau}(\bu))$. 
    Then, by the uniqueness of limits it follows that $\bsigma = \Srel(\tr_{\tau}(\bu))$, which proves the claim.  
\end{proof}
	
	In case~\ref{itm:case-impl-coerc} we may use monotonicity tools to identify the limits in the nonlinear boundary term. 
	Recall that $\bsigma^k = \widehat{\bsigma}^k - \lambda \tr_{\tau}(\bu^k)$ converges to $\bsigma = \widehat{\bsigma} - \lambda \tr_{\tau}(\bu)$, cf.~Lemma~\ref{lem:unst-conv-u}. 
	Thus, it remains to identify that  
	\begin{align*}
		\gbd(\widehat \bsigma,\tr_{\tau}(\bu)) = \b0 \quad \text{ a.e.~on }I \times \Gamma. 
	\end{align*}
	
	\begin{lemma}[identification of the trace term in case~\ref{itm:case-impl-coerc}]
		\label{lem:unst-conv-s-c2}
		Under the conditions of Theorem~\ref{thm:main-unsteady} in case~\ref{itm:case-impl-coerc}, recall that in particular $r \in [1,2]$. 
        Let $\widehat{\bsigma}^k  \in \mathcal{L}^0_0(J_{\delta_k}; L^{r'}(\Gamma)^d) $
		be the piecewise constant interpolant of the sequence $(\mathbf{\mathcal{s}}_{\varepsilon_k}(\tr_\tau(\bu^k_j)))_j$, and let $\widehat{\bsigma}$ be its limit, cf.~Lemma~\ref{lem:unst-conv-u}. 
		Then we have that 
		\begin{align*}
			\gbd(\widehat{\bsigma},\tr_{\tau}(\bu)) = \b0 \quad \text{ a.e.~on } I \times \Gamma. 
		\end{align*}
		In particular, we have that $\gbd(\bsigma + \lambda \tr_{\tau}(\bu),\tr_{\tau}(\bu)) = \b0$ a.e.~on $I \times \Gamma$, see Definition~\ref{def:unst-w-sol}. 
	\end{lemma}
	\begin{proof}
		Since $r \leq 2 < 2^\sharp$, we have  with Lemma~\ref{lem:unst-conv-u} that 
		\begin{alignat}{5}
			\tr_\tau (\bu^k) &\to \tr_\tau(\bu) \qquad &&\text{ strongly in } L^r(I;L^r(\Gamma)^d),
			\quad
			&& \\
			\widehat{\bsigma}^k &\wsconv \widehat{\bsigma} \qquad &&\text{ weakly* in } L^{r'}(I;L^{r'}(\Gamma)^d)&& 
		\end{alignat}	
		as $k \to \infty$.
		In this case due to the strong convergence of the traces we may directly apply Lemma~\ref{lem:conv-simple} which is based on Assumption~\ref{assump:gbd-mon} and~\ref{assump:gbd-reg}~\ref{itm:gbd-eps-0}--\ref{itm:gbd-eps-approx}. 
	\end{proof}
	
	\subsection{Extensions for $r>2$}\label{sec:extensions_rgeq2}

    Thanks to the strong convergence of $\tr_\tau(\bu^k)$ in $L^r(I;L^r(\Gamma)^d)$  for $r\in [1,2]$ the identification of the nonlinearities in Lemmas~\ref{lem:unst-conv-s-c1} and~\ref{lem:unst-conv-s-c2} is relatively straightforward. 
    In contrast, for $r\in (2,2^\sharp)$ in the coercive case~\ref{itm:case-impl-coerc} by Lemma~\ref{lem:unst-conv-u} we have in general merely strong convergence in $L^{p}(I;L^r(\Gamma)^d)$ for $p<r$ of the traces. 
    This does not suffice to apply the convergence Lemma~\ref{lem:conv-simple}.
   In the following we present special situations in which the identification for a certain range of $r>2$ can still be established. 
   Recall that the a priori estimates and convergence results in Sections~\ref{sec:unst-apriori}--\ref{sec:unst-limit} hold in the situation of Assumption~\ref{assump:general}. 

For the time-dependent problems in Subsection~\ref{sec:extensions_rgeq2_beta}--\ref{sec:skew-sym}  we shall only consider the coercive case~\ref{itm:case-impl-coerc}. 
   This is due to the fact that we in the non-coercive case the condition $r\leq 2$ is key to obtain the estimates~\eqref{est:unst-ap-5-c1} and~\eqref{est:unst-ap-5} and integrability in time of the boundary term. 
   
   In the stationary case in Section~\ref{sec:steady} below, we address both cases~\ref{itm:case-impl-coerc} and~\ref{itm:case-expl-noncoerc}.   

\subsubsection{Dynamic boundary conditions ($\beta>0$)} \label{sec:extensions_rgeq2_beta}
All previous a priori estimates and convergence results hold uniformly in $\beta \geq 0$, and hence in particular for $\beta=0$.
In the special case $\beta>0$  we obtain better control on the trace term. 
Indeed, additionally $(\tr(\bu^k))_{k \in \mathbb{N}}$ is bounded $L^\infty(I;L^2(\Gamma)^d)$ with constant depending on $\beta$. 
Recall that by~\eqref{est:ap-pwconst} the sequence $(\tr(\bu^k))_{k \in \mathbb{N}}$ is bounded in $L^2(I;L^p(\Gamma)^d)$ for $p \in [2,\infty)$ such that $p\leq 2^\sharp$. 
In combination, by function space interpolation one can show that
$(\tr(\bu^k))_{k \in \mathbb{N}}$ is bounded in $L^s(I;L^s(\Gamma)^d)$ for $s  = \tfrac{4}{p}(p-1)$.  
Again by interpolation and the fact that $\tr(\bu^k) \to \tr(\bu)$ converges strongly in $L^2(I;L^2(\Gamma)^d)$ it follows that 
\begin{align}\label{eq:beta_convergence}
    \tr(\bu^k) \to \tr(\bu) 
\quad 
\text{strongly in } L^r(I;L^r(\Gamma)^d)\; \text{ as } k \to \infty,
\end{align} 
provided that 
\begin{equation}\label{eq:beta_range}
 r < \begin{cases}
    3 \quad & \text{ for }d=3, \\
    4 & \text{ for }d=2.
  \end{cases}
\end{equation}

This allows us to conclude the following:

\begin{proposition}\label{prop:beta}
  In the situation of Assumption~\ref{assump:general} in  case~\ref{itm:case-impl-coerc} assume that  $\beta >0$ and that $r$  as in~\eqref{eq:beta_range}. 
  Then, the statement of Theorem~\ref{thm:main-unsteady} holds, and moreover the convergence~\eqref{eq:beta_convergence} is available. 
\end{proposition}

\subsubsection{Explicit coercive relation}\label{sec:explicit}

Let us consider the explicit coercive case, i.e., in the situation of Assumption~\ref{assump:gbd-mon}, see case~\ref{itm:case-impl-coerc}, assume additionally that the graph can be represented by a continuous function $\Srel\colon \R^d\to \R^d$, cf.~\eqref{eq:gbd-expl}. 
Comparing with Assumption~\ref{assump:general} note that the a~priori and convergence results from the previous sections are available. 
In this case for any $r \in [1,\infty)$ as in Lemma~\ref{lem:unst-conv-u} we obtain 
\begin{align*}
\Srel(\tr_\tau(\bu^k)) \wsconv \bsigma \quad \text{ weakly* in } L^{r'}(I;L^{r'}(\Gamma)^d) \qquad \text{ as } k \to \infty. 
\end{align*}
Furthermore, by Lemma~\ref{lem:unst-conv-u} the traces $\tr_{\tau}(\bu^k)\to \tr_{\tau}(\bu)$ converge strongly in $L^2(I;L^2(\Gamma)^d)$, and thus there exists a subsequence, which converges pointwisely as $k \to \infty$. 
Thus, since $\Srel$ is continuous, as in Lemma~\ref{lem:unst-conv-s-c1} we may conclude that $\bsigma = \Srel(\tr_{\tau}(\bu))$. 

\begin{proposition}\label{prop:explicit}
  In the situation of Assumption~\ref{assump:general} in case~\ref{itm:case-impl-coerc} assume that the coercive boundary relation is represented by a continuous function as in~\eqref{eq:gbd-expl}. 
  Then, without further restrictions on $r\in [1,\infty)$, the statements of Theorem~\ref{thm:main-unsteady} hold. 
\end{proposition}

\subsubsection{Skew-symmetric Nitsche method $(\theta = -1)$}
\label{sec:skew-sym}
Now we consider the situation in which only the weak convergence of $\tr_\tau(\bu^k)$ in $L^r(I;L^r(\Gamma)^d)$ is employed, and the identification is carried out using the Minty-type convergence Lemma~\ref{lem:minty}. 
To prove the required limsup estimate~\eqref{est:minty-lsc} we exploit an energy identity for $\bu$ and an energy inequality for $\bu^k$. 
The latter is available for the skew-symmetric variant of the Nitsche method, i.e., \eqref{eq:unst-discr} with $ \theta=-1$, because in~\eqref{est:unst-ap-1} the  terms involving the velocity gradient on the boundary vanish. 

Let us note that the argument does not not make use of the pointwise convergence of the traces, and hence there might be room for improvement.  

\begin{proposition} \label{prop:skew-symmetric-Nitsche}
  In the situation of Assumption~\ref{assump:general} consider the case~\ref{itm:case-impl-coerc}. 
  Assume additionally, that $\theta = -1$. 
  Then, without further restrictions on $r \in [1,\infty)$, the statement of Theorem~\ref{thm:main-unsteady} holds. 
\end{proposition} 
\begin{proof}
It suffices to consider the case $r>2$.
By Lemma~\ref{lem:unst-conv-u} in this case we have 
		\begin{alignat}{5}
			\tr_\tau (\bu^k) &\wconv \tr_\tau(\bu) \qquad &&\text{ weakly in } L^r(I;L^r(\Gamma)^d),
			\quad
			&& \\
			\widehat{\bsigma}^k &\wsconv \widehat{\bsigma} \qquad &&\text{ weakly* in } L^{r'}(I;L^{r'}(\Gamma)^d)&& 
		\end{alignat}	
		as $k \to \infty$.  
	In this case we apply the Minty type convergence Lemma~\ref{lem:minty},  which is based on Assumption~\ref{assump:gbd-mon} and~\ref{assump:gbd-reg}~\ref{itm:gbd-eps-0}--\ref{itm:gbd-eps-approx}. 
		With the fact that $\widehat{\bsigma}^k = \Seps(\tr_{\tau}(\bu^k))$ and thanks to the weak($^*$) convergence, it remains to verify that~\eqref{est:minty-lsc} holds, i.e., that 
		\begin{align}\label{est:minty-lsc-claim}
			\limsup_{k \to \infty}	\int_{J} \int_{\Omega}  \widehat{\bsigma}^k \cdot \tr_{\tau}(\bu^k)  \dx \dt \leq\int_{J} \int_{\Omega} \widehat{\bsigma} \cdot \tr_{\tau}(\bu)  \dx \dt,
		\end{align}
		for suitable sets $J \subset I$. 
		This can be derived by the energy identity in Proposition~\ref{prop:limit-eq} and the corresponding energy inequality for $\bu^k$.
		For a.e.~$0< s < T$ by Propposition~\ref{prop:limit-eq} we have 
		\begin{equation}\label{est:minty1}
			\begin{aligned}
				\int_{0}^{s} \skp{\widehat \bsigma}{\tr_{\tau} (\bu)}_{\Gamma} \dt 
				= \int_{0}^{s} (\bf,\bu)_B \dt  &- 2 \nu \int_{0}^{s}  \norm{\BD \bu}_{L^2(\Omega)}^2 \dt   
				+  \lambda \int_{0}^{s}  \norm{\tr_{\tau} (\bu)}^2_{L^2(\Gamma)} \dt   \\
				& + 	\tfrac{1}{2} \norm{\bu_0}_{B}^2  
				- \tfrac{1}{2} \norm{\bu(s)}_{B}^2. 
			\end{aligned}
		\end{equation}
		Let us fix an arbitrary value $s \in I$. 
		Using the identity~\eqref{est:unst-ap-2} in~\eqref{est:unst-ap-1} and integrating over $(0,s)$ we have 	

		\begin{align*}
			\tfrac{1}{2} \norm{\widetilde \bu^k(s)}_{B}^2 &-  	\tfrac{1}{2}  \norm{\bu^k_0}_{B}^2	+ 2 \nu \int_{0}^{s} \norm{\BD \bu^k}_{L^2(\Omega)}^2 \dt 
			+ \int_{0}^{s}\skp{\widehat{\bsigma}^k 
			}{\tr_{\tau}(\bu^k_j) }_{\Gamma} \dt\\
			&- 	\lambda \int_{0}^{s} \norm{\tr_{\tau}(\bu^k)}_{L^2(\Gamma)}^2 \dt
			+ \nu \alpha  \int_{0}^{s} \norm{h_{\Gamma}^{-1/2} \bu^k_j \cdot \bn}_{L^2(\Gamma)}^2  \dt\leq   \int_{0}^{s}({\bf}^k,{\bu}^k)_B \dt.
		\end{align*}
    Recall that $\widetilde{\bu}^k$ represents the piecewise linear interpolant of the discrete velocities.
		Rearranging, we have 
		\begin{equation}\label{est:minty-3}
			\begin{aligned}
				\int_{0}^{s}\skp{\widehat{\bsigma}^k 
				}{\tr_{\tau}(\bu^k_j) }_{\Gamma} \dt 
				&\leq 
				- \tfrac{1}{2} \norm{\widetilde \bu^k(s)}_{B}^2 
				+  	\tfrac{1}{2}  \norm{\bu^k_0}_{B}^2	- 2 \nu \int_{0}^{s} \norm{\BD \bu^k}_{L^2(\Omega)}^2 \dt 
				\\
				& \quad 
				+  	 \int_{0}^{s} \norm{\tr_{\tau}(\bu^k)}_{L^2(\Gamma)}^2 \dt
				- \nu \alpha  \int_{0}^{s}\norm{ h_{\Gamma}^{-1/2} \bu^k \cdot \bn}_{L^2(\Gamma)}^2 \dt+   \int_{0}^{s}({\bf}^k,{\bu}^k)_B \dt\\
				& \eqqcolon \mathrm{I}_{1}+...+\mathrm{I}_{6}. 
			\end{aligned}
		\end{equation}
    We consider term by term on the right-hand side, starting from $\mathrm{I}_1$. 
		Since $\abs{\widetilde{\bu}^k_j - \bu^k_j}  = \abs{\bu^k_j - \bu^k_{j-1}}$, from the estimates in~\eqref{est:unst-ap-10} it follows that 
		\begin{align*}
			\widetilde{\bu}^k - \bu^k \to 0 \quad \text{ strongly in } L^2(I;B) 
            \quad \text{ as  } k \to \infty. 
		\end{align*}
		Hence, from $\bu^k \to \bu$ strongly in $L^2(I;B)$, by Lemma~\ref{lem:unst-conv-u} it follows that $\widetilde{\bu}^k \to \bu$ strongly in $L^2(I;B)$ as $k \to \infty$.  
		In particular, there is a subsequence such that for a.e.~$s \in I$ one has  
		\begin{align*}
			\mathrm{I}_1 = \norm{\widetilde{\bu}^k(s)}_B^2 \to \norm{\bu(s)}_{B}^2 \quad \text{ as }k \to \infty. 		
		\end{align*}
		For the second term, it follows directly by~\eqref{eq:conv-u0} that 
		\begin{align*}
			\mathrm{I}_2 = \norm{{\bu}^k_0}_B^2 \to \norm{\bu_0}_{B}^2 \quad \text{ as }k \to \infty. 		
		\end{align*}
		On the third term, by the weak convergence of $\BD \bu^k \to \BD \bu$ weakly in $L^2(I;L^2(\Omega)^{d \times d})$, see Lemma~\ref{lem:unst-conv-u}, and lower semicontinuity of the norm with respect to weak convergence, it follows that 
		\begin{align*}
			\limsup_{k \to \infty} \mathrm{I}_3	
			& = \limsup_{k \to \infty}\left( - 2 \nu \norm{\BD \bu^k}_{L^2((0,s)\times \Omega)}^2 \right)  = - 2 \nu \liminf_{k \to \infty }  \norm{\BD \bu^k}_{L^2((0,s)\times \Omega)}^2  \\
			&\leq - 2 \nu  \norm{\BD \bu}_{L^2((0,s)\times \Omega)}^2. 
		\end{align*}
		On the fourth term, by the convergence of $\tr_{\tau}(\bu^k) \to \tr_{\tau}(\bu)$  strongly in $L^2(I;L^2(\Gamma)^d)$, see Lemma~\ref{lem:unst-conv-u}, we find that 
		\begin{align*}
			\mathrm{I}_4 = \norm{\tr_{\tau}(\bu^k)}_{L^2((0,s);L^2(\Gamma))}^2 \to \norm{\tr_{\tau}(\bu)}_{L^2((0,s);L^2(\Gamma))}^2
		\end{align*}
		as $k \to \infty$. 
		On the 5th term we simply use that 
		\begin{align*}
			\mathrm{I}_5 = - \nu \alpha \norm{h_{\Gamma}^{-1/2}\bu^k \cdot \bn}_{L^2((0,s);L^2(\Gamma))} \leq 0. 
		\end{align*}
    The term $I_6$ can be readily handled with the available convergences, and we thus deduce that
		\begin{equation}\label{est:minty-4}
			\begin{aligned}
				\limsup_{k \to \infty} \int_{0}^{s}\skp{\widehat{\bsigma}^k 
				}{\tr_{\tau}(\bu^k_j) }_{\Gamma} \dt 
				\leq & 
				- \tfrac{1}{2} \norm{ \bu(s)}_{B}^2 
				+  	\tfrac{1}{2}  \norm{\bu_0}_{B}^2	- 2 \nu \int_{0}^{s} \norm{\BD \bu}_{L^2(\Omega)}^2 \dt 
				\\
				&
				+  	 \int_{0}^{s} \norm{\tr_{\tau}(\bu)}_{L^2(\Gamma)}^2 \dt
				+   \int_{0}^{s}({\bf},{\bu})_B \dt\\
				&= 
				\int_{0}^s \skp{\widehat{\bsigma}}{\tr_{\tau}(\bu)}_{\Gamma} \dt,
			\end{aligned}
		\end{equation}
		where we have used~\eqref{est:minty1}. 
		This shows that the conditions of the Minty type convergence Lemma~\ref{lem:minty} are fulfilled, and thus concludes the proof. 
	\end{proof}

	\begin{remark}
		The reason why we focus on the skew-symmetric Nitsche method ($\theta=-1$) for $r>2$ is the fact that in this case the energy identity~\eqref{est:unst-ap-1} does not contain the term 
		\begin{align*}
\skp{(\Du^k \, \bn)\cdot \bn}{\bu^k \cdot \bn}_{\Gamma},
		\end{align*}
On this term we do not have enough information to show that it vanishes as $k \to 0$. 
	\end{remark}

	\begin{table}[ht!] \small
 \begin{TAB}(r)[5pt]{|c|c|c|c|c|}{|cc|cc|c|c|c|}
		convergence & boundary condition   & $\theta$ 
        & $\beta  $ &  range of $r$ \\
       result 	 & Assumption~\ref{assump:general}   &  &  &  
		\\
		Thm.~\ref{thm:main-unsteady} &    \ref{itm:case-expl-noncoerc}   &   $\theta \in [-1,1]$ 	
        &  $\beta \geq 0$ & $ [1,2]$\\
        & 
        \ref{itm:case-impl-coerc}
        &   $\theta \in [-1,1]$ 	
     &  $\beta \geq 0$ & $ [1,2]$\\
    Prop.~\ref{prop:beta} & \ref{itm:case-impl-coerc}  & $\theta \in [-1,1]$ 
    & $\beta >0$ & 
    $\begin{aligned}
        [1,3) & \;\text{ if } d = 3\\
        [1,4) & \; \text{ if } d = 2
        \end{aligned}
    $
    \\
    Prop.~\ref{prop:explicit} & \ref{itm:case-impl-coerc}  \text{ and explicit} & $\theta \in [-1,1]$ 
    & $\beta\geq 0$ & $[1,\infty)$ 
    \\
     Prop.~\ref{prop:skew-symmetric-Nitsche} & \ref{itm:case-impl-coerc}  & $ \theta = -1$ 
     & $\beta \geq 0$  & $[1,\infty)$
    \end{TAB}
	\caption{Overview on convergence results in the time-dependent case, all of them conditionally nonmonotone.  
	}
	\label{tab:overview}
\end{table} 

\subsubsection{The steady problem}\label{sec:steady}
For the stationary case the arguments are much simpler, and convergence is available for a larger range of $r$. 
The steady formulation related to the weak formulation in Definition~\ref{def:unst-w-sol} is obtained in a straightforward manner.
Note that in this case the solution is sought in $(\bu,\bsigma) \in H^1_{\bn, \diver}(\Omega) \times L^{r'}(\Gamma)^d$.
The corresponding discrete formulation is obtained by formally letting $\delta\to \infty$ in~\eqref{eq:unst-discr}; 
thus the terms corresponding to the time derivative are not present and only two approximation indices are required: $\bu^k := \bu^{\varepsilon_k,h_k}$, where $(\varepsilon_k,h_k)\to 0$ as $k\to \infty$. 

\begin{assumption}[general setting steady]\label{assump:general-steady}
Let $\Omega \subset \setR^d$ for $d \in \{2,3\}$, be a bounded and open Lipschitz domain with polyhedral boundary $\Gamma$. 	
	Let $\nu>0$  and 
	let $ \bf \in L^2(\Omega)^d$ be given. 
	Let $r \in [1,\infty)$ in both cases 
    \begin{align*}
   \begin{cases} 
     \ref{itm:case-expl-noncoerc}, \quad \text{see~Assumption~\ref{assump:s-expl}},\\
\ref{itm:case-impl-coerc},  \quad \text{see~Assumption~\ref{assump:gbd-mon}},
    \end{cases}
    \end{align*}
    and let $\lambda\geq 0$ be sufficiently small, and let $\alpha>0$ be sufficiently large.  
    Assume that $(\tria)_{h>0}$ is a sequence of shape-regular triangulations of $\Omega$ and let $(X_h,Q_h)$ be pairs of mixed finite element spaces satisfying Assumption~\ref{assump:fem}. 
	Furthermore, in case~\ref{itm:case-impl-coerc} let Assumption~\ref{assump:gbd-reg}~\ref{itm:gbd-eps-0}--\ref{itm:gbd-eps-approx} be satisfied.
\end{assumption}

Under those assumptions one can show analogously to above certain a priori estimates, and conclude the following convergence results 
\begin{alignat*}{5}\label{eq:steady_convergence}
  \bu^k&\to\bu &&
  \text{weakly in }H^1(\Omega)^d
              && \\
\tr_{\tau}(\bu^k) 
&\to 
\tr_{\tau}(\bu) \qquad && \text{strongly in }L^p(\Gamma)^d\qquad
&&\text{for any }
p<2^\sharp,\\
\tr_{\tau}(\bu^k) 
&\wsconv
\tr_{\tau}(\bu) 
\quad 
&&\text{weakly* in }
L^p(\Gamma)^d
\quad
&&\text{for any }
p\leq 2^\sharp,\\
 \Srel(\tr_{\tau}(\bu^k)) 
&\wsconv
\Srel(\tr_{\tau}(\bu)) \quad && \text{weakly* in }L^{r'}(\Gamma)^d\quad
&&\text{if }
r \leq 2^\sharp \text{ in case~\ref{itm:case-expl-noncoerc}},\\
\widehat{\bsigma}^k
&\wsconv
\widehat{\bsigma} \quad && \text{weakly* in }L^{r'}(\Gamma)^d\quad
&& \text{in case~\ref{itm:case-impl-coerc}},\\
\tr_{\tau}(\bu^k) 
&\wconv 
\tr_{\tau}(\bu) \quad && \text{weakly in }L^{r}(\Gamma)^d\quad
&& \text{in case~\ref{itm:case-impl-coerc}}
\end{alignat*}
as $k \to \infty$, and $\bu$ can be shown to be a solution to the steady problem.

\begin{theorem}[convergence steady]
\label{thm:main-steady}
In the setting of Assumption~\ref{assump:general-steady} the discrete approximations converge to a weak solution to the steady system in the following cases
\begin{enumerate}[label = (\roman*)]
     \item \label{itm:st-1} in case~\ref{itm:case-expl-noncoerc} if $r \in [1,2]$;
     \item \label{itm:st-2} in case~\ref{itm:case-impl-coerc} if $r\in [1, 2^\sharp)$;
     \item \label{itm:st-3} in case~\ref{itm:case-impl-coerc} and the boundary relation is represented by a continuous (explicit) relation as in~\eqref{eq:gbd-expl}, without further restriction on $r \in [1,\infty)$;
     \item  \label{itm:st-4} in case~\ref{itm:case-impl-coerc} and $\theta = -1$ 
without further restriction on $r \in [1,\infty)$. \end{enumerate}
\end{theorem}

\begin{proof}
    The proof of~\ref{itm:st-1} follows as in Lemma~\ref{lem:unst-conv-s-c1}, \ref{itm:st-2} as in Lemma~\ref{lem:unst-conv-s-c2}, \ref{itm:st-3} as in Proposition~\ref{prop:explicit}, and~\ref{itm:st-4} as in Proposition~\ref{prop:skew-symmetric-Nitsche}. 
\end{proof}

\begin{table}[ht!] \small
	\begin{TAB}(r)[5pt]{|c|c|c|c|}{|cc|cccc|}
		convergence & boundary condition   & $\theta$ 
          &  range of $r$ \\
       result 	 & Assumption~\ref{assump:general-steady}   &   &  
		\\
		Thm.~\ref{thm:main-steady} &    \ref{itm:case-expl-noncoerc}   &   $\theta \in [-1,1]$ 	
         & $ [1,2]$\\
        & 
        \ref{itm:case-impl-coerc}
        &   $\theta \in [-1,1]$ 	
     & $ [1,2^\sharp)$\\
     & \ref{itm:case-impl-coerc}  \text{ and explicit} & $\theta \in [-1,1]$    & $[1,\infty)$ 
    \\
      & \ref{itm:case-impl-coerc}  & $ \theta = -1$ 
       & $[1,\infty)$
    \end{TAB}
	\caption{Overview on convergence results in the steady case, all of them conditionally nonmonotone.  
	}
	\label{tab:overview-steady}
\end{table} 
	    
	\section{Numerical experiments}
	\label{sec:num-exp}
	\noindent
	In this section, we present numerical examples for several constitutive relations; the computational domain is chosen as the unit square $\Omega = (0,1)^2$ in dimension $d = 2$ with viscosity $\nu = 1$.
	The slip boundary is taken to be the top wall~$\Gamma_s = (0,1)\times \{1\}$ and no-slip conditions are imposed on the complement~$\Gamma_d\coloneqq \partial\Omega \setminus \Gamma_s$.
	All examples are implemented in \texttt{firedrake}~\cite{Firedrake} using the Taylor--Hood finite element pair with 50403 spatial degrees of freedom,  time step size~$\delta=0.005$, and  penalty parameter~$\alpha=10$.
	The discrete systems are linearised by means of Newton's method with the CP line search from PETSc~\cite{petsc-user-ref}, and the linear systems were solved using MUMPS~\cite{MUMPS}. 
	The code used to implement the numerical experiments, including the exact components of \texttt{firedrake} that were used, has been archived in Zenodo~\cite{zenodo}.

	\subsection{Steady flow with nonmonotone slip}
	\label{sec:num-exp-1}
	
	In this example we consider the steady system with constitutive relations modelling two types of nonmonotone slip behaviour, one of them smooth and the second containing an activation criterion. 

 \subsubsection{Explicit example} 
	The smooth constitutive relation is inspired by a constitutive relation proposed in the bulk by Le Roux and Rajagopal~\cite{LeRoux2013}, see Example~\ref{ex:gbd-nonmon}: 
	\begin{equation}\label{eq:non-monotone-slip}
		\bsigma = \left(a (1+ b|\bu_\tau|^2)^\theta + c \right)\bu_\tau.
	\end{equation}
	We choose~$a=1.0$, $b=0.1$, $c=0.001$, $\theta=-0.75$; with this choice the relation is indeed nonmonotone. 
	Let~$\hat{\bu}$ and $\hat{p}$ be a Taylor--Green vortex
	\begin{equation}\label{eq:taylor_green}
		\begin{aligned}
			\hat{\bu}(x,y) &\coloneqq \Lambda_\star [\sin(\pi x) \cos(\pi y), - \cos(\pi x) \sin(\pi y)]^\top, \\
			\hat{p}(x,y) &\coloneqq  \frac{\Lambda_\star}{4}(\cos(2\pi x) + \sin(2 \pi y)),
		\end{aligned}
	\end{equation}
	for $(x,y)\in \Omega$, where $\Lambda_\star>0$ is a parameter determining the amplitude of $\hat{\bu}$ and $\hat{p}$. 
	For large $\Lambda_\star$, which leads to larger magnitude of $\bu_\tau$ on $\Gamma_s$, the solution is expected to exhibit nonmonotone behaviour.
	The forcing term in the momentum equation is determined by $\bm{f} \coloneqq - 2\diver( \BD\hat{\bu}) + \diver(\hat{\bu}\otimes \hat{\bu}) + \nabla \hat{p}$, i.e., we work with  a manufactured solution. 
	
	\subsubsection{Implicit example}
	\label{ex:non-monotone-tresca}
	For the non-smooth case we consider the boundary condition in~\cite{Fang2020}, see Example~\ref{ex:Fang}. 
	We follow~\cite{JHYW.2018} for the definition of the source term, by considering the functions
	\begin{equation}\label{eq:exact_solution}
		\begin{aligned}
			\hat{\bu}(x,y) &\coloneqq  \Lambda_\star[20x^2 (x-1)^2 y (y-1) (2y-1), -20x(x-1)(2x-1)y^2 (y-1)^2]^\top, \\
			\hat{p}(x,y) &\coloneqq 20\Lambda_\star(2x-1)(2y-1),
		\end{aligned}
	\end{equation}
	and then set $\bm{f}= -\diver(2\nu_\star \BD\hat{\bu}) + \diver(\hat{\bu}\otimes \hat{\bu}) + \nabla \hat{p}$, as before.
	We consider the following constitutive relation, as  proposed in~\cite{Fang2020}:
	\begin{equation}\label{eq:nonsmooth_nonmonotone}
		\begin{array}{rlr}
			\bsigma &= \mu(|\bu_\tau|)\frac{\bu_\tau}{|\bu_\tau|}  & \text{ if }\bu_\tau \neq 0  \\
			|\bsigma| &\leq \mu(0)  & \text{ if }\bu_\tau = 0
		\end{array}
		\qquad \text{ with  }
		\mu(t) = (a-b)e^{-\beta t} + b,
	\end{equation}
	for parameters $a=1.6$, $b=1.5$, $\beta=10$. 
	This describes a nonmonotone variant of the Tresca constitutive relation, since the `activation parameter' decreases with the magnitude of the tangential velocity. 
	We consider the regularisation
	\begin{equation}
		\bsigma_\epsilon = \Seps(\bu_\tau) \coloneqq \mu(|\bu_\tau|) \frac{\bu_\tau}{\sqrt{\varepsilon^2 + |\bu_\tau|^2}} 
		\qquad
		\text{ for } \varepsilon>0,
	\end{equation}
	and we choose $\varepsilon = 0.0002$ in the computations. 
	Note that $\bv \mapsto \mu(\abs{\bv}) \frac{\bv}{\abs{\bv}} + \lambda \bv$ for $\lambda = \beta(a - b)$ is monotone;
	regularising this as above yields that $\Seps$ satisfies Assumption~\ref{assump:gbd-reg} (see Example~\ref{ex:reg-rel}).
	
	Let us stress that the analysis from~\cite{Fang2020}, see also~\cite{HCJ.2021}, is carried out for a formulation using a variational inequality, which allows for multivalued solutions. 
	In contrast, in our framework the solutions are single-valued functions and the constitutive relation is imposed pointwise at the boundary.
	
	Note that $\hat{\bu}|_{\partial\Omega}=\bm{0}$, and hence $\hat{\bu}$ is an exact solution to the equation subject to no-slip boundary conditions.
	On the other hand, the wall stress at the slip part of the $\Gamma_s$ satisfies 
	\begin{equation}
		\max_{\overline{\Gamma_s}} |\bsigma| = \max_{x\in [0,1]}|20  \Lambda_\star x^2(x-1)^2|
		= \tfrac{5}{4}  \Lambda_\star.
	\end{equation}
	In particular, for sufficiently small constant $\Lambda_\star$ the wall stress $|\bsigma|$  does not reach the
	critical value $|\mu(\bu_\tau)|$, and thus the solution is in fact the no-slip solution~\eqref{eq:exact_solution}. 
	For larger values of $\Lambda_\star$ 
	this is no longer the case and nonmonotone slip occurs. 
	\medskip 
	
	Figure~\ref{fig:CR_smooth} shows the values obtained for $\abs{\bu_\tau}$ and $\abs{\bsigma}$ on the boundary $\Gamma_s$ in the smooth case~\eqref{eq:non-monotone-slip}, and the $x$-dependence along the wall is depicted in Figure~\ref{fig:slice_smooth}. 
	In particular one can observe in Figure~\ref{fig:CR_smooth} (B) that the stress is nonmonotone in the tangential velocity. 
	The corresponding plots for the non-smooth problem with~\eqref{eq:nonsmooth_nonmonotone} are shown in Figures~\ref{fig:CR_nonsmooth} and~\ref{fig:slice_nonsmooth}. The results match with the behaviour expected from the model.
	
	\begin{figure}
		\centering
		\subfloat[C][{\centering {\normalsize $\Lambda_\star=1$}.}]{{%
				\includegraphics[width=0.49\textwidth]{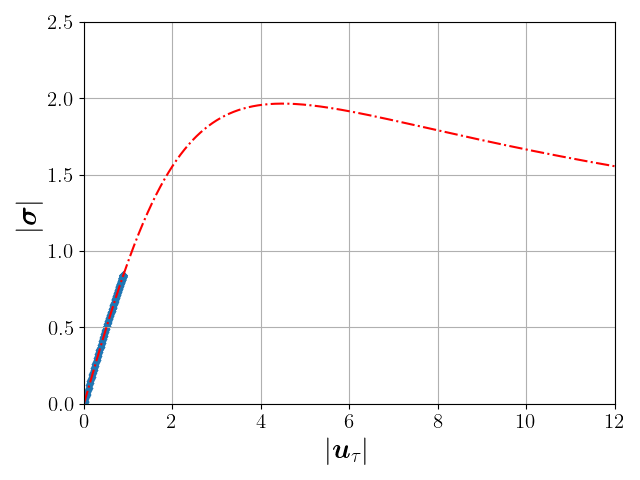}%
		}}%
		\subfloat[D][{\centering {\normalsize $\Lambda_\star = 10$}.}]{{%
				\includegraphics[width=0.49\textwidth]{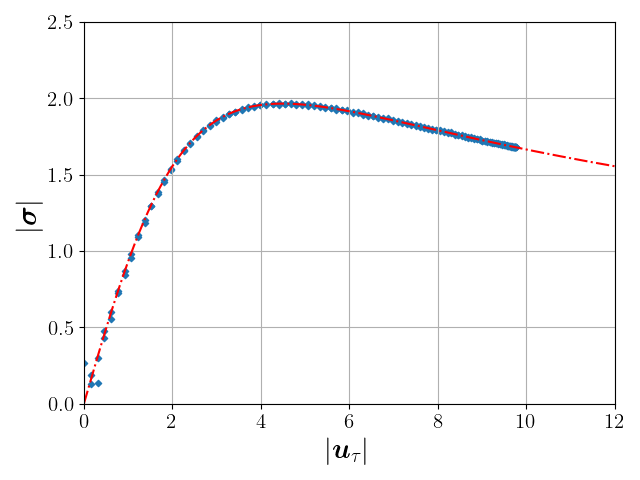}%
		}}\\
		\caption{Exact (red) and computed (blue) constitutive relation on $\Gamma_s$ for the smooth relation~\eqref{eq:non-monotone-slip}.}%
		\label{fig:CR_smooth}
	\end{figure}
	
	\begin{figure}
		\centering
		\subfloat[C][{\centering {\normalsize $\Lambda_\star=1$}.}]{{%
				\includegraphics[width=0.49\textwidth]{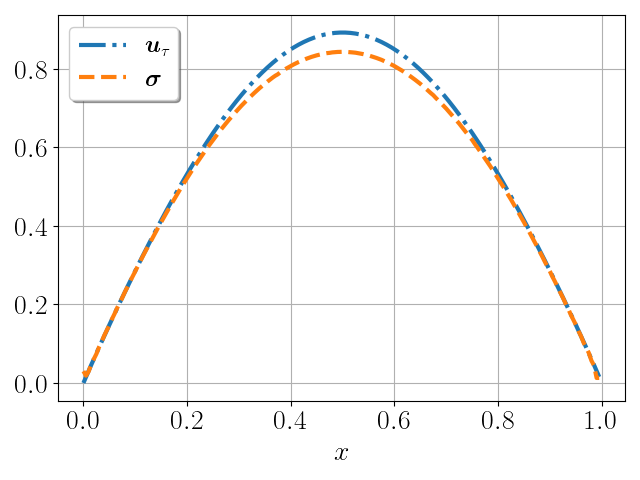}%
		}}%
		\subfloat[D][{\centering {\normalsize $\Lambda_\star = 10$}.}]{{%
				\includegraphics[width=0.49\textwidth]{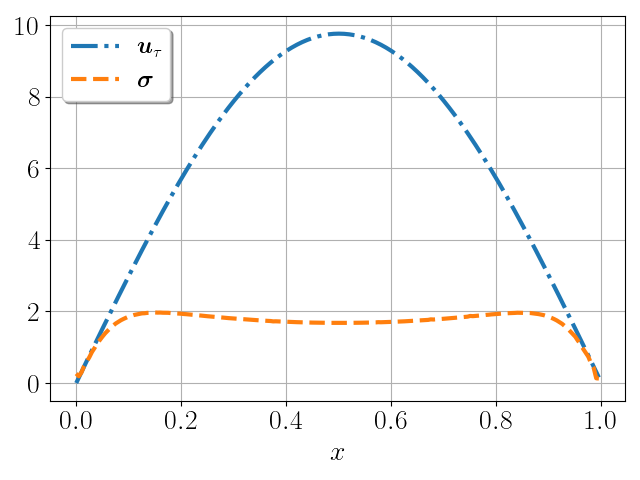}%
		}}\\
		\caption{Wall stress and tangential velocity on $\Gamma_s$ for the smooth relation~\eqref{eq:non-monotone-slip}.}%
		\label{fig:slice_smooth}
	\end{figure}
	
	\begin{figure}
		\centering
		\subfloat[C][{\centering {\normalsize $\Lambda_\star=0.6$}.}]{{%
				\includegraphics[width=0.49\textwidth]{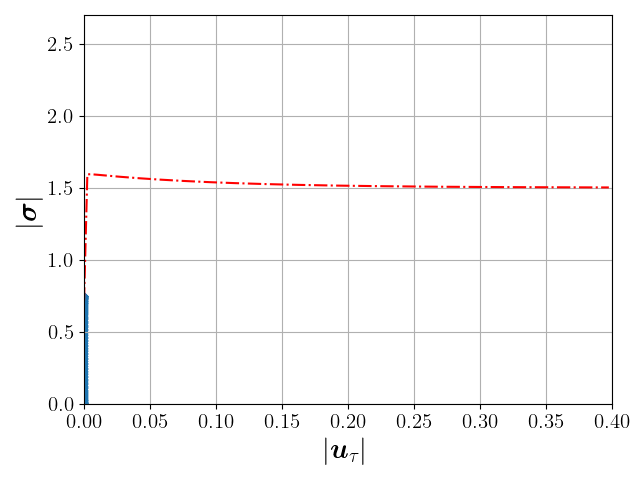}%
		}}%
		\subfloat[D][{\centering {\normalsize $\Lambda_\star = 5$}.}]{{%
				\includegraphics[width=0.49\textwidth]{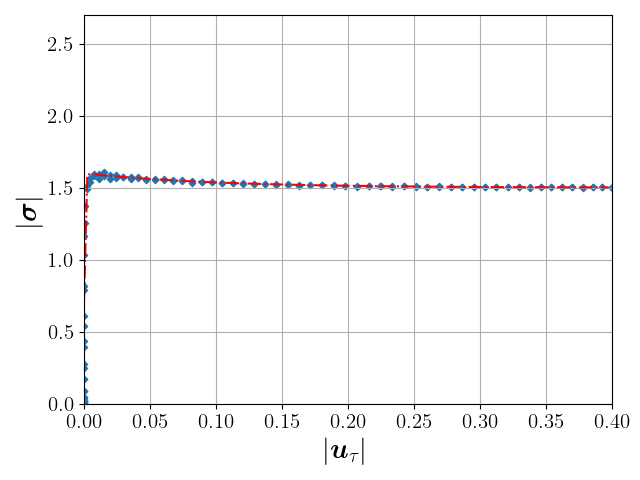}%
		}}\\
		\caption{Exact (red) and computed (blue) constitutive relation on $\Gamma_s$ for the non-smooth problem~\eqref{eq:nonsmooth_nonmonotone}.}%
		\label{fig:CR_nonsmooth}
	\end{figure}
	
	\begin{figure}
		\centering
		\subfloat[C][{\centering {\normalsize $\Lambda_\star=0.6$}.}]{{%
				\includegraphics[width=0.49\textwidth]{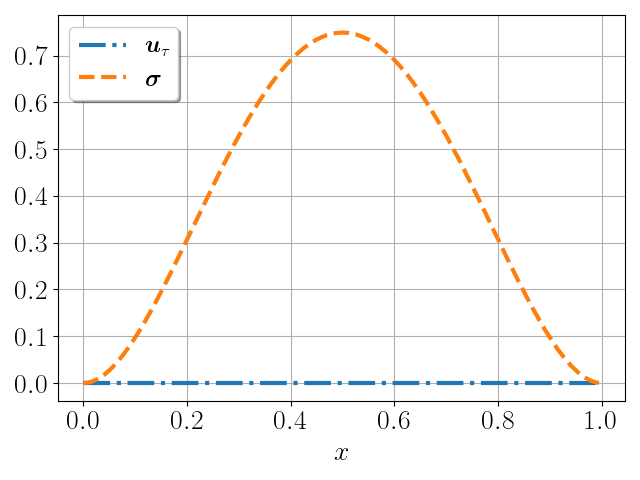}%
		}}%
		\subfloat[D][{\centering {\normalsize $\Lambda_\star = 5$}.}]{{%
				\includegraphics[width=0.49\textwidth]{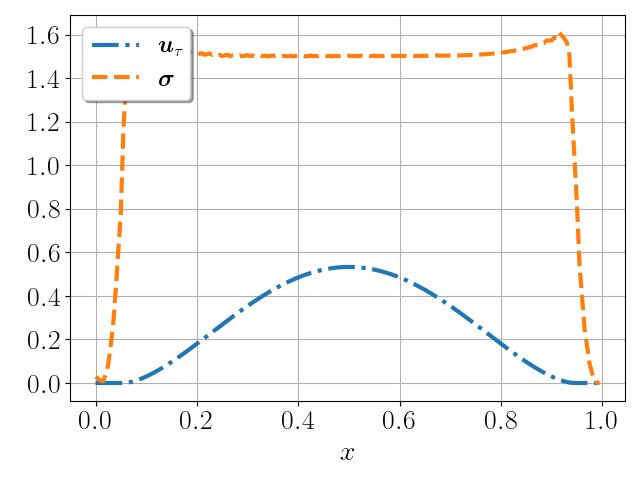}%
		}}\\
		\caption{Wall stress and tangential velocity on $\Gamma_s$ for the non-smooth problem~\eqref{eq:nonsmooth_nonmonotone}.}%
		\label{fig:slice_nonsmooth}
	\end{figure}

	\subsection{Tresca and stick-slip}
	\label{sec:num-exp-2}
	
	Let us consider the unsteady problem with stick-slip or Tresca boundary condition, see Example~\ref{ex:Tresca}, given by
	\begin{equation}\label{eq:stick_slip}
		\begin{array}{rlr}
			\bsigma &= \gamma_\star \bu_\tau + \mu_\star \frac{\bu_\tau}{|\bu_\tau|}  & \text{ if }\bu_\tau \neq 0  \\
			|\bsigma| &\leq \mu_\star  & \text{ if }\bu_\tau = 0
		\end{array}
		\qquad
		\text{ for constants   }	\gamma_\star,\mu_\star \geq 0.
	\end{equation}
	For~$\gamma_\star=0$ this is the Tresca boundary condition.  
	In this example we do not consider dynamic slip effects, i.e., we have $\beta_\star=0$.
	We employ the following regularisation, see Example~\ref{ex:reg-rel},
	\begin{equation}
		\Seps(\bu_\tau) = \gamma_\star \bu_\tau + \mu_\star \frac{\bu_\tau}{\sqrt{\varepsilon^2 + |\bu_\tau|^2}}
		\qquad \text{ for }
		\varepsilon>0.
	\end{equation}
	In the computations we choose $\varepsilon = 0.0002$ and $\mu_\star = 1$. 
	We set the initial condition to zero and the forcing term to $$\bm{f}\coloneqq \partial_t \tilde{\bu} - 
	2 
	\diver( \BD\tilde{\bu}) + \diver(\tilde{\bu}\otimes \tilde{\bu}) + \nabla\tilde{p},$$ where $\tilde{\bu}(t,x)= t \, \hat{\bu}(t,x)$ and $\tilde{p}(t,x)= t\, \hat{p}(t,x)$, for $\hat u, \hat \pi$ defined in~\eqref{eq:exact_solution}.
	That means, the fluid starts from rest and for small times the exact solution is the no-slip solution $(\tilde{\bu},\tilde{p})$. 
	As time evolves, the magnitude of the forcing term induces activation and slip occurs at the slip boundary $\Gamma_s$.
	This is precisely the observed behaviour in the computed solutions; this is shown in Figures~\ref{fig:slice_st-slip} and~\ref{fig:CR_stick-slip}.
	
	\begin{figure}
		\centering
		\subfloat[C][{\centering {\normalsize $\gamma_\star=0$, $t=0.5$}.}]{{%
				\includegraphics[width=0.49\textwidth]{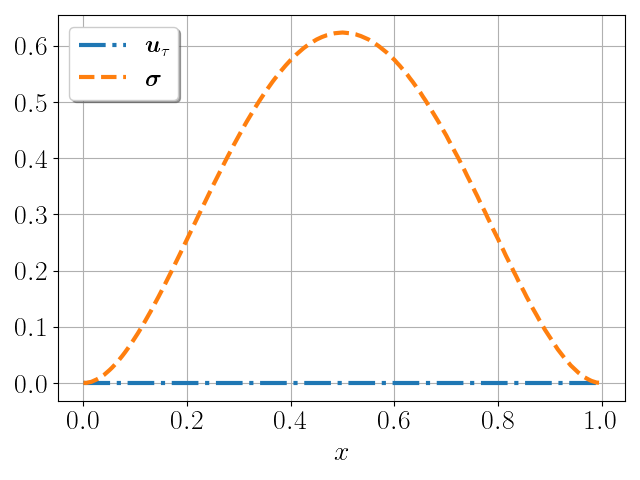}%
		}}%
		\subfloat[D][{\centering {\normalsize $\gamma_\star = 0$, $t=1.5$}.}]{{%
				\includegraphics[width=0.49\textwidth]{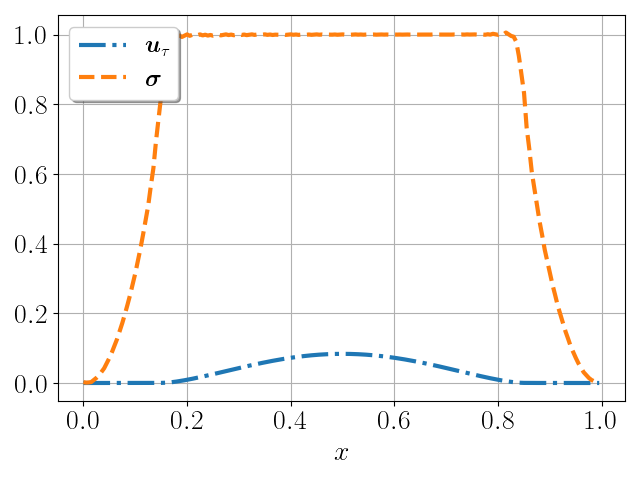}%
		}}\\
		\subfloat[C][{\centering {\normalsize $\gamma_\star=2$, $t=0.5$}.}]{{%
				\includegraphics[width=0.49\textwidth]{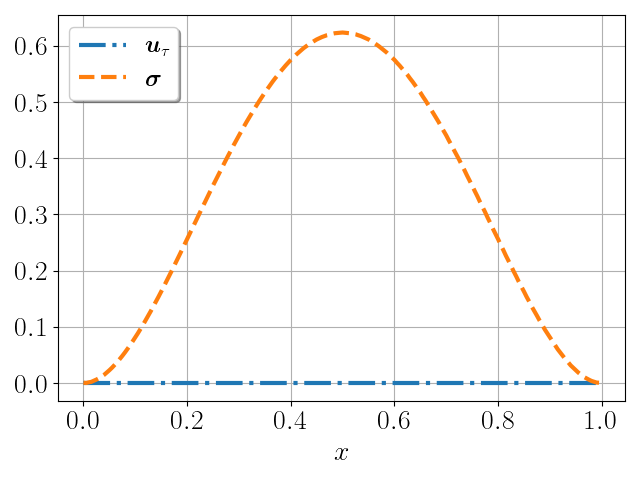}%
		}}%
		\subfloat[D][{\centering {\normalsize $\gamma_\star = 2$, $t=1.5$}.}]{{%
				\includegraphics[width=0.49\textwidth]{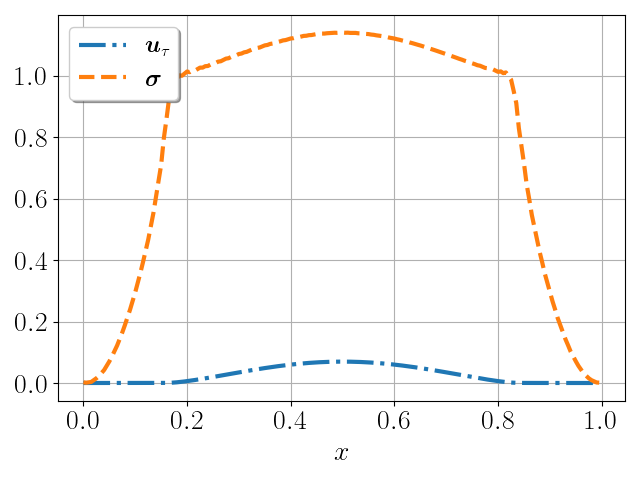}%
		}}\\
		\caption{Wall stress and tangential velocity on $\Gamma_S$ for the unsteady stick-slip condition~\eqref{eq:stick_slip}.}%
		\label{fig:slice_st-slip}
	\end{figure}
	
	\begin{figure}
		\centering
		\subfloat[C][{\centering {\normalsize $\gamma_\star=0$, $t=0.5$}.}]{{%
				\includegraphics[width=0.49\textwidth]{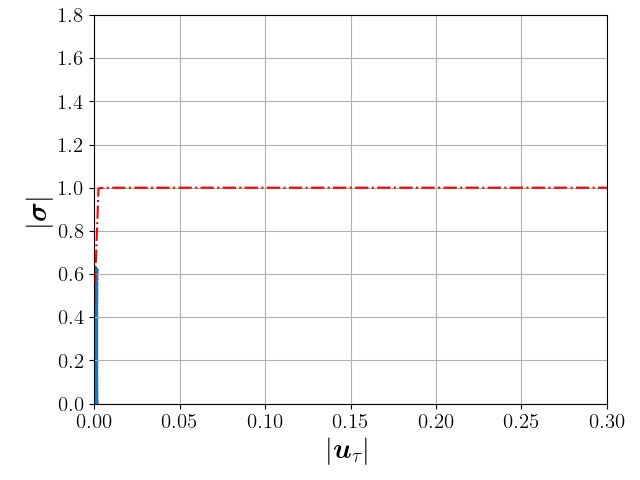}%
		}}%
		\subfloat[D][{\centering {\normalsize $\gamma_\star = 0$, $t=2.0$}.}]{{%
				\includegraphics[width=0.49\textwidth]{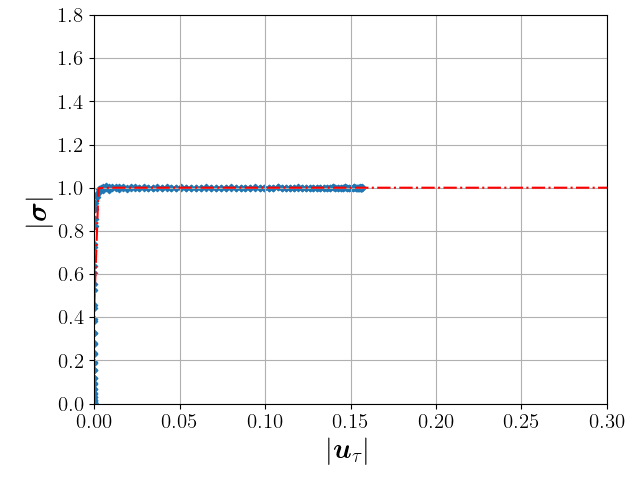}%
		}}\\
		\subfloat[C][{\centering {\normalsize $\gamma_\star=2$, $t=0.5$}.}]{{%
				\includegraphics[width=0.49\textwidth]{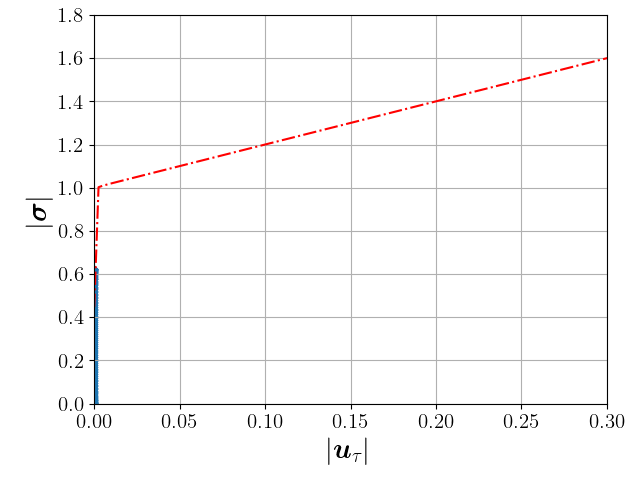}%
		}}%
		\subfloat[D][{\centering {\normalsize $\gamma_\star = 2$, $t=2.0$}.}]{{%
				\includegraphics[width=0.49\textwidth]{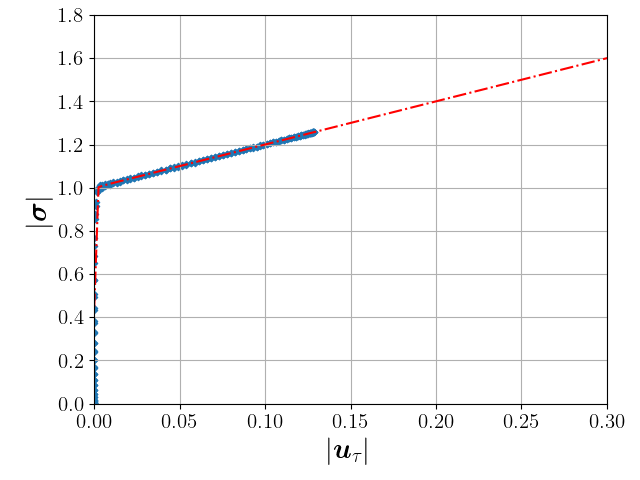}%
		}}\\
		\caption{Exact (red) and computed (blue) constitutive relation on $\Gamma_S$ for the unsteady stick-slip condition~\eqref{eq:stick_slip}.}%
		\label{fig:CR_stick-slip}
	\end{figure}
	
	\subsection{Dynamic boundary conditions}
	\label{sec:num-exp-3}
	Following~\cite{Abbatiello2021} we consider flow induced by a moving boundary $\Gamma_s$ on which dynamic slip boundary conditions are imposed; 
	the slip boundary is chosen as in the previous two sections.
	Taking into account the moving boundary, the boundary condition  reads
	\begin{equation}\label{eq:linear_dynamic_BC}
		\bsigma \coloneqq -(\BS\bn)_\tau = \beta_\star \partial_t(\bu_\tau - \hat{\bu}_b) + \gamma_\star (\bu_\tau - \hat{\bu}_b)
		\qquad
		\beta_\star,\gamma_\star\geq 0,
	\end{equation}
	where $\hat{\bu}_b= [\min\{\tfrac{t}{\theta_\star},1\}, 0]^\top$. 
	The parameter $\theta_\star$ represents the acceleration of the wall for 
	a certain time interval and is set to $\theta_\star=0.01$; 
	the presence of $\hat{\bu}_b$ models that the wall moves at constant speed, once it has reached the velocity of magnitude 1.
	The forcing term 
	$\bf$ and the initial velocity are chosen as zero, meaning that the motion is driven solely by the presence of the function $\hat{\bu}_b$. 
	
	In~\cite{Abbatiello2021} this problem is posed for $d = 3$ as a 3+1-dimensional problem and then reduced to a 1+1-dimensional problem and solved semi-analytically; 
	here, in contrast, we pose it in two spatial dimensions and solve it numerically. 
	One of the interesting phenomena that is captured by the dynamic slip for $\beta_\star >0$ is the appearance of relaxation behaviour as the wall starts to move; as observed in~\cite{Abbatiello2021} (cf.~\cite{Hatzikiriakos2012}), this is not possible for the traditional slip model with $\beta_\star=0$. 
	
	A plot of the actual slip velocity at the midpoint $(x,y)=(0.5,1.0)$ for $\gamma_\star=1$ and several values of $\beta_\star$ is depicted in Figure~\ref{fig:dynamic_slip}. 
	The results confirm the theoretical analysis in~\cite{Abbatiello2021}, which establishes that the Navier slip response ($\beta_\star = 0$) is monotonic in time: 
	the `relaxation time' is effectively zero and the effect on the fluid is instantaneous. 
	In contrast, the dynamic term with $\beta_\star>0$ allows for a nonmonotonic response in time. 
	
	\begin{figure}
		\centering
		\includegraphics[width=0.5\textwidth]{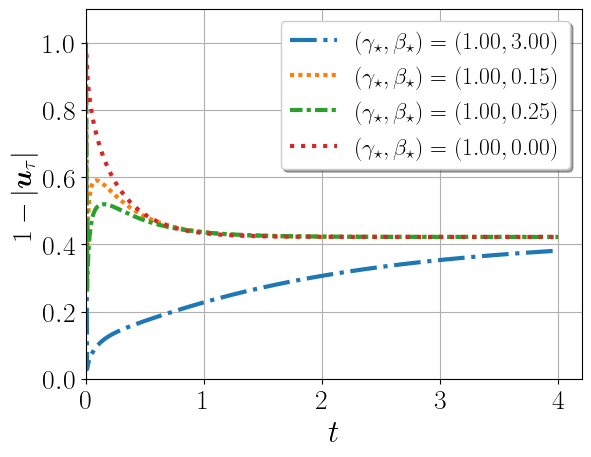}
		\caption{Time-dependence of the slip velocity for the dynamic boundary condition~\eqref{eq:linear_dynamic_BC} at $(x,y)=(0.5,1)$.}%
		\label{fig:dynamic_slip}
	\end{figure}
	
	\subsection*{Acknowledgements} 
	We thank Ridgway Scott for helpful discussions on the Scott--Zhang interpolation operator, see Lemma~\ref{lem:SZ}.
	
	Alexei Gazca-Orozco gratefully acknowledges support from the Charles University Research program no.~PRIMUS/25/SCI/025 and UNCE/24/SCI/005. 
	Franz Gmeineder was supported by the Hector foundation (Hector II, Project Number FP 626/21). 
	The work by Tabea Tscherpel was supported by the German Research Foundation (DFG) via grant TRR 154, subproject C09, project number 239904186. 
	Part of this work was carried out while the authors were hosted by the Institute of Science and Technology Austria, whose support is gratefully acknowledged.
	
	\appendix 
	
	\section{Properties of the regularised Tresca relation}\label{app:ex-reg-Tresca}
		
		To verify~\ref{itm:gbd-eps-approx} for the regularised Tresca relation in Example~\ref{ex:reg-rel}, let $\bv\in L^1(M)^d$ and $\bs\in L^\infty(M)^d$ be such that $\gbd(\bs,\bv)=\bm{0}$ a.e.~in $M$ with $\gbd$ as in~\eqref{eq:gbd-Tresca}.
			We introduce the functions 
			\begin{align*}
				\bphi \coloneqq \tfrac{1}{\sqrt{2}} (\bv + \bs) \quad \text{ and } 	\quad  \bpsi \coloneqq \tfrac{1}{\sqrt{2}} (-\bv + \bs).
			\end{align*}
			This means, that with the so-called \emph{Cayley transform}, $\Phi \colon \setR^d \to \setR^d$ we have $( \bphi,  \bpsi) = \Phi(\bv,\bs)$, cf.~\cite[p.~265]{AlbertiAmbrosio1999}. 
			The reverse transformation is given by $(\bv,\bs)  = \Phi^{-1}(\bphi,\bpsi) =  \tfrac{1}{\sqrt{2}} (\bphi - \bpsi,\bphi + \bpsi)$. 
			Since $\Seps$ is a maximal monotone function, by~\cite[Prop.~1.1]{AlbertiAmbrosio1999}  the mapping $ F_{\epsilon} \colon \setR^d \to \setR^d, \overline \bphi \mapsto \overline  \bpsi_\varepsilon$, defined by 
			\begin{align*}
				\Seps\left(\tfrac{1}{\sqrt{2}}( \overline \bphi- \overline \bpsi_\varepsilon )\right) = \tfrac{1}{\sqrt{2}}(\overline\bphi +  \overline \bpsi_\varepsilon),
			\end{align*}
			is a $1$-Lipschitz function. 
			Analogously, the maximal monotone zero set of $\gbd$ can be represented by a $1$-Lipschitz function $F$ with argument $\overline \bphi$. 
			With $\Seps$ defined as above one can show that  $F_\varepsilon(\overline \bphi)\to F(\overline \bphi)$ for any $\overline \bphi \in \setR^d$, and uniformly on every compact set.
			Note also that $F(\b0) = F_{\epsilon}(\b0) = \b0$, and that $\indicator_{S}\bphi \in L^\infty(M)$ for $S \coloneqq \{\abs{\bv}\leq \mu_\star \}$. 
			This means that the image of $\indicator_S \bphi$ is contained in a compact set, and defining  $\bpsi_\varepsilon \coloneqq F_{\varepsilon}( \indicator_{S}\bphi )$ we thus have 
			\begin{align}\label{conv:psi}
				\indicator_{S} ( \bpsi_\epsilon - \bpsi  )  =  F_{\varepsilon}( \indicator_{S}\bphi ) - F(\indicator_{S}\bphi ) \to 0
			\end{align}
			converges strongly in $L^\infty(M)^d$ as $\varepsilon \to 0$. 
			Then, we choose 
			\begin{align}\label{conv_veps}
				\bv_{\varepsilon} &\coloneqq  \tfrac{1}{\sqrt{2}} (\bphi - \bpsi_{\varepsilon})    \indicator_{S} + \bv \indicator_{S^c} \to  
				\tfrac{1}{\sqrt{2}} (\bphi - \bpsi)    \indicator_{S} + \bv \indicator_{S^c} 
				= \bv, 
			\end{align}
			which converges in particular strongly in $L^1(M)^d$, as $\varepsilon \to 0$. 
		
			Next, let us show that $\Seps(\bv) \indicator_{S^c} \to \bs \indicator_{S^c} = \mu_\star \frac{\bv}{\abs{\bv}} \indicator_{S^c}$, as $\varepsilon \to 0$. 
			Noting that on $S^c \coloneqq \{ \abs{\bv} > \mu_\star\}$, it follows that 
			\begin{align}\label{est:mu}
				\frac{1}{\sqrt{\abs{\bv}^2 + \varepsilon^2}} \leq \frac{1}{\abs{\bv}} < \frac{1}{\mu_\star}. 
			\end{align}
			Thus, for any $\bw$ with $\abs{\bw} >\mu_\star$ we have with~\eqref{est:mu} and with Taylor expansion that 
			\begin{align*}
				\abs{\Seps(\bw) - \mu_{\star} \frac{\bw}{\abs{\bw}}} 
				&= 
				\mu_\star \abs{\bw} \abs{\frac{1}{\sqrt{\abs{\bw}^2 + \epsilon^2 }} - \frac{1}{\abs{\bw}} } \\
				&= \mu_\star \frac{1}{\sqrt{\abs{\bw}^2 + \epsilon^2 }}   \left( \sqrt{\abs{\bw}^2 + \epsilon^2 } - \sqrt{\abs{\bw}^2} \right)  \leq \frac{\varepsilon^2}{\mu_{\star}}. 
			\end{align*}      
			This means that  we have the strong convergence in $L^\infty(M)^d$ of
			\begin{align*}
				\Seps(\bv) \indicator_{S^c} \to \bs \indicator_{S^c} = \mu_\star \frac{\bv}{\abs{\bv}} \indicator_{S^c} \quad \text{ as } \varepsilon \to 0. 
			\end{align*}
			In combination with $\indicator_S \bphi \in L^\infty(M)^d$ and~\eqref{conv:psi} we obtain
			\begin{align}\label{conv:seps}
				\Seps(\bv_{\epsilon}) &= 
				\tfrac{1}{\sqrt{2}} (\bphi + \bpsi_{\varepsilon}) \indicator_{S} + \Seps(\bv) \indicator_{S^c}
				\to \tfrac{1}{\sqrt{2}} (\bphi + \bpsi) \indicator_{S} 
				+ \mu_{\star} \frac{\bv}{\abs{\bv}}\indicator_{S^c}  = \bs
			\end{align}	
			converges strongly in $L^\infty(M)^d$ as $\varepsilon \to 0$. 
			Combining~\eqref{conv_veps} and~\eqref{conv:seps} proves the claim. 
	
	\printbibliography
	
	\vspace{-1em}
	
\end{document}